\documentclass[a4paper,reqno]{amsart}

\usepackage[T1]{fontenc}
\usepackage[utf8]{inputenc}
\usepackage[unicode]{hyperref}

\usepackage{amsmath,amssymb,amsthm,mathtools}
\usepackage{cite}

\usepackage{tikz}\usetikzlibrary{cd,calc,positioning,matrix,arrows,decorations.markings}

\usepackage{enumitem}
\setlist[enumerate]{label = \textup{\textup{(\alph*)}},ref = \textup{(\alph*}),itemsep=1ex}

\theoremstyle{plain}
\newtheorem{thm}{Theorem}[section]
\newtheorem{lem}[thm]{Lemma}
\newtheorem{prop}[thm]{Proposition}
\newtheorem{cor}[thm]{Corollary}
\theoremstyle{definition}
\newtheorem{dfn}[thm]{Definition}
\newtheorem{ex}[thm]{Example}
\newtheorem{rem}[thm]{Remark}
\newtheorem{ass}[thm]{Assumption}

\numberwithin{equation}{section}
\numberwithin{figure}{section}

\def\Z{\mathbb{Z}}\def\Q{\mathbb{Q}}\def\C{\mathbb{C}}\def\K{\mathbb{K}}

\DeclareMathOperator\id{id}

\DeclareMathOperator\Ker{Ker}\DeclareMathOperator\Coker{Coker}

\def\U{\mathrm{U}}\def\SU{\mathrm{SU}}

\def\cA{\mathcal{A}}\def\cH{\mathcal{H}}

\def\al{\alpha}
\def\be{\beta}
\def\ga{\gamma}
\def\de{\delta}
\def\ep{\epsilon}
\def\ze{\zeta}
\def\th{\theta}

\def\ka{\kappa}
\def\la{\lambda}
\def\si{\sigma}

\def\om{\omega}

\def\De{\Delta}

\def\La{\Lambda}
\def\Om{\Omega}
\def\Ga{\Gamma}
\def\Si{\Sigma}
\def\Th{\Theta}

\def\e#1\e{\begin{equation}#1\end{equation}}
\def\ea#1\ea{\begin{align}#1\end{align}}

\def\op{\oplus}
\def\ot{\otimes}
\def\iy{\infty}
\def\longra{\longrightarrow}
\def\t{\times}
\def\an#1{\langle #1 \rangle}

\def\6{\partial}

\DeclareMathOperator\ch{ch}
\DeclareMathOperator\ind{ind}
\DeclareMathOperator\PE{PE}

\DeclareMathOperator*\colim{colim}
\DeclareMathOperator\ev{ev}

\DeclareMathOperator\Map{Map}
\def\cla{\mathrm{cla}}

\def\top{\mathrm{top}}

\usepackage[mathscr]{euscript}

\def\G{{B\U(1)}}
\def\F{{B\U}}
\def\y{c}
\def\z{\operatorname{ch}}

\usepackage{stmaryrd}

\def\ul{\underline}
\def\ol{\overline}

\DeclareMathOperator\rk{rk}

\tikzset{
  symbol/.style={
    draw=none,
    every to/.append style={
      edge node={node [sloped, allow upside down, auto=false]{$#1$}}}
  }
}

\begin{document}

\title[Homological Lie brackets and pushforward operations]{
Homological Lie brackets on moduli spaces and\\
pushforward operations in twisted K-theory
}
\author{Markus Upmeier}
\date{\today}

\begin{abstract}
We develop a general theory of pushforward operations for principal $G$-bundles equipped with a certain type of orientation.

In the case $G=\G$ and orientations in twisted K-theory we construct two pushforward operations, the projective Euler operation, whose existence was conjectured by Joyce, and the projective rank operation. We classify all stable pushforward operations in this context and show that they are all generated by the projective Euler and rank operation.

As an application, we construct a graded Lie algebra structure on the homology of a commutative H-space with a compatible $\G$-action and orientation. These play an important role in the context of wall-crossing formulas in enumerative geometry.
\end{abstract}


\maketitle


\section{Introduction}
\label{s1}

Many traditional coarse moduli spaces $P$ are obtained from a moduli stack that has a scalar action by $\C^*$ on the morphism sets. Examples are moduli spaces of coherent sheaves, connections, or quiver representations. Topologically, the action leads to a principal $\G$-bundle $P\to {B=P/\G},$ where $\G$ is a topological group model for the classifying space of complex line bundles. Enumerative geometry studies the intersection theory of virtual fundamental classes in homology. These usually depend on auxiliary parameters and the resulting virtual fundamental classes should then be related by wall-crossing formulas. Joyce~\cite{Joy} describes a comprehensive, partly conjectural, new theory of wall-crossing formulas in which the relationship between the homology $H_*(P)$ of the coarse moduli space and the homology $H_*(B)$ in the sense of stacks plays a key role. The main purpose of this paper is to clarify this relationship.

A principal $\G$-bundle $P\to B$ determines a cohomology class in $H^3(B;\Z),$ whose image in $H^3(B;\Q)$ we denote by $\eta_P$. We say that $P$ is \emph{rationally trivial} if $\eta_P=0.$ In this case there is a K\"unneth decomposition $H^*(P;\Q)\cong \Q[c_1]\ot H^*(B;\Q),$ where $c_1$ is the generator of the cohomology of $\G\simeq\mathbb{CP}^\iy.$ In particular, there are many projections from $H^*(P)$ to $H^*(B).$ In the general case we will show that one can still construct a certain combination of these projections using additional data, called an \emph{orientation}. If $P$ is trivial, an orientation amounts to a class $\vartheta\in K(B)$ in complex K-theory. Let $r$ be the rank of $\vartheta.$ In this case the projective Euler operation is given by
\begin{align}
\label{BabyPE}
 H^*(P;\Q)&\overset{\smash{\pi_!^\vartheta}}{\longra} H^{*+2r+2}(B;\Q), &c_1^i\t\al&\longmapsto \al\cup c_{i+r+1}(\vartheta),
\end{align}
where $c_j(\vartheta)$ denotes the $j$\textsuperscript{th} Chern class of $\vartheta.$
We will prove that this globalizes to a construction for general $P,$ provided $\vartheta$ is replaced by a class $\th\in K_P(B)$ in \emph{twisted} complex K-theory, as we explain after a brief digression on twisted K-theory.

Let $P\to B$ be a principal $\G$-bundle. Let $\{U_i\}_{i\in I}$ be an open cover of $B$ and write $U_{ij}=U_i\cap U_j,$ etc. Given sections $\si_i$ of $P|_{U_i},$ the transition functions $\ga^P_{ij}\colon U_{ij}\to\G$ satisfy $\ga^P_{ij}\cdot\si_j|_{U_{ij}}=\si_i|_{U_{ij}}$ and classify complex line bundles $L_{ij}\to U_{ij}.$ The cocycle identity for $\ga^P_{ij}$ yields isomorphisms $\ga^P_{ijk}\colon L_{ij}|_{U_{ijk}}\ot L_{jk}|_{U_{ijk}}\to L_{ik}|_{U_{ijk}}.$ Each collection $\th=(\{\vartheta_i\}_{i\in I},\{\ga^\th_{ij}\}_{i,j\in I})$ of complex vector bundles $\vartheta_i\to U_i$ and isomorphisms $\ga^\th_{ij}\colon L_{ij}\ot\vartheta_j|_{U_{ij}}\to \vartheta_i|_{U_{ij}}$ satisfying the twisted cocycle identity $\ga_{ij}^\th\circ(\id_{L_{ij}}\ot\ga_{jk}^\th)=\ga_{ik}^\th\circ(\ga_{ijk}^P\ot \id_{\vartheta_k})$ determines a class in twisted K-theory $K_P(B).$ If $B$ is a compact Hausdorff space, we could define $K_P(B)$ in this way as a group completion; the official definition is given in \S\ref{s2}. Using the universal complex line bundle $V(1)\to\G,$ we can construct vector bundles $V(1)\boxtimes\vartheta_i$ over $\G\t U_i$ for which the isomorphisms $\{\ga^\th_{ij}\}$ can be viewed as descent data, yielding a vector bundle $\tilde\th$ over $P.$ Therefore, every twisted K-theory class $\th\in K_P(B)$ has an \emph{underlying complex K-theory class} $\tilde\th\in K(P).$ If $P$ is trivial, then $\th$ amounts to a class $\vartheta\in K(B)$ in ordinary K-theory and $\tilde\th=V(1)\boxtimes\vartheta.$

According to Atiyah--Segal~\cite[Prop.~8.8]{AtSe06}, characteristic classes for twisted K-theory classes of rank $r$ are in bijection with cohomology classes in $H^*(\F\t\{r\})$ that are invariant under the $\G$-action. The Chern class $c_{r+1}$ has this property and thus extends to a characteristic class $c_{r+1}(\th)$ for twisted K-theory classes $\th$ of rank $r$. To see the invariance property of $c_{r+1}$, observe that for any (virtual) complex vector bundle $\vartheta$ of rank $r$ and complex line bundle $L$ we have
\e
\label{cj-Tensor}
 c_j(L\ot\vartheta)=\sum_{k+\ell=j} \binom{r-\ell}{k} c_1(L)^k\cup c_\ell(\vartheta).
\e
We make use of binomial coefficients for integers, which we review in Appendix~\ref{sAppdx}. Putting $j=r+1$ in \eqref{cj-Tensor}, all binomial coefficients with $k\neq 0$ vanish, so $c_{r+1}(L\ot\vartheta)=c_{r+1}(\vartheta)$, thus verifying the invariance property of $c_{r+1}.$

The projective Euler class is similar to a transfer map in the sense of \cite[\S15]{MS}: the pull-push $\pi_!^\th\circ\pi^*$ is the multiplication by an `Euler class', which here is the characteristic class $\pi_!^\th(1_{H^*(P)})=c_{r+1}(\th)$ of Atiyah--Segal.

A key observation is that \eqref{cj-Tensor} implies that a change of trivialization $\ga\colon B\to\G$ in \eqref{BabyPE} leads to an automorphism of $H^*(P)$ under which $\pi_!^\vartheta$ becomes $\pi_!^{L\ot\vartheta},$ where $L$ is the line bundle corresponding to $\ga.$ This explains why orientations in twisted K-theory are needed.

Our first theorem summarizes the properties of the projective Euler class. To state it, we need more notation. A pullback diagram of principal\/ $\G$-bundles
\begin{equation}
\label{pullback-diagram}
\begin{tikzcd}
 P_1\arrow[r,"\Phi"]\dar{\pi_1} & P_2\dar{\pi_2}\\
 B_1\arrow[r,"\phi"] & B_2
\end{tikzcd}
\end{equation}
determines a pullback morphism $\Phi^*\colon K_{P_2}(B_2)\to K_{P_1}(B_1)$ in twisted K-theory.

The classifying space $\G$ will always be taken to be a topological abelian group $G.$ We can then define the \emph{dual principal $G$-bundle} $\breve P\to B$ by precomposing the action by the inversion in $G.$ Taking duals of vector bundles defines a map $K_P(B)\to K_{\breve P}(B),$ $\th\mapsto\breve\th,$ of twisted K-theory groups. The \emph{tensor product} $P_3=P_1\ot_G P_2$ of principal $G$-bundles $\pi_1\colon P_1\to B$ and $\pi_2\colon P_2\to B$ is the quotient of the fiber product $P_1\t_B P_2$ by the $G$-action $\ell_{P_1\t_B P_2}(g,(p_1,p_2))=(g^{-1}p_1,gp_2).$
\label{ConventionFPaction}
Write $p_1\ot_G p_2$ for the orbit of $(p_1,p_2)$ under this action. The action on the orbit space $P_3=P_1\ot_G P_2$ is defined by $\ell_{P_3}(g,p_1\ot_G p_2)=(gp_1)\ot_G p_2=p_1\ot_G(gp_2).$

Under the obvious homeomorphisms\/ $\pi_2^*(P_1)\cong P_1\t_B P_2\cong\pi_1^*(P_2)$ the map $\ka_3(p_1,p_2)=p_1\ot_G p_2$ may be regarded in two ways as a morphism of principal bundles, $\pi_1^*(P_2)\to P_3$ or $\pi_2^*(P_1)\to P_3,$ or as a principal $G$-bundle $P_1\t_B P_2\to P_3.$
The projection $\ka_1(p_1,p_2)=p_1$ can similarly be viewed as a principal $G$-bundle $\pi_1^*(P_2)\to P_1$ or as a morphism $\pi_2^*(P_1)\to P_1$ or $P_1\t_B P_2\to \breve P_1$; similarly for $\ka_2.$

\begin{thm}
\label{s1thm1}
For each principal\/ $\G$-bundle\/ ${\pi\colon P\to B}$ and orientation\/ $\th\in K_P(B)$ in twisted K-theory there is a \textbf{projective Euler operation}
\[
 \pi^\th_!\colon H^*(P;\Q)\longra H^{*+2r+2}(B;\Q),
\]
where\/ $r$ is the rank of $\th,$ uniquely determined by the following properties.
\begin{enumerate}
\item
Naturality:~For every pullback diagram \eqref{pullback-diagram} and for the pullback orientation\/ $\th_1=\Phi^*(\th_2)$ on\/ $P_1$ of rank\/ $r,$ there is a commutative square
\begin{equation}
\label{s1eqn3}
\begin{tikzcd}
H^*(P_1;\Q)\arrow[d,"(\pi_1)^{\th_1}_!"] & H^*(P_2;\Q)\arrow[l,"\Phi^*"']\arrow[d,"(\pi_2)^{\th_2}_!"]\\
H^{*+2r+2}(B_1;\Q) & H^{*+2r+2}(B_2;\Q)\mathrlap{.}\arrow[l,"\phi^*"']
\end{tikzcd}	
\end{equation}
\item
Stability:~For the pullback orientation\/ $\th\t S^1=\pi_P^*(\th)$ on\/ $P\t S^1,$ where $\pi_P\colon P\t S^1\to P$ is the projection, there is a commutative square
\begin{equation}
\label{PEstable}
\begin{tikzcd}[column sep=huge]
H^*(P;\Q)\ar[d,"\pi^\th_!"]\rar["{\t[S^1]}"] & H^{*+1}(P\t S^1;\Q)\ar[d,"{(\pi\t\id_{S^1})^{\th\t S^1}_!}"]\\
H^{*+2r+2}(B;\Q)\rar["{\t[S^1]}"] & H^{*+2r+3}(B\t S^1;\Q).
\end{tikzcd}
\end{equation}
\item
Normalization:~If\/ $P=\G\t B$ is trivial,\/ $\pi_!^\th(c_1^i\t 1_{H^*(B)})=c_{i+r+1}(\vartheta),$ where\/ $\vartheta\in K(B)$ is the ordinary K-theory class corresponding to\/ $\th.$
\item
Euler class:~$\pi^\th_!(1_{H^*(P)})=c_{r+1}(\th)$ is the characteristic class of Atiyah--Segal.
\item\label{thmB-d}
Base-linearity:~For all\/ $\al\in H^*(P;\Q)$ and\/ $\be\in H^*(B;\Q),$ we have
\e
\label{s1eqn2}
\pi^\th_!(\al\mathbin\cup\pi^*(\be))=\pi^\th_!(\al)\mathbin\cup\be.
\e
In particular, the pull-push formula\/ $\bigl(\pi_!^\th\circ\pi^*\bigr)(\be)=c_{r+1}(\th)\cup\be$ holds.
\item\label{thmB-g}
Push-pull formula:~For the underlying complex K-theory class\/ $\tilde\th\in K(P)$ of the orientation\/ $\th$ and using the operation\/ \textup{`$t\diamond$'} from Definition~\textup{\ref{DiamondAction}}, we have\/ $\bigl(\pi^*\circ \pi^\th_!\bigr)(\al)=\sum\nolimits_{i\geqslant0}\bigl(\frac{t^i}{i!}\diamond\al\bigr)\cup c_{i+r+1}(\tilde{\th})$ for all\/ $\al\in H^*(P;\Q).$
\item\label{thmB-e}
Duality:~For the dual principal\/ $\G$-bundle\/ $\breve{\pi}\colon\breve{P}\to B$ and the dual orientation\/ $\breve\th\in K_{\breve{P}}(B),$ we have\/ $\breve{\pi}^{\breve\th}_! = (-1)^{r+1}\pi^\th_!.$
\item\label{thmB-f}
Composition:~Let\/ $\pi_1\colon P_1\to B$ and\/ $\pi_2\colon P_2\to B$ be principal\/ $\G$-bundles and let\/ $P_3=P_1\ot_\G P_2.$ Let\/ $\th_k\in K_{P_k}(B)$ for\/ $1\leqslant k\leqslant3.$ Then
\e
\label{eqn:composition}
 \hspace{\leftmargin}
(\pi_2)^{\th_2}_!\circ (\ka_2)^{\ka_1^*\th_1+\ka_3^*\th_3}_! - (\pi_1)^{\th_1}_!\circ(\ka_1)^{\ka_2^*\th_2+\ka_3^*\th_3}_!
=(-1)^{r_1}(\pi_3)^{\th_3}_!\circ(\ka_3)^{\ka_1^*\breve\th_1+\ka_2^*\th_2}_!
\e
for the pushforwards around the three routes in the following commutative diagram of bundle projections of principal\/ $\G$-bundles.
\begin{equation}
\label{composition}
 \begin{tikzcd}[row sep=small]
  \pi_2^*(P_1)\cong P_1\t_B P_2\cong\pi_1^*(P_2)
  \ar[rr,"\ka_1"]
  \ar[dd,shift right=1.6cm,"\ka_2"]
  \ar[rd,start anchor={[xshift=-0.65cm]},end anchor={[xshift=-1cm,yshift=0.00cm]},shorten >=2ex,"\ka_3"] & &
  P_1\ar[dd,"\pi_1"]\\
  & \hspace{-2cm}P_3\ar[rd,start anchor={[xshift=-0.75cm,yshift=-0.1cm]},shorten <=0.5ex,"\pi_3"] &\\
  \hspace{-3.25cm}P_2
  \ar[rr,start anchor={[xshift=-1.25cm]},"\pi_2"] &[shift left=1cm] & B
 \end{tikzcd}
\end{equation}
Here,\/ $r_1$ is the rank of\/ $\th_1,$ we use the obvious homeomorphisms\/ $\pi_2^*(P_1)\cong P_1\t_B P_2\cong\pi_1^*(P_2)$ to identify all the domain cohomology groups in \eqref{eqn:composition}, and the orientation\/ $\ka_1^*\th_1+\ka_3^*\th_3$ on\/ $\pi_2^*(P_1)$ is constructed by pullback along the morphisms\/ $\ka_1\colon \pi_2^*(P_1)\to P_1$ and\/ $\ka_3\colon \pi_2^*(P_1)\to P_3$; similarly for the orientation\/ $\ka_2^*\th_2+\ka_3^*\th_3$ on\/ $\pi_1^*(P_2)$ and for\/ $\ka_1^*\breve\th_1+\ka_2^*\th_2$ on\/ $P_1\t_B P_2.$
\item\label{thmB-h}
For all\/ $\al\in H^*(P),$ $\pi^\th_!(\al)\cup\eta_P=0$ for the characteristic class\/ $\eta_P$ of\/ $P.$
\end{enumerate}
\end{thm}

The proof of Theorem~\ref{s1thm1} occupies \S\ref{sGenSection} and \S\ref{s43}. Dually, in Theorem~\ref{thmHomologEuler} we show that there is also a projective Euler operation $\pi_\th^!\colon H_*(B;\Q)\to H_{*-2r-2}(P;\Q)$ in homology, proving a conjecture by Joyce \cite[Question 2.41]{Joy}.

\begin{thm}
\label{Prop_S_Theta}
For every principal\/ $\G$-bundle\/ ${\pi\colon P\to B}$ and orientation\/ $\th\in K_P(B)$ in twisted K-theory, there is a \textbf{projective rank operation}
\[
 s_\th^*\colon H^*(P;\Q)\longra H^*(B;\Q),
\]
uniquely determined by the following properties.
\begin{enumerate}
\item
Naturality:~For every pullback diagram \eqref{pullback-diagram} and the pullback orientation\/ $\th_1=\Phi^*(\th_2),$ we have\/ $\phi^*\circ s_{\th_2}^*=s_{\th_1}^*\circ\Phi^*.$
\item
Stability:~$(s\t\id_{S^1})_{\th\t S^1}^*(\al\t[S^1])=s_\th^*(\al)\t[S^1]$ for all\/ $\al\in H^*(P;\Q)$
\item
Normalization:~If\/ $P$ is trivial,\/ $s^*_\th(c_1^i\t 1_{H^*(B)})=(-1)^i i!\z_i(\vartheta)$ for all\/ $i\geqslant0,$ where\/ $\vartheta\in K(B)$ is the ordinary K-theory class corresponding to\/ $\th.$
\item
Euler class:~$s_\th^*(1_{H^*(P)})=r\,1_{H^*(B)},$ where\/ $r$ is the rank of\/ $\th$
\item
Base-linearity:~For all\/ $\al\in H^*(P;\Q)$ and\/ $\be\in H^*(B;\Q),$ we have\/ $s_\th^*(\al\mathbin\cup\pi^*(\be))=s_\th^*(\al)\mathbin\cup \be.$ In particular,\/ $\bigl(s_\th^*\circ\pi^*\bigr)(\be)=r\be.$
\item
Push-pull formula:~For the underlying complex K-theory class\/ $\tilde\th\in K(P)$ of the orientation, we have\/ $\bigl(\pi^*\circ s_\th^*\bigr)(\al)=\sum\nolimits_{i\geqslant0}(-1)^i(t^i\diamond\al)\cup \z_{i}(\tilde{\th}).$
\end{enumerate}
\end{thm}

We prove Theorem~\ref{Prop_S_Theta} in \S\ref{sGenSection}.

\begin{rem}
Following Atiyah--Segal~\cite{AtSe04}, the projective unitary group $P\U(\cH)=\U(\cH)/\U(1)$ of a separable complex Hilbert space is a non-abelian topological group model for $\G.$ Let $L\to P\U(\cH)$ be the universal complex line bundle. If $P\to B$ is a principal $P\U(\cH)$-bundle, a twisted complex K-theory class is represented by an equivariant family of Fredholm operators $\th=\{\th_p\mid p\in P\}$ on $\cH$ satisfying
\ea
\label{th-equivariant}
\th_{pg}&=L|_g\ot \th_p,		&&\forall p\in P, g\in P\U(\cH).
\ea
In this model, the rank of $\th$ becomes the index. The one-dimensional complex vector spaces $\La(\th_p)=\La^\top(\Ker\th_p)\ot\La^\top(\Coker\th_p)^*$ are the fibers of the \emph{determinant line bundle} $\La(\th)\to P.$ From \eqref{th-equivariant} we find $\La(\th_{pg})\cong\La(L|_g\ot \th_p)\cong(L|_g)^{\ot\ind\th_p}\ot\La(\th_p).$
If ${\ind\th_p=0},$ it follows that $\La(\th)$ descends to a complex line bundle on $B$ whose first Chern class is the Atiyah--Segal characteristic class $c_1(\th).$ 
If ${\ind\th_p=1},$ the determinant line bundle is $P\U(\cH)$-equivariant and its classifying map ${f_{\La(\th)}\colon P\to P\U(\cH)}$ determines a global trivialization $(\pi,f_{\La(\th)})\colon P\to P\U(\cH)\t B$ of $P.$ The projective rank operation $s_\th^*$ is the pullback along the corresponding global section of $P,$ thus explaining the notation.
\end{rem}

Our next result states that $\pi_!^\th$ and $s_\th^*$ are the basic generators of all operations of this type. To explain this, we need more terminology. In \S\ref{s2} we will study in a general context a new algebro-topological tool, roughly the bundle version of traditional cohomology operations. Let $G$ be a topological group. For every principal $G$-bundle $\pi\colon P\to B$ with a certain type of orientation $\th,$ a \emph{pushforward operation} of type $(m,n)$ determines a map
\[
	\Xi_{P,\th}^{m|n}\colon H^m(P)\longra H^n(B),
\]
natural in $P$ and $\th.$ In general, these are hard to classify. This simplifies once we introduce \emph{stable} pushforward operations whose maps are defined for all $m$ with fixed degree $k=n-m$ and which satisfy the analogue of \eqref{PEstable}.

For $G=\G$ and $r\in\Z$ let $\Pi_{\smash{\F\t\{r\}}}^k$ be the group of stable pushforward operations $H^*(P;\Q)\to H^{*+k}(B;\Q)$ of degree $k$ for principal $\G$-bundles with orientations in twisted complex K-theory of rank $r.$ Given a pushforward operation, we can construct new pushforward operations by composing with endomorphisms of $H^*(P;\Q)$ and $H^*(B;\Q).$ Let $\ell_P\colon G\t P\to P$ denote the principal action and let $t\in H_2(\G)$ be the class dual to $c_1.$ As explained in Definition~\ref{DiamondAction} below, the ring ${\Q\llbracket t\rrbracket}$ acts on $H^*(P;\Q)$ by $t\diamond\al=t\backslash\ell_P^*(\al)$. Moreover, ${\Q[c_1,c_2,\ldots]}$ acts by multiplication with the Chern classes of the underlying complex K-theory class ${\tilde\th\in K(P)}.$ This determines on $\Pi_{\smash{\F\t\{r\}}}^*$ a graded module structure over the semi-direct product algebra ${S=\Q\llbracket t\rrbracket\rtimes\Q[c_1,c_2,\ldots]},$ where $t$ has degree $-2$ and $c_j$ has degree $2j.$

As we will see, in odd degrees there is a third basic operation, $\eta\in\Pi_{\smash{\F\t\{0\}}}^3,$ related to the characteristic class $\eta_P\in H^3(B;\Q)$ of the bundle.

\begin{thm}
\label{thmA}
\hangindent\leftmargini
\textup{(a)}\hskip\labelsep
If\/ $r=0,$ the even part\/ $\Pi_{\smash{\F\t\{0\}}}^\mathrm{ev}$ is generated as an\/ $S$-module by the projective Euler operation of degree\/ $2.$ The elements of the odd part\/ ${\Pi_{\smash{\F\t\{0\}}}^\mathrm{odd}\cong\Q}$ are only the rational multiples of\/ $\eta\in\Pi_{\smash{\F\t\{0\}}}^3.$ 
 \begin{enumerate}
\setcounter{enumi}{1}
 \item
 If\/ $r\neq0,$ then\/ $\eta=0$ \textup(rationally trivial case\textup) and the\/ $S$-module\/ $\Pi_{\smash{\F\t\{r\}}}^*$ is generated by the projective rank operation\/ $s^*_\th$ of degree\/ $0.$ In particular, we have\/ $\Pi_{\smash{\F\t\{r\}}}^\mathrm{odd}=\{0\}$ in this case.
 \end{enumerate}
\end{thm}

Theorem~\ref{thmA} is proven in \S\ref{s42}. In fact, we will actually use this classification result to construct the projective Euler and the projective rank operation. It would be interesting to determine all of the relations in the $S$-module $\Pi^*_{\F\t\{r\}}.$
\smallskip

Finally, we give the following application of the projective Euler class in \S\ref{s5}. Let $M$ be a commutative H-space with a free action of $\G$ that is compatible with the H-space operation $\Phi\colon M\t M\to M.$ It is well-known that the rational homology $H_*(M)$ is then a graded associative algebra. We will show that, given additional orientations, the homology $H_*(M/\G)$ can be made into a graded Lie algebra. Let $(\pi_0(M),+)$ be the commutative monoid of path-components and write $M_\alpha\subset M$ for the path-component representing $\alpha\in\pi_0(M).$ There are maps
\[
 (\Phi_{\al,\be}/\G)_*\colon H_*((M_\al\t M_\be)/\G;\Q) \longra H_*(M_{\al+\be}/\G;\Q).
\]
To relate the domain with $H_*(M_\al/\G\t M_\be/\G),$ note that there is a principal $\G$-bundle $\pi_{\al,\be}\colon(M_\al\t M_\be)/\G\to M_\al/\G\t M_\be/\G.$ If we assume orientations $\th_{\al,\be}$ on $(M_\al\t M_\be)/\G$ of ranks $\chi(\al,\be),$ we obtain from Theorem~\ref{thmHomologEuler} a projective Euler operation 
\[
\begin{tikzcd}[column sep=large]
	H_*(M_\al/\G \t M_\be/\G;\Q)\rar{(\pi_{\al,\be})_{\th_{\al,\be}}^!} &H_{*-2-2\chi(\al,\be)}((M_\al\t M_\be)/\G;\Q).
\end{tikzcd}
\]

\begin{thm}
\label{GradedLie}
Let\/ $(M,\Phi,\Psi)$ be a commutative H-space with\/ $\G$-action, orientations\/ $\th_{\al,\be},$ and signs\/ $\ep_{\al,\be}$ satisfying Assumption~\textup{\ref{ass:lie}}.
Define a grading on\/ $H_*(M/\G;\Q)=\bigoplus\nolimits_{\al\in\pi_0(M)} H_*(M_\al/\G;\Q)$ by\/ $|\ze|'=a+2-\chi(\al,\al)$ for\/ $\ze\in H_a(M_\al/\G).$ Then there is a graded Lie bracket on\/ $H_*(M/\G;\Q)$ defined by
\begin{equation}
\label{def:Lie}
[\ze,\eta]=\ep_{\al,\be}(-1)^{a\chi(\be,\be)}\bigl(\Phi_{\al,\be}/\G\bigr)_*(\pi_{\al,\be})_{\th_{\al,\be}}^!(\ze\t\eta)
\end{equation}
for\/ $\ze\in H_a(M_\al/\G)$ and\/ $\eta\in H_b(M_\be/\G).$
\end{thm}

We prove Theorem~\ref{GradedLie} in \S\ref{s5}. In the rationally trivial case, this result is due to Joyce~\cite[\S3.4]{Joy}, where the reader may also find applications of Theorem~\ref{GradedLie}. For instance, the input data for Theorem~\ref{GradedLie} arises naturally from an additive $\C$-linear dg-category $\cA$: the direct sum of objects defines the H-space operation on the topological realization of $\cA,$ the $\G$-action comes from scaling morphisms by phase, and the orientations are obtained from the Ext-complexes.
\smallskip

The paper ends with an Appendix~\ref{sAppdx} in which we prove two technical but elementary results that are used in the text.
\smallskip

\noindent
\emph{Acknowledgments.}~The author would like to thank D.~Joyce for many useful discussions. He would also like to thank an anonymous referee for numerous suggestions.

\section{Pushforward operations and stability}
\label{s2}

In this section, we study pushforward operations for general groups. These are defined in \S\ref{s21} and we also prove an elementary classification result there. In \S\ref{s22} we define the notion of stability and prove a better classification result for stable pushforward operations. The behaviour with respect to the multiplicative structure is studied in \S\ref{SecMultProp}. In \S\ref{sComputation} we prove a technical result that will be useful in \S\ref{s4}. Throughout, we always use singular cohomology with rational coefficients.

\subsection*{Conventions for group actions}

We always consider an action of a group $G$ on a space $X$ to be from the left and denote the action by $\ell_X\colon G\t X\to X.$ The quotient space is written as $X/G=\{Gx\mid x\in X\}.$ The Cartesian product of $G$-spaces is equipped with the diagonal $G$-action. If $A\subset X$ is a $G$-invariant subspace, then the quotient space $X/A$ has a natural $G$-action.

\subsection{Unstable pushforward operations}
\label{s21}

Let $G$ be a topological group with a left action $\ell_F\colon G\t F\to F$ on a topological space $F.$

\begin{dfn}
\label{s2dfn1}
\hangindent\leftmargini
\text{(a)}\hskip\labelsep
Let $P\to B$ be a principal $G$-bundle with principal action $\ell_P\colon G\t P\to P.$ The \emph{associated fiber bundle} $P\ot_G F$ is the quotient of $P\t F$ by the anti-diagonal $G$-action $\ell_{P\t F}(g,(p,f))=(\ell_P(g,p),\ell_F(g,f)).$ The orbit of $(p,f)$ is denoted by $p\ot_G f.$
\begin{enumerate}
\setcounter{enumi}{1}
\item
An \emph{$F$-orientation} on a principal $G$-bundle $P\to B$ is a homotopy class $\th$ of $G$-equivariant maps $f_\th\colon P\to F,$ meaning they satisfy $f_\th(\ell_P(g,p))=\ell_F(g,f_\th(p)).$ If $\Phi\colon P_1\to P_2$ is a morphism of principal $G$-bundles and $\th_2$ is an $F$-orientation on $P_2,$ there is a \emph{pullback orientation} $\th_1=\Phi^*(\th_2)$ on $P_1$ realized by $f_{\th_1}=f_{\th_2}\circ\Phi$.
\item
A \emph{pushforward operation $\smash{\Xi^{m|n}}$ of type $(m,n)$} assigns to each principal $G$-bundle $P\to B$ with $F$-orientation $\th$ a possibly non-linear map
 \[
  \smash{\Xi^{m|n}_{P,\th}}\colon H^m(P)\longra H^n(B).
 \]
 For every morphism $\Phi\colon P_1\to P_2$ and $F$-orientations satisfying $\th_1=\Phi^*(\th_2),$ we require a commutative naturality diagram
 \begin{equation}
 \label{s2eq2}
  \begin{tikzcd}
   H^m(P_2)\arrow[r,"\Phi^*"]\arrow[d,"\Xi^{m|n}_{P_2,\th_2}"] & H^m(P_1)\arrow[d,"\Xi^{m|n}_{P_1,\th_1}"]\\
   H^n(B_2)\arrow[r,"\phi^*"] & H^n(B_1)\mathrlap{.}
  \end{tikzcd}
 \end{equation}
\item
A pushforward operation $\smash{\Xi^{m|n}}$ is \emph{pointed} if $\smash{\Xi^{m|n}_{P,\th}}(0_P)=0_B$ for all $P\to B$ and $\th$; it is \emph{linear} if each $\smash{\Xi^{m|n}_{P,\th}}$ is a linear map.
\end{enumerate}
\end{dfn}

In this paper, the main example of such a setup will be twisted K-theory, where we have $G=B\U(1)$ and $F=B\U\t\Z$. Our presentation follows Atiyah--Segal~\cite{AtSe04} and Freed--Hopkins--Teleman~\cite{FHT}.

\begin{dfn}
\label{Dfn_TwistedK}
Let $V(r)\to B\U(r)$ be the universal complex vector bundle of rank $r$ and let $\F=\colim B\U(r)$ be the classifying space for stable complex vector bundles. The external tensor product $V(1)\boxtimes V(r)$ by the universal complex line bundle is classified by a map $\G\t B\U(r)\to B\U(r).$ Consider the standard embeddings $B\U(r)\subset\F\t\{r\}\subset\F\t\Z.$ In suitable models for $\G$ and $\F,$ this construction can be made into a topological group action $\ell_{\F\t\Z}$ of a topological abelian group $\G$ on $\F\t\Z$ that is compatible with taking direct sums in $\F\t\Z.$ From this point of view, twisted K-theory $K_P(B),$ where $P\to B$ is a principal $\G$-bundle, is the set of all homotopy classes of $\G$-equivariant maps $f_\th\colon P\to\F\t\Z.$ Hence, in this case $F$-orientations are just classes in twisted K-theory. The direct sum operation on $\F\t\Z$ makes $K_P(B)$ into an abelian group. By ignoring the equivariance of $\th,$ every twisted K-theory class determines an \emph{underlying complex K-theory class} $\tilde\th\in K(P)$ satisfying $\ell_P^*(\tilde\th)=V(1)\boxtimes\tilde\th,$ where $\ell_P\colon G\t P\to P$ denotes the principal action. Applying $\ell_P$ to \eqref{cj-Tensor}, we find
\e
\label{PullbackUnderlyingK}
 \ell_P^*(c_j(\tilde\th))=\sum_{k+\ell=j}\binom{r-\ell}{k}c_1^k\t c_\ell(\tilde\th),\qquad r=\operatorname{rank}(\tilde\th),
\e
which will be useful later. If $P$ is trivial, then we can identify $f_\th$ with a homotopy class of maps $B\to\F\t\Z$ and thus an ordinary K-theory class $\vartheta\in K(B).$
\end{dfn}

We next review background from \cite[Def.~2.1]{MLS1986} on stable homotopy theory necessary for the construction and classification of pushforward operations.

\begin{dfn}
\label{s2dfn2}
Let $G$ be a topological group. A \emph{$G$-spectrum} $\{F_m,\varphi_m\}$ is a sequence of pointed $G$-spaces $(F_m,*_{F_m})$ for all $m\geqslant 0$ and $G$-equivariant based \emph{connecting maps} $\varphi_m\colon \Si F_m\to F_{m+1}$ on the reduced suspensions (the smash product by $S^1$ on the right). If all the adjoint maps $\varphi_m^\dagger\colon F_m\to \Om F_{m+1}$ are homeomorphisms, we call $\{F_m,\varphi_m\}$ an \emph{$\Om$-$G$-spectrum}.
\end{dfn}

\begin{ex}
\label{s2ex1}
The \emph{Eilenberg--Mac Lane $\Om$-spectrum} $\{H_m,\eta_m\}$ is characterized by $\pi_k(H_m)=0$ for $k\neq m$ and $\pi_m(H_m)=\Q$ (more precisely, this spectrum is commonly called $H\Q$). Here, $G$ is trivial. As in \cite{Eil}, there are natural isomorphisms $[X,H_m]\cong H^m(X)$ to the cohomology groups under which the connecting maps correspond to the suspension isomorphism. Putting $X=H_m,$ we obtain the \emph{tautological class} $[\id_{H_m}]=\jmath_m\in H^m(H_m).$ Every $\al\in H^m(X)$ determines a unique homotopy class of maps $h_\al\colon X\to H_m$ such that $\al=h_\al^*(\jmath_m).$ For example, the connecting map $\eta_m\colon\Si H_m\to H_{m+1}$ corresponds to the suspension $\si(\jmath_m)\in H^{m+1}(\Si H_m).$
\end{ex}

\begin{ex}
\label{s2ex2}
The \emph{mapping spectrum} $\{\Map(G,H_m),\mu_m^\dagger\}$ has the connecting maps $\mu_m^\dagger=(\eta_m^\dagger)_*\colon\Map(G,H_m)\to\Map(G,\Om H_{m+1})\cong\Om\Map(G,H_{m+1})$ and is an $\Om$-$G$-spectrum. Write $\mu_m\colon \Si\Map(G,H_m)\to\Map(G,H_{m+1})$ for the adjoint. Our convention for the action of $g\in G$ on $\mu\in\Map(G,H_m)$ is
\[
 \ell_{\Map}\colon G\t \Map(G,H_m)\longra \Map(G,H_m),\qquad \ell_{\Map}(g,\mu)(x)=\mu(xg).
\]
The base-point $\ast_{\Map}$ is the constant map to the base-point $\ast_{H_m}.$
\end{ex}

Let $EG\to BG$ be the universal principal $G$-bundle with base-points $*_{EG}$ and $*_{BG}$ and let $\ell_{EG}$ be the principal action. The next definition constructs a universal example $(P_m^\cla,\th_m^\cla,\al_m^\cla)$ among all triples $(P,\th,\al)$ consisting of a principal $G$-bundle $P$, an $F$-orientation $\th$ on $P$, and a cohomology class $\al\in H^m(P)$. The universality of this example is proven in Proposition~\ref{prop_univ_example} below. We also introduce spaces $E_m$ which classify pointed pushforward operations by Proposition~\ref{s2prop2}(b) below.

\begin{dfn}
\label{dfn_univ_example}
Let $F_m^\cla=F\t\Map(G,H_m).$ The space $P_m^\cla=EG\t F\t\Map(G,H_m)$ with the diagonal $G$-action is the total space of a principal $G$-bundle with quotient $B_m^\cla=EG\ot_G(F\t \Map(G,H_m)).$ The projections $F_m^\cla\to F$, $P_m^\cla\to EG\t F$, and $B_m^\cla\to EG\ot_G F$ are fiber bundles with fiber $\Map(G,H_m)$. The base-point $*_{\Map}$ determines sections $F\to F_m^\cla$, $EG\t F\to P_m^\cla$, and $EG\ot_G F\to B_m^\cla$ of these fiber bundles and we denote their cofibers by $F_m$, $P_m$, and $E_m$. All this is summarized in the diagram \eqref{diag_Overview_Of_Spaces}, which also shows the embeddings $k$ and $k_m$ defined using the base-point $*_{EG}\in EG$. Observe that each of the spaces in the first two rows of \eqref{diag_Overview_Of_Spaces} has a natural $G$-action. Moreover, the spaces in the bottom row of \eqref{diag_Overview_Of_Spaces} are all fiber bundles over $BG.$

The bundle $P_m^\cla$ has a natural $F$-orientation $\th_m^\cla$ given by $f_{\th_m^\cla}\colon(e,f,\mu)\mapsto f$ and a cohomology class $\al_m^\cla\in H^m(P_m^\cla)$ represented by $h_{\al_m^\cla}(e,f,\mu)=\mu(1).$

\begin{equation}
\label{diag_Overview_Of_Spaces}
\begin{tikzcd}
F\arrow[r,hook,bend left=10,start anchor={east},end anchor={west},yshift=0.5ex,"i_m"]\arrow[d,"k"] & F_m^\cla=F\t\Map(G,H_m)\arrow[r,two heads,"q_m"]\arrow[d,"k_m"]\arrow[l,bend left=10,start anchor={west},end anchor={east},two heads,yshift=-0.5ex] & F_m=\dfrac{F_m^\cla}{i_m(F)}\arrow[d,"k_m"]\\
EG\t F\arrow[r,hook,bend left=10,start anchor={east},end anchor={west},yshift=0.5ex,"i_m"]\arrow[d] & P_m^\cla=EG\t F\t\Map(G,H_m)\arrow[r,two heads,"q_m"]\arrow[d]\arrow[l,bend left=10,start anchor={west},end anchor={east},two heads,yshift=-0.5ex] & P_m=\dfrac{P_m^\cla}{i_m(EG\t F)}\arrow[d]\\
EG\ot_G F\arrow[r,hook,bend left=10,start anchor={east},end anchor={west},yshift=0.5ex,"i_m"] & B_m^\cla=EG\ot_G(F\t\Map(G,H_m))\arrow[r,two heads,"q_m"]\arrow[l,bend left=10,start anchor={west},end anchor={east},two heads,yshift=-0.5ex] & E_m=\dfrac{B_m^\cla}{i_m(EG\ot_G F)}
\end{tikzcd}
\end{equation}
\end{dfn}


\begin{prop}
\label{prop_univ_example}
For every triple\/ $(P,\th,\al)$ of a principal\/ $G$-bundle\/ $P\to B$ with $F$-orientation\/ $\th$ and\/ $\al\in H^m(P)$ there is a morphism\/ $\Phi_{P,\th,\al}=(\Phi_P,f_\th,\mu_\al)\colon P\to P_m^\cla$ of principal\/ $G$-bundles, unique up to homotopy, such that\/ $\th=\Phi_{P,\th,\al}^*(\th_m^\cla)$ and\/ $\al=\Phi_{P,\th,\al}^*(\al_m^\cla).$ If\/ $\al=h_\al^*(\jmath_m)$ for\/ $h_\al\colon P\to H_m,$ then\/ $\mu_\al\colon P\to\Map(G,H_m)$, $p\mapsto h_\al(gp).$

In particular, given\/ $(P_k,\th_k,\al_k)$ for\/ $k=1,2$ and a morphism\/ $\Phi\colon P_1\to P_2$ of principal\/ $G$-bundles such that\/ $\th_1=\Phi^*(\th_2)$ and\/ $\al_1=\Phi^*(\al_2),$ the morphisms\/ $\Phi_{P_1,\th_1,\al_1}$ and\/ $\Phi_{P_2,\th_2,\al_2}\circ\Phi$ are homotopic.
\end{prop}

\begin{proof}
A morphism $\Phi_{P,\th,\al}\colon P\to P_m^\cla$ has three components $\Phi_{P,\th,\al}=(\Phi_P,f_\th,\mu_\al).$ The first component $\Phi_P$ is a classifying map for the principal $G$-bundle $P$, which always exists and whose homotopy class is uniquely determined by $P$. The homotopy class of the second component $f_\th\colon P\to F$ amounts to an $F$-orientation $\th$ on $P$, which can be written as $\th=\Phi_{P,\th,\al}^*(\th_m^\cla)$. The third component is a $G$-equivariant map $\mu_{\al}\colon P\to\Map(G,H_m),$ $p\mapsto \mu_{\al,p},$ which corresponds to a non-equivariant map $h_\al\colon P\to H_m,$ via the formula $\mu_{\al,p}(g)=h_\al(gp)$ for all $p\in P$, $g\in G.$ The homotopy class of the third component thus amounts to a cohomology class $\al\in H^m(P)$, which can be written as $\al=\Phi_{P,\th,\al}^*(\al_m^\cla)$.
\end{proof}

The morphism $\Phi_{P,\th,\al}\colon P\to P_m^\cla$ in Proposition~\ref{prop_univ_example} is called a \emph{classifying morphism} for $(P,\th,\al)$ and we write $\phi_{P,\th,\al}\colon B\to B_m^\cla$ for the quotient \emph{classifying map}.
\begin{equation}
 \begin{tikzcd}[column sep=huge]
 	P\arrow[d,"\pi"]\arrow[r,"\Phi_{P,\th,\al}"] & P_m^\cla\arrow[d,"\pi_m^\cla"]\\
 	B\arrow[r,"\phi_{P,\th,\al}"] & B_m^\cla
 \end{tikzcd}
\label{ClassifyingDiagram}
\end{equation}

\begin{prop}
\label{s2prop2}
\hangindent\leftmargini
\textup{(a)}\hskip\labelsep
Pushforward operations\/ $\Xi^{m|n}$ for principal\/ $G$-bundles with\/ $F$-orientation are in \textup{1-1} correspondence with classes\/ $\smash{\Xi^{m|n}_\cla}\in H^n(B_m^\cla).$ This correspondence is defined by
\ea
 \Xi_{P,\th}^{m|n}(\al)&=\phi_{P,\th,\al}^*(\Xi^{m|n}_\cla),\label{XiFromCla}\\
 \Xi^{m|n}_\cla&=\Xi^{m|n}_{P_m^\cla,\th_m^\cla}(\al_m^\cla).\label{ClaFromXi}
\ea
\begin{enumerate}
\setcounter{enumi}{1}
\item
A pushforward operation\/ $\smash{\Xi^{m|n}}$ is pointed if and only if\/ $i_m^*\bigl(\Xi^{m|n}_\cla\bigr)=0.$
\item
If\/ $P=G\t B$ is trivial, then an\/ $F$-orientation corresponds to a map\/ $f_\vartheta\colon B\to F$ by\/ $f_\vartheta(b)=f_\th(1,b)$ and a cohomology class\/ $\al\in H^m(P)$ corresponds to a map\/ $h_\al^\dagger\colon B\to\Map(G,H_m)$ by\/ $\al=h_\al^*(\jmath_m)$ for the adjoint\/ $h_\al\colon G\t B\to H_m$ of\/ $h_\al^\dagger.$
Then\/ $\phi_{P,\th,\al}=*_{EG}\ot_G(f_\vartheta,h_\al^\dagger)$ is a classifying map, which maps into the fiber\/ $F\t\Map(G,H_m)$ of\/ $B_m^\cla\to BG.$ If\/ $\Xi^{m|n}_\cla\big|_{F\t\Map(G,H_m)}=\xi\t x$ for\/ $\xi\in H^*(F)$ and $x\in H^*(\Map(G,H_m)),$ then
\e
\label{XmnTrivialFormula}
\Xi^{m|n}_{P,\th}(\al)=f_\vartheta^*(\xi)\cup\bigl(h_\al^\dagger\bigr)^*(x).
\e
\end{enumerate}
\end{prop}

\begin{proof}
(a)\hskip\labelsep
By the naturality of the classifying maps in Proposition~\ref{prop_univ_example}, the construction \eqref{XiFromCla} indeed defines a pushforward operation for each $\Xi^{m|n}_\cla\in H^n(B_m^\cla).$ By the uniqueness part of Proposition~\ref{prop_univ_example}, $\phi_{P_m^\cla,\th_m^\cla,\al_m^\cla}=\id_{B_m^\cla},$ from which it follows that the class of this operation is the original class $\Xi^{m|n}_\cla.$ Conversely, applying \eqref{XiFromCla} to the class from \eqref{ClaFromXi} gives
\[
 \phi^*_{P,\th,\al}(\Xi^{m|n}_\cla)=\phi^*_{P,\th,\al}(\Xi^{m|n}_{P_m^\cla,\th_m^\cla}(\al_m^\cla))\overset{\eqref{s2eq2}}{=}\Xi^{m|n}_{P,\th}(\Phi^*_{P,\th,\al}(\al_m^\cla))=\Xi^{m|n}_{P,\th}(\al).
\]

\noindent
(b)\hskip\labelsep
The space $Q^\cla=EG\t F$ is the total space of a principal $G$-bundle $Q^\cla\to EG\ot_G F$. The map $Q^\cla\to F,$ $(e,f)\mapsto f$ represents an $F$-orientation $\th^\cla$ on $Q^\cla$.

Since for $\al=0$ we can take the constant map $h_\al\colon EG\t F\to H_m$, the section $i_m\colon EG\ot_G F\to B_m^\cla$ from Definition~\ref{dfn_univ_example} is a classifying map $\phi_{Q^\cla,\th^\cla,0}$ for $(Q^\cla,\th^\cla,0).$ If $\Xi^{m|n}$ is pointed, then $i_m^*(\Xi^{m|n}_\cla)=\phi_{P,\th,0}^*(\Xi^{m|n}_\cla)\overset{\eqref{XiFromCla}}{=}\Xi^{m|n}_{P,\th}(0)=0.$

Conversely, suppose that $i_m^*(\Xi^{m|n}_\cla)=0$. Let $P\to B$ be a principal $G$-bundle with $F$-orientation $\th.$ Choose a classifying morphism $\Phi_P\colon P\to EG$ and a representative $f_\th\colon P\to F$ of the $F$-orientation $\th$. Together, these define a classifying morphism $\Phi_{P,\th}=(\Phi_P,f_\th)\colon P\to Q^\cla$ with $\th=\Phi_{P,\th}^*(\th^\cla)$. Let $\phi_{P,\th}\colon B\to EG\ot_G F$ be the quotient classifying map. Since $\Phi_{P,\th}^*(0)=0,$ the uniqueness part of Proposition~\ref{prop_univ_example} implies that the classifying map $\phi_{P,\th,0}\colon B\to B_m^\cla$ is homotopic to $i_m\circ \phi_{P,\th},$ hence $\Xi^{m|n}_{P,\th}(0)=\phi_{P,\th,0}^*(\Xi^{m|n}_\cla)=\phi_{P,\th}^*i_m^*(\Xi^{m|n}_\cla)=0.$
\smallskip

\noindent
(c)\hskip\labelsep
For the section $s\colon B\to G\t B,$ $b\mapsto (1,b)$ of $P$ we have $\pi\circ s=\id_B$ and hence the commutativity of \eqref{ClassifyingDiagram} implies $\phi_{P,\th,\al}=\pi_m^\cla\circ\Phi_{P,\th,\al}\circ s.$ The trivial bundle is classified by $\Phi_P\colon P\to EG,$ $(g,b)\mapsto\ell_{EG}(g,*_{EG}),$ so $\phi_{P,\th,\al}(b)=*_{EG}\ot_G (f_\th(1,b),\mu_\al(1,b)).$ Since $\mu_\al(1,b)\colon g\mapsto h_\al(g,b),$ we have $\mu_\al(1,b)=h_\al^\dagger,$ proving the claim about the classifying map. If we identify $*_{EG}\ot_G (F\t\Map(G,H_m))\cong F\t\Map(G,H_m)$ (as done in our assumption on $\Xi^{m|n}_\cla$), then $\phi_{P,\th,\al}$ becomes $(f_\vartheta,h_\al^\dagger)$ and hence $\Xi^{m|n}_{P,\th}(\al)=\phi_{P,\th,\al}^*(\Xi^{m|n}_\cla)=(f_\vartheta,h_\al^\dagger)^*(\xi\t x)=f_\vartheta^*(\xi)\cup(h_\al^\dagger)^*(x).$
\end{proof}

Consider the bottom row in \eqref{diag_Overview_Of_Spaces}. This is a cofiber sequence and the map $i_m$ has a retraction, so the long exact sequence of the pair $(B_m^\cla,EG\ot_G F)$ reduces to a short exact sequence
\begin{equation}
\label{KESij}
 \begin{tikzcd}[row sep=0]
 	0\rar & \widetilde H^n(E_m)\rar{q_m^*} & H^n(B_m^\cla)\rar{i_m^*} & H^n(EG\ot_G F)\rar & 0.
 \end{tikzcd}
\end{equation}
By Proposition~\ref{s2prop2}(b), the class of a pointed pushforward operation is in the image of $q_m^*$ and we can thus regard $\Xi^{m|n}_\cla$ as a class in the subspace $\widetilde H^n(E_m).$

\subsection{Stable pushforward operations}
\label{s22}

\begin{dfn}
\label{s22def0}
Let $\smash{\Xi^{m+1|n+1}}$ be a pointed pushforward operation for principal $G$-bundles with $F$-orientation of type $(m+1,n+1).$ The \emph{desuspension} is the pointed pushforward operation $\si(\Xi^{m+1|n+1})$ of type $(m,n)$ defined as follows. Let $P\to B$ be a principal $G$-bundle and equip the pullback $P\t S^1\to B\t S^1$ along the projection $\pi_B\colon B\t S^1\to B$ with the pullback $F$-orientation $\th\t S^1=\pi_P^*(\th),$ where $\pi_P\colon P\t S^1\to P$ is the projection. The long exact sequence of the pair $(P\t S^1, P\t\{1\})$ splits into short exact sequences
 \begin{equation}
 \label{s2eq7}
  \begin{tikzcd}[column sep=6ex]
   0\arrow[r]
   &
   H^m(P)\arrow[r,"\varsigma_P"]\arrow[dashed,d,"\si(\Xi^{m+1|n+1})_{P,\th}"]
   &
   H^{m+1}(P\t S^1)\rar\arrow[d,"\Xi^{m+1|n+1}_{P\t S^1,\th\t S^1}"]
   &
   H^{m+1}(P)\arrow[r]\arrow[d,"\Xi^{m+1|n+1}_{P,\th}"]
   &
   0
   \\
   0\arrow[r]
   &
   H^n(B)
   \arrow[r,"\varsigma_B"]
   &
   H^{n+1}(B\t S^1)\rar
   &
   H^{n+1}(B)\arrow[r]
   &
   0\mathrlap{,}
  \end{tikzcd}
 \end{equation}
where the maps\/ $\varsigma_P$, $\varsigma_B$ are given by cross product with the fundamental cohomology class in $H^1(S^1)$.
 The right-hand square commutes by naturality \eqref{s2eq2}. Define $\si(\Xi^{m+1|n+1})_{P,\th}$ as the restriction indicated in \eqref{s2eq7}, using exactness and $\Xi^{m+1|n+1}_{P,\th}(0)=0.$
\end{dfn} 

\begin{dfn}
A \emph{stable} pushforward operation $\smash{\Xi^{*|*+k}}$ of degree $k\in\Z$ for principal $G$-bundles with $F$-orientation is a sequence of pointed pushforward operations $\Xi^{m|m+k}$ of type $(m,m+k)$ for all $m\geqslant 0$ such that $\si(\Xi^{m+1|m+1+k})=\Xi^{m|m+k}.$ We sometimes omit the superscripts and write $\Xi_{P,\th}(\al)=\Xi_{P,\th}^{m|m+n}(\al)$ if $\al\in H^m(P).$
\end{dfn}

Our next goal is to classify stable pushforward operations.

\begin{dfn}
Recall that $B_m^\cla=EG\ot_G(F\t\Map(G,H_m))$ is a fiber bundle over $EG\ot_G F.$ The fiberwise suspension of this bundle is the space $\Si^f B_m^\cla=EG\ot_G(F\t\Si\Map(G,H_m)),$ which is again a fiber bundle over $EG\ot_G F$ with a natural base-point section. There are connecting maps
\begin{align*}
 \be_m\colon\Si^f B_m^\cla&\longra B_{m+1}^\cla,
&e\ot_G(f,\mu\wedge t)&\longmapsto e\ot_G(f,\eta_m(\mu(-)\wedge t)),
\end{align*}
where $e\in EG,$ $f\in F,$ $\mu\in\Map(G,H_m),$ $t\in S^1,$ and where $\eta_m$ is the connecting map of the Eilenberg--Mac Lane spectrum.

The maps $B_m^\cla\to\Om^f B_{m+1}^\cla$ adjoint to $\be_m$ into the fiberwise loop space are homeomorphisms.

\end{dfn}

\begin{rem}
The collection of fiber bundles $B_m^\cla\to EG\ot_G F$ together with the base-point sections $i_m$ and maps $\be_m\colon\Si^f B_m^\cla\to B_{m+1}^\cla$ form a \emph{parameterized spectrum} over $EG\ot_G F$. For a comprehensive introduction to (equivariant) parameterized spectra, see May--Sigurdsson \cite{MS}.
\end{rem}

\begin{dfn}
\label{Dfn_EmFmSpectra}
Recall the cofiber sequence $EG\ot_G F\xrightarrow{i_m} B_m^\cla\xrightarrow{q_m}E_m$ and the inclusion $k_m\colon F_m\to E_m$ from \eqref{diag_Overview_Of_Spaces}. There are unique connecting maps $\ep_m\colon\Si E_m\to E_{m+1}$ and $\varphi_m\colon \Si F_m\to F_{m+1}$ such that the following diagram commutes.
\begin{equation}
\begin{tikzcd}
\Si^f B_m^\cla\rar{\Si^f q_m}\dar{\be_m} & \Si E_m\rar{\Si k_m}\dar{\ep_m} & \Si F_m\dar{\varphi_m}\\
B_{m+1}^\cla\rar{q_{m+1}} & E_{m+1}\rar{k_{m+1}} & F_{m+1}
\end{tikzcd}
\label{diag:DefEpsilonVarphi}
\end{equation}
Each $F_m$ has a $G$-action and the connecting maps $\varphi_m$ are $G$-equivariant (we may thus view $\{F_m,\varphi_m\}$ as a na\"ive $G$-spectrum).

The long exact sequences in cohomology of the pairs $(E_m,F_m)$ and $(\Si E_m, \Si F_m)$ and the suspension isomorphisms yield a commutative diagram with exact rows:
\begin{equation}
\label{inverseSystem}
\begin{tikzcd}[column sep=3.5ex]
	\rar & H^{n+1}(E_{m+1},F_{m+1})\dar{(\ep_m,\varphi_m)^*}\rar & \widetilde H^{n+1}(E_{m+1})\dar{\ep_m^*}\rar & \widetilde H^{n+1}(F_{m+1})\dar{\varphi_m^*}\rar & \,\\
	\rar & \begin{array}[t]{r}H^{n+1}(\Si E_m,\Si F_m)\\\cong H^n(E_m,F_m)\end{array}\rar & \begin{array}[t]{r}\widetilde H^{n+1}(\Si E_m)\\\cong \widetilde H^n(E_m)\end{array}\rar & \begin{array}[t]{r}\widetilde H^{n+1}(\Si F_m)\\\cong \widetilde H^n(F_m).\end{array}\rar & \,
\end{tikzcd}
\end{equation}
The vertical maps determine inverse systems for each $m.$ Since taking the inverse limit preserves exactness over $\Q,$ we obtain for all $m, k$ an exact sequence:
 \begin{equation}
 \label{s2eq9}
  \begin{tikzcd}[column sep=small, nodes in empty cells]
   \,\arrow[r] & \lim H^{m+k}(E_m,F_m)\arrow[r] & \lim \widetilde H^{m+k}(E_m)\arrow[r] & \lim \widetilde H^{m+k}(F_m)\arrow[r] &\cdots.\!\!\!
  \end{tikzcd}
 \end{equation}
\end{dfn}

We can now state the classification result for stable pushforward operations.

\begin{thm}
\label{s2thm1}
The class of the desuspension\/ $\si(\Xi^{m+1|n+1})_\cla \in \widetilde{H}^n(E_m)$ of a pointed pushforward operation\/ $\Xi^{m+1|n+1}$ is the image of\/ $\Xi^{m+1|n+1}_\cla\in\widetilde{H}^{n+1}(E_{m+1})$ under the middle vertical map in \eqref{inverseSystem}.
In particular, stable pushforward operations of degree\/ $k$ for principal\/ $G$-bundles with\/ $F$-orientation correspond bijectively to elements\/ $\Xi_\cla^k=\bigl\{\Xi_\cla^{m|m+k}\bigr\}_{m\geqslant 0}$ in the inverse limit\/ $
\lim\widetilde H^{m+k}(E_m),$ viewed as a subset of the Cartesian product $\prod_m \widetilde H^{m+k}(E_m).$
\end{thm}

\begin{proof}
Using the notation of Definition~\ref{dfn_univ_example}, equip the pullback bundle $P_m^\cla\t S^1$ with the pullback $F$-orientation $\th_m^\cla\t S^1$ and the cohomology class $\varsigma_{P_m^\cla}(\al_m^\cla)\in H^{m+1}(P_m^\cla\t S^1).$
Let $\Ga_m=\Phi_{P_m^\cla\t S^1,\th_m^\cla\t S^1,\varsigma_{P_m^\cla}(\al_m^\cla)}$ be the classifying morphism from Proposition~\ref{prop_univ_example}, which satisfies $\Ga_m^*(\al_{m+1}^\cla)=\varsigma_{P_m^\cla}(\al_m^\cla)$ and $\Ga_m^*(\th_{m+1}^\cla)=\th_m\t S^1.$ Let $\ga_m$ be the quotient map of $\Ga_m.$ We then compute
 \begin{align*}
 \ga_m^*\bigl(\Xi^{m+1|n+1}_\cla\bigr)&=\ga_m^*\bigl(\Xi^{m+1|n+1}_{P_{m+1}^\cla,\th_{m+1}^\cla}(\al^\cla_{m+1})\bigr) && \text{by \eqref{ClaFromXi}}\\
  &=\Xi^{m+1|n+1}_{P_m^\cla\t S^1,\th_m^\cla\t S^1}\bigl(\Ga_m^*(\al^\cla_{m+1})\bigr) && \text{by nat.~\eqref{s2eq2}}\\
  &=\Xi^{m+1|n+1}_{P_m^\cla\t S^1,\th_m^\cla\t S^1}(\varsigma_{P_m^\cla}(\al_m^\cla))\\
  &=\varsigma_{B_m^\cla}\bigl(\si(\Xi^{m+1|n+1})_{P_m^\cla,\th_m^\cla}(\al_m^\cla)\bigr) &&\text{by \eqref{s2eq7}.}
 \end{align*}
Hence $\ga_m^*(\Xi^{m+1|n+1}_\cla)=\varsigma_{B_m^\cla}(\si(\Xi^{m+1|n+1})_\cla)$ holds in $H^{n+1}(B_m^\cla\t S^1).$ This equation appears in the top row of the commutative diagram
\[
\begin{tikzcd}
	H^{n+1}(B_{m+1}^\cla)\rar{\ga_m^*} & H^{n+1}(B_m^\cla\t S^1) & H^n(B_m^\cla)\arrow[l,hook',"\varsigma_{B_m^\cla}"']\\
	\widetilde H^{n+1}(E_{m+1})\rar{\ep_m^*}\arrow[u,hook',"q_{m+1}^*"] & \widetilde H^{n+1}(\Si E_m)\arrow[u,"(q_m\t\id_{S^1})^*"] & \widetilde H^n(E_m)\arrow[u,hook',"q_m^*"]\mathrlap{,}\arrow[l,"\cong"']
\end{tikzcd}
\]
whose outer vertical maps are injective, by \eqref{KESij}. Since the suspension at the bottom is bijective, the middle vertical map is injective as well. We conclude that $\ep_m^*(\Xi^{m+1|n+1}_\cla)$ is the suspension of $\Xi^{m|n}_\cla$ in the group $\widetilde H^{n+1}(\Si E_m),$ as required.
\end{proof}

\begin{thm}
Every stable pushforward operation is automatically\/ $\Q$-linear.
\end{thm}

\begin{proof}
Every group morphism between $\Q$-vector spaces is automatically $\Q$-linear, so it suffices to prove that stable pushforward operations $\Xi^{m|m+k}_{P,\th}\colon H^m(P)\to H^{m+k}(B)$ are additive. By \eqref{KESij}, a pointed pushforward operation $\Xi^{m|m+k}_{P,\th}$ is represented by a class $\Xi^{m|m+k}_\cla$ in the subspace $\widetilde H^{m+k}(E_m)\subset H^n(B_m^\cla)$, so there is a map $h_m\colon B_m^\cla\to H_{m+k}$ with $h_m\circ i_m=*_{H_{m+k}}$ such that $\Xi^{m|m+k}_\cla=h_m^*(\jmath_{m+k})$, where $\jmath_{m+k}$ is the tautological class from Example~\ref{s2ex1}. Consider the diagram
\begin{equation}
\label{s2eq10}
 \begin{tikzcd}[column sep=huge]
 	\Si E_m\arrow[dr,"\Si^f h_m\circ\Si^f q_m"' xshift=2ex]\arrow[rrr,"\ep_m",bend left=10] & \Si^f B_m^\cla\arrow[d,"\Si^f h_m"]\arrow[r,"\be_m"]\arrow[l,"\Si^f q_m"] & B_{m+1}^\cla\arrow[r,"q_{m+1}"']\arrow[d,"h_{m+1}"'] & E_{m+1}\mathrlap{.}\arrow[dl,"h_{m+1}\circ q_{m+1}"]\\
 	& \Si H_{m+k}\rar{\eta_{m+k}} & H_{m+1+k}
 \end{tikzcd}
\end{equation}
The top quadrilateral commutes by \eqref{diag:DefEpsilonVarphi} and the outer triangles obviously commute. By Theorem~\ref{s2thm1} the class $\Xi^{m|m+k}_\cla$ is the image of $\Xi^{m+1|m+1+k}_\cla$ under the middle vertical map in \eqref{inverseSystem}, which means that the outer square commutes up to homotopy. It follows that the bottom square in \eqref{s2eq10} commutes up to homotopy.

Let $\star$ be the loop composition and $\star^f$ the fiberwise loop composition on the fiberwise loop space $\Om^f.$ Let $P\to B$ be a principal $G$-bundle and let $\al, \be\in H^m(P).$ Then the representative map $h_{\al+\be}$ for $\al+\be$ can be taken to be the composite
\[
\begin{tikzcd}[column sep=7ex]
P\rar{(h_\al,h_\be)} & H_m\t H_m\rar{\eta_m^\dagger\t\eta_m^\dagger} & \Om H_{m+1}\t\Om H_{m+1}\rar{\star} & \Om H_{m+1}\rar{(\eta_m^\dagger)^{-1}} & H_m.
\end{tikzcd}
\]
Hence $\phi_{P,\th,\al+\be}$ is the composition of $(\phi_{P,\th,\al},\phi_{P,\th,\be})\colon B\to B_m^\cla\t_C B_m^\cla,$ where $C=EG\ot_G F,$ with the upper row of the homotopy-commutative diagram
\[
\begin{tikzcd}[row sep=large, column sep=7ex]
B_m^\cla\t_C B_m^\cla\rar{\be_m^\dagger\t_C\be_m^\dagger}\dar[swap]{(h_m)^{\t2}}\arrow[rd,phantom,pos=0.1,"\text{\small\eqref{s2eq10}}"] & \Om^f B_{m+1}^\cla\t_C\Om^f B_{m+1}^\cla\rar{\star^f}\dar[swap]{\Om^f(h_{m+1})^{\t2}}\arrow[rd,phantom,pos=0.1,"\text{\small(obvious)}"] & \Om^f B_{m+1}^\cla\rar{(\be_m^\dagger)^{-1}}\dar[swap]{\Om^f(h_{m+1})}\arrow[rd,phantom,pos=0.2,"\text{\small\eqref{s2eq10}}"] & B_m^\cla\dar[swap]{h_m}\\
H_{m+k}\t H_{m+k}\rar[outer sep=5pt]{\eta_{m+k}^\dagger\t\eta_{m+k}^\dagger} & \Om H_{m+1+k}\t\Om H_{m+1+k}\rar{\star} & \Om H_{m+1+k}\rar{(\eta_{m+k}^\dagger)^{-1}} & H_{m+k}\mathrlap{.}
\end{tikzcd}
\]
Here, the outer two squares homotopy-commute by the adjoint of the inner square in \eqref{s2eq10}. The lower row represents the addition $a\colon H_{m+k}\t H_{m+k}\to H_{m+k},$ so $a^*(\jmath_{m+k})=\pi_1^*(\jmath_{m+k})+\pi_2^*(\jmath_{m+k})$ for the two projections $\pi_1$ and $\pi_2.$ Compute
\begin{align*}
&\Xi^{m|m+k}_{P,\th}(\al+\be)=\phi_{P,\th,\al+\be}^*(\Xi^{m|m+k}_\cla)=\phi_{P,\th,\al+\be}^*h_m^*(\jmath_{m+k})\\
&=(\phi_{P,\th,\al},\phi_{P,\th,\be})^*(h_m\t h_m)^*a^*(\jmath_{m+k})
=\phi_{P,\th,\al}^*h_m^*(\jmath_{m+k})+\phi_{P,\th,\be}^*h_m^*(\jmath_{m+k})\\
&=\phi_{P,\th,\al}^*(\Xi^{m|m+k}_\cla)+\phi_{P,\th,\be}^*(\Xi^{m|m+k}_\cla)
=\Xi^{m|m+k}_{P,\th}(\al)+\Xi^{m|m+k}_{P,\th}(\be).\qedhere
\end{align*}
\end{proof}

\subsection{Multiplicative properties}
\label{SecMultProp}

Recall that the Eilenberg-Mac Lane ring spectrum is equipped with product maps $\mathrm{M}_{m_1,m_2}\colon H_{m_1}\t H_{m_2}\to H_{m_1+m_2}$ characterized by $\mathrm{M}_{m_1,m_2}^*(\jmath_{m_1+m_2})=\jmath_{m_1}\t\jmath_{m_2}$ for the tautological classes from Example~\ref{s2ex1}. Define $G$-equivariant maps $\ol\rho_{m_1,m_2}\colon \Map(G,H_{m_1})\t H_{m_2}\to\Map(G,H_{m_1+m_2})$ by $\ol\rho_{m_1,m_2}(\mu,h)\colon g\mapsto \mathrm{M}_{m_1,m_2}(\mu(g),h)$. In the following diagram, observe that the upper horizontal map is $G$-equivariant and therefore descends to a map $\rho_{m_1,m_2}$.
\begin{equation}
\label{rhoPullback}
 \begin{tikzcd}[column sep=20ex]
 	P_{m_1}^\cla\t H_{m_2}\rar{\id_{EG}\t\id_F\t\ol\rho_{m_1,m_2}}\dar & P_{m_1+m_2}^\cla\dar\\
 	B_{m_1}^\cla\t H_{m_2}\rar{\rho_{m_1,m_2}} & B_{m_1+m_2}^\cla
 \end{tikzcd}
\end{equation}
The $F$-orientation $\th_{m_1+m_2}^\cla$ and $\al_{m_1+m_2}^\cla$ on $P_{m_1+m_2}^\cla$ pull back to an $F$-orientation $\tilde\th_{m_1,m_2}^\cla$ and the class $\tilde\al_{m_1,m_2}^\cla\allowbreak=(\id_{EG}\t\allowbreak\id_F\t\allowbreak\ol\rho_{m_1,m_2})^*\allowbreak(\al_{m_1+m_2}^\cla)=\al_{m_1}^\cla\t\jmath_{m_2}$.

\begin{prop}
\label{s2prop5}
For a stable pushforward operation\/ $\Xi^{*|*+k}$ with class\/ $\Xi^k_\cla=\bigl\{\Xi^{m|m+k}_\cla\bigr\},$ where\/ $\Xi^{m|m+k}_\cla\in H^{m+k}(B_m^\cla),$ the following are equivalent.

\begin{enumerate}
\item
For every principal\/ $G$-bundle\/ $P\to B$ with\/ $F$-orientation\/ $\th$, $\al\in H^{m_1}(P)$, topological space\/ $B'$, $\be\in H^{m_2}(B')$, and projection\/ $\pi_P\colon P\t B'\to P$, we have\/ $\Xi^{m_1+m_2|m_1+m_2+k}_{P\t B',\pi_P^*(\th)}(\al\t\be)=\Xi^{m_1|m_1+k}_{P,\th}(\al)\t\be$.

Equivalently,\/ $\Xi^{m_1+m_2|m_1+m_2+k}_{P,\th}(\al\cup\pi^*(\be))=\Xi^{m_1|m_1+k}_{P,\th}(\al)\cup\be$ for every principal\/ $G$-bundle\/ $P$, $\th$, $\al\in H^{m_1}(P)$, and\/ $\be\in H^{m_2}(B).$
\item
$\rho_{m_1,m_2}^*\bigl(\Xi^{m_1+m_2|m_1+m_2+k}_\cla\bigr)=\Xi^{m_1|m_1+k}_\cla\t\jmath_{m_2}.$
\end{enumerate}


\end{prop}

\begin{proof}
Observe that the naturality \eqref{s2eq2} implies that (a) can be stated equivalently for internal or for external products.
\smallskip

\noindent
$\text{(b)}\Rightarrow\text{(a)}.$
Choose a map $h_\be\colon B'\to H_{m_2}$ with $h_\be^*(\jmath_{m_2})=\be$. If $\Phi_{P,\th,\al}$ is a classifying morphism as in Proposition~\ref{prop_univ_example}, then $\Phi_{P,\th,\al}\t h_\be\colon P\t B'\to P_{m_1}^\cla\t H_{m_2}$ composed with the upper horizontal map in \eqref{rhoPullback} is a classifying morphism $\Phi_{P\t B',\pi_P^*(\th),\al\t\be}$ since $(\Phi_{P,\th,\al}\t h_\be)^*(\tilde\al_{m_1,m_2}^\cla)=\Phi_{P,\th,\al}^*(\al_{m_1}^\cla)\t h_\be^*(\jmath_{m_2})=\al\t\be$. Taking quotients, this implies $\rho_{m_1,m_2}\circ(\phi_{P,\th,\al}\t h_\be)=\phi_{P\t B',\pi_P^*(\th),\al\t\be}$ for the classifying maps. The pullback of equation (b) along $\phi_{P,\th,\al}\t h_\be\colon B\t B'\to B_{m_1}^\cla\t H_{m_2}$ is therefore $(\phi_{P\t B',\pi_P^*(\th),\al\t\be})^*\allowbreak(\Xi^{m_1+m_2|m_1+m_2+k}_\cla)=(\phi_{P,\th,\al}\t h_\be)^*\allowbreak(\Xi^{m_1|m_1+k}_\cla\t\jmath_{m_2})=\phi_{P,\th,\al}^*(\Xi^{m_1|m_1+k}_\cla)\t\be$ and then \eqref{XiFromCla} completes the proof of (a).
\smallskip

\noindent
$\text{(a)}\Rightarrow\text{(b)}.$
By naturality \eqref{s2eq2} applied to \eqref{rhoPullback} we get the equation
\[
 \Xi^{m_1+m_2|m_1+m_2+k}_{P_{m_1}^\cla\t H_{m_2},\tilde\th_{m_1,m_2}^\cla}(\al_{m_1}^\cla\t\jmath_{m_2})=\rho_{m_1,m_2}^*\bigl(\Xi_{P_{m_1+m_2}^\cla,\th_{m_1+m_2}^\cla}^{m_1+m_2|m_1+m_2+k}(\al_{m_1+m_2}^\cla)\bigr).
\]
As $\tilde\th_{m_1,m_2}^\cla=\pi_{P_{m_1}^\cla}^*(\th_{m_1}^\cla)$ for the projection $\pi_{P_{m_1}}\colon P_{m_1}\t H_{m_2}\to P_{m_1}$, we can use (a) to rewrite the left hand side of the equation as $\Xi^{m_1|m_1+k}_{P_{m_1}^\cla,\th_{m_1}^\cla}(\al_{m_1}^\cla)\t\jmath_{m_2}$, while by \eqref{ClaFromXi} the right hand side of the equation is $\rho_{m_1,m_2}^*(\Xi_\cla^{m_1+m_2|m_1+m_2+k})$.
\end{proof}

\begin{dfn}
\label{dfn_EquivCharClass}
Let $P\to B$ be a principal $G$-bundle with $F$-orientation $\th.$ Choose a representative map $f_\th\colon P\to F$ for $\th$ and write $S_\th\colon B\to P\ot_G F$ for the associated section, defined by $S_\th(b)=p\ot_G f$ if $f_\th(p)=f$ and $p\in P|_b.$ Let $\Phi_P\colon P\to EG$ be a classifying morphism for $P.$ For $\xi_0^G\in H^n(EG\ot_G F),$ define the \emph{characteristic class}
\e
\label{eqn_EquivCharClass}
 \xi_0^G(\th)=S_\th^*(\Phi_P\ot_G\id_F)^*(\xi_0^G)\in H^n(B).
\e
\end{dfn}

Let $\pi_F$ and $\pi_{\Map}$ be the projections of $B_0^\cla=EG\ot_G(F\t \Map(G,H_0))$ onto $EG\ot_G F$ and $EG\ot_G\Map(G,H_0).$ Let $j\colon \Map(G,H_0)\to EG\ot_G\Map(G,H_0)$ be the inclusion $j(\mu)=*_{EG}\ot_G\mu$ and $\ev\colon G\t\Map(G,H_m)\to H_m$ be the evaluation map. If $G$ is connected, then $H_0(G)\cong\Q$ has a canonical basis $t^0\in H_0(G)$ characterized by $\an{t^0,1_{H^*(G)}}=1$ using the Kronecker pairing.

For a topological space $X$, write $1_{H^*(X)}$ for the unit in the ring $H^*(X)$. Write $H_m(X)\t H^{m+n}(X\t Y)\to H^n(Y),$ $(t,\al)\mapsto t\backslash\al$ for the slant product.

\begin{prop}
\label{PropImageof1}
Let\/ $G$ be a connected topological group.
Suppose that\/ $\Xi^{0|n}_\cla=\pi_F^*(\xi_0^G)\cup\pi_{\Map}^*(x)$ for\/ $\xi_0^G\in H^n(EG\ot_G F)$ and for\/ $x\in H^0(EG\ot_G\Map(G,H_0))$ such that\/ $j^*(x)=t^0\backslash\ev^*(\jmath_0).$ Then\/ $\Xi^{0|n}_{P,\th}(1_{H^*(P)})=\xi_0^G(\th)$ for every principal\/ $G$-bundle\/ $\pi\colon P\to B$ with\/ $F$-orientation\/ $\th.$
\end{prop}

\begin{proof}
Let $\ka\colon H_0\to\Map(G,H_0)$ be the inclusion of $H_0$ as the constant maps. Let $h_{1_{H^*(B)}}\colon B\to H_0$ and $h_{1_{H^*(P)}}\colon P\to H_0$ be the constant maps to the base-point of $H_0,$ which represent the cohomology classes $1_{H^*(B)}$ and $1_{H^*(P)}.$ If $\Phi_P\colon P\to EG$ is a classifying map for $P,$ then $\Phi_{P,\th,1_{H^*(P)}}=(\Phi_P,f_\th,\ka\circ h_{1_{H^*(P)}})$ is a classifying morphism for $(P,\th,1_{H^*(P)})$ in the sense of Proposition~\ref{prop_univ_example}, hence $\pi_{\Map}\circ\phi_{P,\th,1_{H^*(P)}}=j\circ\ka\circ h_{1_{H^*(B)}}$ and $\pi_F\circ\phi_{P,\th,1_{H^*(P)}}=(\Phi_P\ot_G\id_F)\circ S_\th.$ Let $\pi_{H_0}\colon G\t H_0\to H_0$ be the projection. Then
\begin{align*}
	&\Xi_{P,\th}^{0|n}(1_{H^*(P)})=\phi_{P,\th,1_{H^*(P)}}^*(\Xi^{0|n}_\cla)\mathrlap{=\phi_{P,\th,1_{H^*(P)}}^*\bigl(\pi_F^*(\xi_0^G)\cup\pi_{\Map}^*(x)\bigr)}\\
	&=S_\th^*(\Phi_P\ot_G\id_F)^*(\xi_0^G)\cup h_{1_{H^*(B)}}^*\ka^*j^*(x)\\
	&=\xi_0^G(\th)\cup h_{1_{H^*(B)}}^*\ka^*(t^0\backslash\ev^*(\jmath_0)) && \text{by \eqref{eqn_EquivCharClass}}\\
	&=\xi_0^G(\th)\cup h_{1_{H^*(B)}}^*\bigl(t^0\backslash(\id_G\t\ka)^*\ev^*(\jmath_0)\bigr) && \text{nat.~of `$\backslash$'}\\
	&=\xi_0^G(\th)\cup h_{1_{H^*(B)}}^*\bigl(t^0\backslash(1_{H^*(G)}\t\jmath_0)\bigr) && \text{by $\ev\circ(\id_G\t\ka)=\pi_{H_0}$}\\
	&=\xi_0^G(\th)\cup h_{1_{H^*(B)}}^*(\jmath_0)=\xi_0^G(\th)\cup 1_{H^*(B)}=\xi_0^G(\th)&& \text{by $\an{t^0,1_{H^*(G)}}=1.$} \qedhere
\end{align*}
\end{proof}

\subsection{A technical result}
\label{sComputation}

By Theorem~\ref{s2thm1}, stable pushforward operations of degree $k$ are in bijection with classes in $\lim \widetilde H^{m+k}(E_m).$ To compute these groups, the inverse limit groups of the fiber spectra $\Ga^k_F=\lim\widetilde H^{m+k}(F_m)$ will play an important role, where the inverse system was defined in \eqref{inverseSystem}. The direct sum over these groups will be written $\Ga^*_F=\bigoplus_{k\in\Z}\lim\widetilde H^{m+k}(F_m)$ and forms a graded $H^*(F)$-module. The following theorem computes this module.

Since $F_m=(F\t\Map(G,H_m))/(F\t\{*_{\Map}\}),$ this inverse system is closely connected to the inverse system $\widetilde H^{m+k}(\Map(G,H_m))$ with the connecting maps $\mu_m^*\colon\widetilde{H}^{m+1+k}(\Map(G,H_{m+1}))\allowbreak\to\widetilde H^{m+1+k}(\Si\Map(G,H_m))\allowbreak\!\cong\!\widetilde{H}^{m+k}(\Map(G,H_m))$ constructed using $\mu_m$ from Example~\textup{\ref{s2ex2}} and the suspension isomorphism.

Assume that $H_*(G)$ is finite-dimensional in each degree and fix a basis $\{b_i\}_{i\in I}$ of homogeneous elements. Using the tautological classes, define \emph{slant product classes}
\e
\label{DefXI}
 x_i^{(m)}=b_i\backslash\ev^*(\jmath_m)\in \widetilde H^{m-|b_i|}\bigl(\Map(G,H_m)\bigr),
\e
observing that $x_i^{(m)}=0$ if $|b_i|>m.$

\begin{thm}
\label{s2thm2}
\textup{(a)}\hskip\labelsep
\hangindent\leftmargini
For each\/ $m,$ $\widetilde{H}^*(\Map(G,H_m))$ is a free graded-commutative $\Q$-algebra on generators\/ $x_i^{(m)}$ for all\/ $i\in I$ with\/ $|b_i|\leqslant m.$
\begin{enumerate}
\stepcounter{enumi}
\item 
The images of the generators under the connecting maps are
\e
\label{mumdagger}
 \mu_m^*\bigl(x_i^{(m+1)}\bigr)=
 \begin{cases}
 x_i^{(m)}&\text{ if\/ $|b_i|\leqslant m,$}\\
 0&\text{ if\/ $|b_i|=m+1,$}
 \end{cases}
\e
 while\/ $\mu_m^*$ annihilates all monomials of degree\/ $\geqslant2$ in the\/ $x_i^{(m+1)}.$
 \item
The inverse system\/ $(\widetilde{H}^{m+k}(F_m),\varphi_m^*)$ is isomorphic to the inverse system $\Bigl(\bigoplus_{p+q=m+k}H^p(F)\ot\widetilde{H}^q(\Map(G,H_m)),\id_{H^*(F)}\ot\mu_m^*\Bigr).$
\item
The component\/ $\Ga^k_F=\lim\widetilde{H}^{m+k}(F_m)$ of the\/ $H^*(F)$-module\/ $\Ga^*_F$ consists of all formal infinite series\/ $\sum_{i\in I} a_i x_i$ with\/ $a_i\in H^{k+|b_i|}(F)$; the projection to\/ $\widetilde{H}^{m+k}(F_m)\cong \bigoplus_{p+q=m+k}H^p(F)\ot H^q(\Map(G,H_m))$ is\/ $\sum_{i\in I} a_i \ot x_i^{(m)},$ where the sum extends over all\/ $i\in I$ with\/ $|b_i|\leqslant m.$
\end{enumerate}
\end{thm}

\begin{proof}
(a)\hskip\labelsep
Represent each $x_i^{(m)}$ by a pointed map $\phi_i^{(m)}\colon \Map(G,H_m)\to H_{m-|b_i|}.$ Taken together, these determine a map
\e
\label{MappingWE}
  \phi^{(m)}\colon \Map(G,H_m)\longra\prod\nolimits_{i\in I: |b_i|\leqslant m} H_{m-|b_i|}.
\e
We will show that \eqref{MappingWE} is a weak equivalence by showing that the induced map $\pi_k(\phi^{(m)})$ of homotopy groups is an isomorphism for all $k.$ We can identify the domain of $\pi_k(\phi^{(m)})$ with $\pi_k(\Map(G,H_m))\cong[G^+\wedge S^k,H_m]\cong\widetilde{H}^m(G^+\wedge S^k),$ the codomain with $\prod\nolimits_{i\in I: |b_i|\leqslant m}\pi_k(H_{m-|b_i|})\cong\prod\nolimits_{i\in I: |b_i|=m-k}\widetilde{H}^{m-|b_i|}(S^k),$ and $\pi_k(\phi^{(m)})$ corresponds to the map that sends $\al\in\widetilde{H}^m(G^+\wedge S^k)$ to $(b_i\backslash\al)_{i\in I: |b_i|=m-k}.$ Since $\{b_i\}_{i\in I}$ is a homogenous basis of $H_*(G),$ the Universal Coefficient Theorem implies that $H^{m-k}(G)\to\prod\nolimits_{i\in I: |b_i|=m-k}\Q,$ $\be\mapsto(\an{\be,b_i})_{i\in I, |b_i|=m-k}$ is an isomorphism, so the codomain of $\pi_k(\phi^{(m)})$ can be further identified with $H^{m-k}(G)\ot\widetilde{H}^k(S^k)$ and then $\pi_k(\phi^{(m)})$ becomes the K\"unneth map $\widetilde{H}^m(G^+\wedge S^k)\to H^{m-k}(G)\ot\widetilde{H}^k(S^k).$ Since we are taking coefficients in $\Q$ and since $H_*(G)$ is finite-dimensional in each degree, the cohomological K\"unneth Theorem implies that the K\"unneth map is an isomorphism, hence $\pi_k(\phi^{(m)})$ is an isomorphism for all $k,$ as claimed.

The fact that $\widetilde H^*(H_{m-|b_i|})$ is the free graded-commutative $\Q$-algebra on $\jmath_{m-|b_i|}$ and the isomorphism in cohomology induced by \eqref{MappingWE} imply that $\widetilde{H}^*(\Map(G,H_m))$ is the free graded-commutative $\Q$-algebra on $\bigl(\phi^{(m)}_i\bigr)^*(\jmath_{m-|b_i|})=x^{(m)}_i.$
\smallskip

\noindent
(b)
In the notation of Example~\textup{\ref{s2ex2}}, we have $\eta_m\circ(\Si\ev)=\ev\circ(\id_G\t\mu_m)$ and $\eta_m^*(\jmath_{m+1})=\si(\jmath_m).$ Using the naturality of slant products and their compatibility with the suspension isomorphisms, we compute
\[
 \mu_m^*(x_i^{(m+1)})=\mu_m^*(b_i\backslash\ev^*(\jmath_{m+1}))=b_i\backslash(\Si \ev)^*(\si(\jmath_m))=\si(b_i\backslash\ev^*(\jmath_m)).
\]
This right hand side is $\si(x_i^{(m)})$ if $|b_i|\leqslant m$ and vanishes if $|b_i|=m+1$ since $\widetilde H^{m-|b_i|}(\Map(G,H_m))=0$ in this case. This proves \eqref{mumdagger}. Finally, $\mu_m^*$ vanishes on monomials of degree $\geqslant2$ because cup products vanish on suspensions.
\smallskip

\noindent
(c) This follows from the naturality of the K\"unneth isomorphism.
\smallskip

\noindent
(d) By (c), the inverse system $\widetilde H^{m+k}(F_m)$ is equivalent to the inverse system $\bigoplus_{p+q=m+k}H^p(F)\ot\tilde H^q(\Map(G,H_m)),$ which by (a) consists of all polynomials
\[
 \sum_{k\geqslant 1}\sum_{i_1,\ldots,i_k} a_{i_1,\ldots,i_k}\ot x_{i_1}^{(m)}\cdots x_{i_k}^{(m)}
\]
with coefficients $a_{i_1,\ldots,i_k}\in H^{k+|b_{i_1}|+\ldots+|b_{i_k}|}(F),$ where the sum ranges over the finitely many $k\geqslant 0$ and $i_1,\ldots, i_k\in I$ for which $|b_{i_1}|+\ldots+|b_{i_k}|\leqslant m.$ According to (b), the connecting map restricts this expression to the summand with $k=1$ and substitutes all $x_i^{(m)}$ with $|b_i|=m$ by zero. An element in the inverse limit is a collection of such elements that is compatible under the connecting maps. These are simply the infinite sums $\sum_{i\in I}a_ix_i^{(m)}$ where $a_i\in H^{k+|b_i|}(F).$
\end{proof}

\section{Projective Euler operations}
\label{s4}

In this section, we compute the group $\Pi^*_{\F\t\Z}$ of stable pushforward operations for principal $\G$-bundles and orientations in twisted K-theory. The first step in \S\ref{s41} is to place the group $\Pi^*_{\F\t\Z}$ in an exact sequence. This allows us to construct the projective Euler operation and the projective rank operation in \S\ref{sGenSection}, where we prove Theorem~\ref{s1thm1}(a)--(f) and Theorem~\ref{Prop_S_Theta}; Parts (g)--(i) of Theorem~\ref{s1thm1} are proven  later in \S\ref{s43}. In \S\ref{Sec_Module_Structure} we introduce the module structure on $\Pi^*_{\F\t\Z}$ used in the classification problem and Theorem~\ref{thmA} is proven in \S\ref{s42}. The homological version of the projective Euler operation is constructed in \S\ref{Ssec_Homol_PE}.

\begin{ass}
\label{standard-setup}
Specialize the situation of \S\ref{s2} to the following setup of spaces.
\begin{enumerate}
\item
A topological abelian group model $G=\G$ for the classifying space of complex line bundles, where the group operation is a classifying map for the tensor product of line bundles.
\item
A commutative H-space model ${B\U\t\Z}$ for the classifying space of (virtual) complex vector bundles, where the H-space operation is a classifying map for the Whitney direct sum of vector bundles. 
\item
An involution $\smash{\breve{(\;)}}$ on ${B\U\t\Z}$ given by taking the dual vector bundle and preserving the rank. Composing $\smash{\breve{(\;)}}$ with an orientation $\th$ defines the \emph{dual orientation} $\breve\th$.
\item
A topological group action $\ell_{B\U\t\Z}\colon B\U(1)\t({B\U\t\Z})\to{B\U\t\Z}$ given by taking the tensor product of a line bundle with a vector bundle. This action preserves the rank and restricts to an operation $\ell_{\smash{\F\t\{r\}}}$ on $\F\t\{r\}.$
\end{enumerate}
\end{ass}

\subsection{Background on Chern classes and the Chern character}
\label{ssec:backgroundChern}

Recall from \cite[p.199]{M} that $H^*(B\U(r))$ is a free commutative algebra $\Q[c_1,\ldots,c_r]$ on the Chern classes of the universal complex vector bundle $V(r)=(E\U(r)\t\C^r)/\U(r)\to B\U(r).$ Letting $\ul\C$ denote the trivial bundle, the bundle $V(r)\op\ul\C\to B\U(r)$ is classified by a map $i_{r,r+1}\colon B\U(r)\to B\U(r+1).$ Let $B\U=\colim B\U(r)$ be the colimit and write $i_r\colon B\U(r)\to B\U$ for the canonical maps. Since $i_{r,r+1}^*\colon H^*(B\U(r+1))\to H^*(B\U(r))$ maps $c_k\mapsto c_k$ if $k\leqslant r$ and $c_{r+1}\mapsto 0,$ one finds that the Mittag--Leffler condition holds for this inverse system and hence by \cite[p.148]{M} the cohomology $H^*(B\U)$ is isomorphic to $\lim H^*(B\U(r))$, thus a polynomial ring $\Q[c_1,c_2,\ldots]$ in infinitely many generators. Here the generators are characterized by the property that $i_r^*(c_k)=c_k(V(r))$ for all $k\leqslant r$ and $i_r^*(c_k)=0$ for all $k>r.$

To state the effect of direct sums and tensor products on the cohomology, we consider the direct sum algebra $\bigoplus_{r\in\Z} H^*(B\U\t\{r\}),$ which is a free commutative algebra on $c_k^{(r)}\in H^{2k}(B\U\t\{r\})$ for all $k\geqslant 1$ and $r\in\Z.$ Let $\mu_{r_1,r_2}\colon B\U(r_1)\t B\U(r_2)\to B\U(r_1+r_2)$ be a classifying map of the direct sum vector bundle $V(r_1)\t V(r_2)\to B\U(r_1)\t B\U(r_2).$ The Whitney sum formula for Chern classes states $\mu_{r_1,r_2}^*\colon c_k^{(r_1+r_2)}\mapsto\sum_{i+j=k}c_i^{(r_1)}\cup c_j^{(r_2)},$ where we define $\y_0^{(r)}=1$ for all $r.$ Since the behaviour of Chern classes under tensor products is complicated, one introduces a different set of generators, the components of the Chern character $\z_k^{(r)}\in H^{2k}(B\U\t\{r\}).$ As in Hirzebruch~\cite[\S 10.1]{Hir}, the classes $\z_k^{(r)}$ are certain universal polynomials (independent of $r$) in the Chern classes $\y_k^{(r)}.$ Let $\tau_{r_1,r_2}\colon B\U(r_1)\t B\U(r_2)\to B\U(r_1r_2)$ be the classifying map of the external tensor product vector bundle $V(r_1)\boxtimes V(r_2)\to B\U(r_1)\t B\U(r_2).$ According to Hirzebruch~\cite[\S 10.1]{Hir} one has $\mu_{r_1,r_2}^*\colon \z_k^{(r_1+r_2)}\mapsto\z_k^{(r_1)}+\z_k^{(r_2)}$ and $\tau_{r_1,r_2}^*\colon \z_k^{(r_1r_2)}\mapsto\sum_{i+j=k}\z_i^{(r_1)}\cup \z_j^{(r_2)},$ where we define $\z_0^{(r)}=r$ for all $r.$ For $k\geqslant 1$ the generators $\y_k^{(r)}$ and $\z_k^{(r)}$ do not depend explicitly on $r,$ so we will drop the rank from the notation and simply write $\y_k$ for $\y_k^{(r)}$ and $\z_k$ for $\z_k^{(r)}.$

\subsection{Exact classification sequence}
\label{s41}

Let $G=\G$ and $F=\F\t\{r\}.$ Recall from \eqref{diag_Overview_Of_Spaces} the space $F_m=(F\t\Map(G,H_m))/(F\t\{*_{\Map}\})$ and the space $E_m,$ obtained by collapsing the base-point section of the principal $G$-bundle $B_m^\cla=EG\ot_G(F\t\Map(G,H_m))\to EG\ot_GF.$ By Theorem~\ref{s2thm1} and Definition~\ref{Dfn_TwistedK}, $\Pi_{\smash{\F\t\{r\}}}^k=\lim\widetilde{H}^{m+k}(E_m)$ is the group of stable pushforward operations of degree $k$ with orientations in twisted K-theory of rank $r.$ Let $\Pi_{\smash{\F\t\{r\}}}^*=\bigoplus_{k\in\Z}\Pi_{\smash{\F\t\{r\}}}^k$ and $\Pi^*_{\F\t\Z}=\prod_{r\in\Z}\Pi^*_{\F\t\{r\}}.$

We can view $B_m^\cla$ as a bundle over $BG=BB\U(1)$ with fiber $F_m$ and over $\Q$ replace the base by $S^3,$ by the following argument from \cite{AtSe06}. Let $\ga\colon\mathbb{CP}^\iy\to\G$ be the classifying map of the tautological complex line bundle. Define a principal $\G$-bundle over $S^3$ by clutching a pair of trivial bundles on the two hemispheres over the boundary $S^2\cong\mathbb{CP}^1$ using the transition function $\ga|_{\mathbb{CP}^1}\colon\mathbb{CP}^1\to\G.$ Let $\phi_{S^3}\colon S^3\to B\G$ be a classifying map for this principal $\G$-bundle. Depending on how one identifies the hemispheres, there are two possibilities differing by an inversion; our convention is fixed by the sign in \eqref{WangConnecting}. We then have a pullback bundle $\phi_{S^3}^*(B_m^\cla)=\phi_{S^3}^*(EG)\ot_G(F\t\Map(G,H_m))\to\phi_{S^3}^*(EG)\ot_G F$ and a quotient space, written $\phi_{S^3}^*(E_m),$ obtained by collapsing the base-point section to a point. The maps $\ep_m$ from Definition~\ref{Dfn_EmFmSpectra} pull back to connecting maps $\Si\phi_{S^3}^*(E_m)\to \phi_{S^3}^*(E_{m+1})$ which determine an inverse system $\widetilde H^{n+1}(\phi_{S^3}^*(E_{m+1}))\to\widetilde H^{n+1}(\Si \phi_{S^3}^*(E_m))\cong \widetilde H^n(\phi_{S^3}^*(E_m))$ as in \eqref{inverseSystem}. Since $BB\U(1)=K(\Z,3)$ is a rational homotopy sphere and the transition function $\ga|_{\mathbb{CP}^1}$ is a generator of $\pi_2(B\U(1))\cong\pi_3(BB\U(1))=\Z,$ we obtain the following.

\begin{prop}
\label{prop:algebraic-char-pushforward}
We have\/ $\Pi_{\smash{\F\t\{r\}}}^*=\lim\widetilde{H}^{m+*}(E_m)\cong\lim\widetilde{H}^{m+*}(\phi_{S^3}^*(E_m)).$
\end{prop}

%

\begin{dfn}
\label{dfnPartial}
Let $\6$ be the algebra derivation on $\bigoplus_{r\in\Z} H^*(B\U\t\{r\})$ defined on $H^*(\F\t\{r\})$ by $\6(\y_j)=(r-j+1)\y_{j-1}$ on the generators.
\end{dfn}

\begin{lem}
\label{altdfnPartial}
	We have\/ $\6(\z_{j+1})=\z_j$ for all\/ $j\geqslant 0$ and\/ $\6(\z_0)=0.$
\end{lem}

\begin{proof}
	By the splitting principle, $\z_j$ on $\bigoplus_{r\in\Z} H^*(B\U\t\{r\})$ is characterized by its additivity under direct sums and by $\z_j(L)=\frac{1}{j!}c_1(L)^j$ for complex line bundles. As $\6$ is linear, $\6(\z_{j+1})$ is again additive under direct sums. If $L$ is a complex line bundle, then
	\[
\6(\z_{j+1}(L))=\6\left(\frac{c_1(L)^{j+1}}{(j+1)!}\right)=\frac{(j+1)c_1(L)^j\6c_1(L)}{(j+1)!}\overset{r=1}{=}\frac{c_1(L)^j}{j!}=\z_j(L).\qedhere
	\]
\end{proof}


Recall that $H^*(\G)$ is a polynomial ring $\Q[c_1]$ on the first Chern class of the universal complex line bundle. Let $t\in H_2(\G)$ be the homology class dual to $c_1.$ The group operation $\mathrm{m}_\G$ on $\G$ induces a product on  $H_*(\G)$ for which the $i$\textsuperscript{th} divided power $t^i/i!$ becomes dual to $c_1^i.$ The group of stable pushforward operations $\Pi_{\smash{\F\t\{r\}}}^{*}$ can be computed using an exact sequence involving the groups $\Ga_\F^*=\bigoplus_{k\in\Z}\Ga_\F^k.$ By applying Theorem~\ref{s2thm2}(d) to the basis $\{t^i\}_{i\geqslant0}$, we see that $\Ga_\F^k=\lim\widetilde{H}^{m+k}(F_m)$ is isomorphic to the group of formal power series $\sum_{i\geqslant 0} \xi_i x_i$ where $\xi_i\in H^{2i+k}(\F).$ Here the projection of an element $\sum_{i\geqslant 0} \xi_i x_i$ of the inverse limit to $\widetilde{H}^{m+k}(F_m)\cong \bigoplus_{p+q=m+k} H^p(\F)\ot H^q(\Map(\G,H_m))$ is $\sum_{0\leqslant i\leqslant m/2} \xi_i\ot x_i^{(m)},$ where $x_i^{(m)}=t^i\backslash\ev^*(\jmath_m).$


\begin{thm}
\label{s3thm1}
 There is an exact sequence
\begin{equation}
\label{s3eq7}
\begin{tikzcd}
  0\rar & \Pi_{\smash{\F\t\{r\}}}^{2*}\rar & \Ga_\F^{2*}\rar{\de_r} & \Ga_\F^{2*-2}\rar & \Pi_{\smash{\F\t\{r\}}}^{2*+1}\rar & 0,
 \end{tikzcd}
\end{equation}
where\/ $\de_r$ is the unique\/ $\Q$-linear map such that for all\/ $\xi\in H^*(\F)$ and\/ $i\geqslant0,$
 \e
 \label{s3eq4}
  \de_r(\xi x_i)=\6(\xi)x_i+\xi x_{i+1}.
 \e
In particular, we can identify $\Pi_{\smash{\F\t\{r\}}}^{2e}$ with a submodule of $\Ga_\F^{2e}$ and then rational stable pushforward operations of degree\/ $2e$ correspond to formal infinite series\/ $\Xi^{2e}_\cla=\sum_{i\geqslant 0}\xi_i x_i,$ where\/ $\xi_i\in H^{2e+2i}(\F\t\{r\})$ are cohomology classes with
\ea
\label{sequence}
 \6\xi_0&=0, &\6\xi_{i+1}&=-\xi_i\enskip(\forall i\geqslant 0).
\ea
Moreover, odd-degree rational stable pushforward operations correspond to elements in the cokernel of\/ $\de_r.$
\end{thm}

\begin{proof}
For each $m,$ consider the Wang exact sequence of the bundle $\phi_{S^3}^*(E_m)\to S^3$ with fiber $F_m.$ This sequence is obtained from the long exact sequence of the pair $(\phi_{S^3}^*(E_m), F_m)$ by using the identification $H^*(\phi_{S^3}^*(E_m),F_m)\cong\widetilde H^{*-3}(F_m)$ induced by the homeomorphism $\phi_{S^3}^*(E_m)/F_m\simeq\Si^3(F_m).$ Taking the inverse limit of Wang sequences over $m$, we can compare the resulting exact sequence to \eqref{s2eq9} and we obtain the following commutative diagram of exact sequences.
\[
\begin{tikzcd}[column sep=3.8ex]
 \rar{\de_r}&\lim H^*(E_m,F_m)\dar\rar& \begin{array}[b]{l}\Pi_{\smash{\F\t\{r\}}}^*\\[1pt]=\lim\widetilde H^*(E_m)\end{array}\dar{\cong}\rar& \begin{array}[b]{l}\Ga_\F^*\\[1pt]=\lim\widetilde H^*(F_m)\end{array}\dar{\id}\rar{\de_r}& {}
\\
 \rar{\partial}&\begin{array}[t]{l}\lim H^*(\phi_{S^3}^*(E_m),F_m)\\[1pt]\cong\lim\widetilde H^{*-3}(F_m)\mathrlap{=\Ga_\F^{*-3}}\end{array}\rar& \lim\widetilde H^*(\phi_{S^3}^*(E_m))\rar& \lim\widetilde H^*(F_m) \rar{\partial}& {}
\end{tikzcd}
\]
The middle vertical map is an isomorphism by Proposition~\ref{prop:algebraic-char-pushforward}, so the left vertical map is an isomorphism by the five lemma. Since $\Ga_\F^*$ is concentrated in even degrees, the upper sequence reduces to \eqref{s3eq7}, where we identify $H^*(E_m,F_m)\cong\Ga_\F^{*-3}.$ 

According to Atiyah--Segal \cite[(3.8)]{AtSe06}, the differential $\partial$ in the Wang sequence is a derivation and can be computed as follows. As $\phi_{S^3}^*(B_m^\cla)\to S^3$ is obtained by clutching a pair of trivial bundles over the hemispheres along $S^2\cong\mathbb{CP}^1$ using the gluing map
\[
 \varphi\colon\mathbb{CP}^1\t F_m\longra \mathbb{CP}^1\t F_m,
 \quad
 (c,f)\longmapsto (c,\ell_{F_m}(\ga(c),f)),
\]
one has $\varphi^*(1\ot z)=1\ot z+[\mathbb{CP}^1]\ot \partial(z),$ so
\e
\label{WangConnecting}
 \partial(z)=[\mathbb{CP}^1]\backslash(\ga|_{\mathbb{CP}^1}\t\id_{\F\t\Map})^*\ell_{\F\t\Map}^*(\ga^*(z)).
\e
We verify \eqref{s3eq4} on the generators $\pi_{\F}^*(c_j)$ and $\pi_{\Map}^*(x_i^{(m)}).$ Let $\mathrm{m}_\G\colon\G\t\G\to\G,$ $(g,h)\mapsto gh.$ Observe that
\ea
 \ev\circ(\id_{\G}\t\ell_{\Map})&=\ev\circ(\mathrm{m}_\G\t\id_{\Map}),\label{ActProject0}\\
 \pi_{\Map}\circ\ell_{\F\t\Map}&=\ell_{\Map}(\id_{\G}\t\pi_{\Map}),\label{ActProject1}\\
\pi_{\F}\circ\ell_{\F\t\Map}&=\ell_{\F\t\{r\}}(\id_{\G}\t\pi_{\F}),
\label{ActProject2}
\ea
for the projections $\pi_\F$ and $\pi_{\Map}$ of $\F\t\Map(\G,H_m).$ We then compute
\begin{align*}
&\partial(\pi_{\Map}^*(x_i^{(m)}))
\\&=[\mathbb{CP}^1]\backslash(\ga|_{\mathbb{CP}^1}\t\pi_{\Map})^*\ell_{\Map}^*(x_i^{(m)})&&\text{by \eqref{WangConnecting} and \eqref{ActProject1}}\\
&=\pi_{\Map}^*\bigl(t\backslash\ell_{\Map}^*(t^i\backslash\ev^*(\jmath_m))\bigr) && \text{nat.~of `$\backslash$' and $\ga_*([\mathbb{CP}^1])=t$}\\
&=\pi_{\Map}^*\bigl((t\t t^i)\backslash(\id_{\G}\t\ell_{\Map})^*\ev^*(\jmath_m)\bigr) && \text{properties of `$\backslash$'}\\
&=\pi_{\Map}^*\bigl((t\t t^i)\backslash(\mathrm{m}_\G\t\id_{\Map})^*\ev^*(\jmath_m)\bigr) && \text{by \eqref{ActProject0}}\\
&=\pi_{\Map}^*\bigl(t^{i+1}\backslash\ev^*(\jmath_m)\bigr)=x_{i+1}^{(m)} && \text{by $(\mathrm{m}_\G)_*(t\t t^i)=t^{i+1}.$}
\end{align*}
To verify \eqref{s3eq4} on $\pi_{\F}^*(c_j),$ observe that \eqref{cj-Tensor} implies
\e
\label{s3eq6}
\ell_{\smash{\F\t\{r\}}}^*(\y_j)=\sum\nolimits_{k+\ell=j}\binom{r-\ell}{k}c_1^k\t \y_\ell,
\e
so
\begin{align*}
 \partial\bigl(\pi_\F^*(\y_j)\bigr)&=[\mathbb{CP}^1]\backslash(\ga|_{\mathbb{CP}^1}\t\pi_\F)^*\ell_{\smash{\F\t\{r\}}}^*(\y_j)&&\text{by \eqref{WangConnecting} and \eqref{ActProject2}}\\
 &=t\backslash(\id_\G\t\pi_\F)^*\ell_{\smash{\F\t\{r\}}}^*(\y_j)&&\text{by $\ga_*([\mathbb{CP}^1])=t$}\\
 &=\pi_\F^*\bigl(t\backslash\,\ell_{\smash{\F\t\{r\}}}^*(\y_j)\bigr)&&\text{nat.~of `$\backslash$'}\\
 &=(r-j+1)\pi_\F^*(\y_{j-1})&&\text{by \eqref{s3eq6}.}\qedhere
\end{align*}
\end{proof}

Let $\Xi^{*|*+2e}$ be a stable pushforward operation. The class $\Xi^{2e}_\cla=\sum_{i\geqslant 0}\xi_i x_i\in\Pi_{\F\t\{r\}}^{2e}$ satisfies \eqref{sequence}. We call $\xi_0\in H^{2e}(B\U\t\{r\})$ the \emph{constant term} of $\Xi^{2e}_\cla.$

\begin{rem}
\label{rem:similarWang}
As in the proof of Theorem~\ref{s3thm1}, the rational equivalence $\phi\colon S^3\to B\G$ implies that, more generally, all fiber bundles $E\to BB\U(1)$ with fiber $F$ admit long exact Wang sequences $\cdots\to\widetilde H^{*-3}(F)\to \widetilde H^*(E)\to \widetilde H^*(F)\to\cdots.$ If the cohomology of $F$ is concentrated in even degrees, these are particularly useful. For example, the Wang sequence for $E\G\ot_\G(\F\t\{r\})\to B\G$ is
\begin{equation}
\label{AtSeSequence}
\begin{multlined}
  0\longra H^{2*}(E\G\ot_\G(\F\t\{r\}))\longra H^{2*}(\F\t\{r\})\\\overset{\partial}{\longra} H^{2*-2}(\F\t\{r\})
  \longra H^{2*+1}(E\G\ot_\G(\F\t\{r\}))\longra0.
\end{multlined}
\end{equation}
The computation after \eqref{s3eq6} shows that $\partial$ satisfies $\partial(c_j)=(r-j+1)c_{j-1}.$ Since $\xi_0$ is in the kernel of $\6$ by \eqref{sequence}, the exact sequence \eqref{AtSeSequence} shows that $\xi_0$ can be lifted uniquely to a class $\xi_0^\G\in H^{2*}(E\G\ot_\G\F)$ and hence determines a characteristic class for twisted K-theory classes of rank $r$ as in Definition~\ref{dfn_EquivCharClass}. This argument is parallel to \cite[Prop.~8.8]{AtSe06} and \eqref{AtSeSequence} is just \cite[(3.5)]{AtSe06}.

Similarly, the cohomology of $\Map(\G,H_{2m})$ is concentrated in even degrees and so there is a similar Wang sequence
\[
\label{AtSeSequence2}
\begin{multlined}
  0\longra H^{2*}(E\G\ot_\G\Map(\G,H_{2m}))\longra H^{2*}(\Map(\G,H_{2m}))\\\overset{\de}{\longra} H^{2*}(\Map(\G,H_{2m}))
  \longra H^{2*+3}(E\G\ot_\G\Map(\G,H_{2m})).
\end{multlined}
\]
The computation after \eqref{ActProject2} shows that $\de(x_i^{(2m)})=x_{i+1}^{(2m)},$ where we keep in mind that $x_{i+1}^{(2m)}=0$ unless $2(i+1)\leqslant 2m.$ In particular, all $x_m^{(2m)}$ can be lifted to $H^0(E\G\ot_\G\Map(\G,H_{2m})).$
\end{rem}

\begin{dfn}
\label{DiamondAction}
Let $P\to B$ be a principal $\G$-bundle with principal action $\ell_P\colon \G\t P\to P.$ Define an action of $H_*(\G)$ on $H^*(P)$ by $u\mathbin{\diamond}\al=u\backslash\ell_P^*(\al)$ for all $u\in H_*(\G)$ and $\al\in H^*(P).$ Recall that $H_*(\G)=\Q[t]$ is a polynomial ring generated by the homology class $t$ dual to $c_1.$ Since $t\diamond$ is a nilpotent endomorphism of $H^*(P),$ this action extends to an action of $\Q\llbracket t\rrbracket$ on $H^*(P).$
\end{dfn}

Suppose the principal bundle has a section $s\colon B\to P.$ We can then define $\ga_{P,s}\colon P\to G$ by $\ga_{P,s}(gs(b))=g$ for all $b\in B$ and $g\in G,$ which is just the fiber projection of the splitting induced by $s.$ Moreover, the underlying K-theory class $\tilde\th\in K(P)$ of $\th\in K_P(B)$ pulls back to a K-theory class $s^*(\tilde\th)\in K(B).$

\begin{cor}
\label{Cor_EvenDegree_Many_Properties}
Let\/ $\Xi^{2e}_\cla=\sum_{i\geqslant 0}\xi_i x_i$ in\/ $\Pi_{\F\t\{r\}}^{2e}$ be the class of an even-degree stable pushforward operation\/ $\Xi^{*|*+2e}.$ Let\/ $\pi\colon P\to B$ be a principal\/ $\G$-bundle with orientation\/ $\th\in K_P(B)$ of rank\/ $r.$
\begin{enumerate}
\item
Let\/ $P=\G\t B$ be the trivial bundle, where the orientation is given by\/ $\tilde\th=V(1)\boxtimes\vartheta$ for $\vartheta\in K(B)$ \textup(recall the bundle\/ $V(1)$ from \textup{\S\ref{ssec:backgroundChern}}\textup). Let\/ $s(b)=(1_\G,b),$ where\/ $1_\G$ is the unit in\/ $\G.$ Then\/ $\Xi_{P,\th}(c_1^i/i!\t 1_{H^*(B)})$ is the characteristic class\/ $\xi_i\in H^{2i}(\F\t\{r\}).$ More generally, if a principal $\G$-bundle\/ $P$ admits a section\/ $s,$ then
\e
 \label{s3eq8}
 \Xi_{P,\th}\Bigl(\ga_{P,s}^*\Bigl(\frac{c_1^i}{i!}\Bigr)\Bigr)=\xi_i(s^*(\tilde\th)).
\e
In particular, even-degree stable pushforward operations are determined by their values on the trivial bundle with arbitrary orientations\/ $\th.$
\item
We have\/ $\Xi_{P,\th}(\al\mathbin\cup\pi^*(\be))=\Xi_{P,\th}(\al)\mathbin\cup\be$ for all\/ $\al\in H^*(P)$ and\/ $\be\in H^*(B).$
\item
The image\/ $\Xi_{P,\th}(1_{H^*(P)})$ of the unit is the characteristic class\/ $\xi^\G_0(\th).$ Hence the pull-push formula\/ $\bigl(\Xi_{P,\th}\circ \pi^*\bigr)(\be)=\xi_0^\G(\th)\mathbin\cup\be$ holds for all\/ $\be\in H^*(B).$
\item
Push-pull formula:\/ $\bigl(\pi^*\circ\Xi_{P,\th}\bigr)(\al)=\sum_{i\geqslant0}(t^i\diamond\al)\mathbin\cup\xi_i(\tilde\th)$ for\/ $\al\in H^*(P).$
\item
Let\/ $A$ be a topological space and let\/ $\pi_B\colon A\t B\to B$ and\/ $\pi_P\colon A\t P\to P$ be the projections. For the bundle\/ $A\t P\cong\pi_B^*(P)$ and the pullback orientation\/ $A\t\th=\pi_P^*(\th),$ we have\/ $
 \Xi_{A\t P,A\t\th}=\id_{H^*(A)}\t\,\Xi_{P,\th}.$
\end{enumerate}
\end{cor}

\begin{proof}
(a)\hskip\labelsep
We apply Proposition~\ref{s2prop2}(c) to $\al=c_1^i\t 1_{H^*(B)}.$ Let $h_{c_1^i}\colon\G\to H_{2i}$ represent the class $c_1^i\in H^{2i}(\G)$ and define $h_\al=h_{c_1^i}\circ\pi_{\G}$ using the projection $\pi_\G\colon \G\t B\to\G.$ Then $h_\al^\dagger\colon B\to\Map(\G,H_{2i})$ takes the constant value $h_{c_1^i}.$ Using $\ev\circ(\id_{\G}\t h_\al^\dagger)=h_{c_1^i}\circ{\pi_\G},$ we compute:
\begin{align*}
&\Xi_{P,\th}(c_1^i\t 1_{\cramped{H^*(B)}})\\
&=\sum_{j\geqslant0}\!\vartheta^*(\xi_j)\cup(h_\al^\dagger)^*(t^j\backslash\ev^*(\jmath_{2i})) && \text{by \eqref{XmnTrivialFormula} and \eqref{DefXI}}\\
 &=\sum_{j\geqslant0}\vartheta^*(\xi_j)\cup t^j\backslash(h_{c_1^i}\circ\pi_\G)^*(\jmath_{2i}) && \text{nat.~of `$\backslash$'}\\
  &=\sum_{j\geqslant0}\vartheta^*(\xi_j)\cup t^j\backslash(c_1^i\t 1_{H^*(B)})=i!\vartheta^*(\xi_i) && 
   \begin{aligned}
   \text{by }\pi_\G^*(c_1^i)&=c_1^i\t 1_{H^*(B)}\\
   (c_1^i\t 1_{H^*(B)})/t^j&=\de_{i,j}i!1_{H^*(B)}.
   \end{aligned}
\end{align*}
This proves \eqref{s3eq8} for $P$ trivial; the general case follows by naturality, using the isomorphism $(\ga_{P,s},\pi)\colon P\to \G\t B.$ The last claim in (a) is obvious, as $\Xi^{2e}_\cla$ is given by its coefficients $\xi_i$ which are determined entirely by their values on arbitrary K-theory classes and these are equivalently orientations on the trivial bundle.
\smallskip

\noindent
(b)\hskip\labelsep To apply Proposition~\ref{s2prop5}, we must verify that the class of the stable pushforward operation $\Xi_\cla^{2e}=\{\Xi^{m|m+2e}_\cla\}_{m\ge 0}$ satisfies $\rho_{m_1,m_2}^*(\Xi^{m_1+m_2|m_1+m_2+2e}_\cla)=\Xi^{m_1|m_1+2e}_\cla\t\jmath_{m_2}$ in $H^{m_1+m_2+2e}(B_{m_1}^\cla\t H_{m_2}).$ The maps $\rho_{m_1,m_2}\colon B_{m_1}^\cla\t H_{m_2}\to B_{m_1+m_2}^\cla$ and $\ol\rho_{m_1,m_2}\colon\Map(\G,H_{m_1})\t H_{m_2}\to \Map(\G,H_{m_1+m_2})$ from \eqref{rhoPullback} satisfy $\ev\circ\allowbreak(\id_\G\t\ol\rho_{m_1,m_2})=\mathrm{M}_{m_1,m_2}\circ(\ev\t\id_{H_{m_2}}),$ so
 \begin{align*}
 &\ol\rho_{m_1,m_2}^*\bigl(x_i^{(m_1+m_2)}\bigr)=\ol\rho_{m_1,m_2}^*\bigl(t^i\backslash\ev^*(\jmath_{m_1+m_2})\bigr) && \text{by \eqref{DefXI}}\\
&=t^i\backslash(\id_{\G}\t\ol\rho_{m_1,m_2})^*\ev^*(\jmath_{m_1+m_2}) && \text{nat.~of `$\backslash$'}\\
  &=t^i\backslash(\ev\t\id_{H_{m_2}})^*\mathrm{M}_{m_1,m_2}^*(\jmath_{m_1+m_2})\\
  &=t^i\backslash(\ev^*(\jmath_{m_1})\t\jmath_{m_2})=x_i^{(m_1)}\t\jmath_{m_2} && \text{by $\mathrm{M}_{m_1,m_2}^*(\jmath_{m_1+m_2})=\jmath_{m_1}\t\jmath_{m_2}.$}
 \end{align*}
Since $\Xi^{2e}_\cla=\sum_{i\geqslant 0}\xi_i x_i$ we have $\Xi^{m|m+2e}_\cla=\sum_{i\geqslant 0}\xi_i x_i^{(m)}$ for all $m,$ hence
\begin{align*}
 \rho_{m_1,m_2}^*(\Xi^{m_1+m_2|m_1+m_2+2e}_\cla)&=\rho_{m_1,m_2}^*\Bigl(\sum\nolimits_{i\geqslant 0} \xi_i x_i^{(m_1+m_2)}\Bigr)\\
 &=\sum\nolimits_{i\geqslant 0} \xi_i \ol\rho_{m_1,m_2}^*\bigl(x_i^{(m_1+m_2)}\bigr)\\
 &=\sum\nolimits_{i\geqslant 0} \xi_i x_i^{(m_1)}\t\jmath_{m_2}=\Xi^{m_1|m_1+2e}_\cla\t\jmath_{m_2}.
\end{align*}

\noindent
(c)\hskip\labelsep
By projecting the element $\Xi^{2e}_\cla=\sum_{i\geqslant 0}\xi_i x_i$ of the inverse limit to the factor with $m=0,$ we have $\Xi_\cla^{0|2e}=\xi_0\ot x_0^{(0)}.$ According to Remark~\ref{rem:similarWang}, $\xi_0$ lifts uniquely to a class $\xi_0^\F$ in $H^0(E\G\ot_\G(\F\t\{r\}))$ and $x_0^{(0)}$ lifts uniquely to a class $x$ in $H^0(E\G\ot_\G\Map(\G,H_0)),$ hence the hypotheses of Proposition~\ref{PropImageof1} are satisfied. This proves the first part of the claim in (c). Given this, the pull-push formula then follows from (b).
\smallskip

\noindent
(d)\hskip\labelsep
Let $\ell_P\colon \G\t P\to P$ be the principal action. Observe that $\ell_P^*(\al)=\sum_{i\geqslant 0}c_1^i\t\bigl(\frac{t^i}{i!}\diamond\al\bigr).$ The pullback $\pi^*(P)$ has the section $s(p)=(p,p).$ The two coordinate projections $\pi_1$ and $\pi_2$ fit into a pullback square
\[
\begin{tikzcd}
	\pi^*(P)\rar{\pi_1}\dar{\pi_2} & P\dar{\pi}\\
	P\rar{\pi} & B\mathrlap{.}
\end{tikzcd}
\]
Moreover, the pullback orientation $\pi_1^*(\th)$ pulls back under the section $s$ to the class $\tilde\th$ on the base $P.$ Pulling back the formula for $\ell_P^*(\al)$ along $(\ga_{\pi^*(P),s},\pi_2)$ and using $\ell_P\circ(\ga_{\pi^*(P),s},\pi_2)=\pi_1$ implies $\pi_1^*(\al)=\sum_{i\geqslant 0}\pi_2^*\bigl(\frac{t^i}{i!}\diamond\al\bigr)\cup\ga_{\pi^*(P),s}^*(c_1^i),$ hence
\begin{align*}
&\bigl(\pi^*\circ\Xi_{P,\th}\bigr)(\al)\overset{\eqref{s2eq2}}{=}\Xi_{\pi^*(P),\pi^*_1(\th)}(\pi_1^*(\al))\\
 &\overset{\text{(b)}}{=}\sum\nolimits_{i\geqslant 0}\Bigl(\frac{t^i}{i!}\diamond\al\Bigr)\cup\Xi_{\pi^*(P),\pi^*_1(\th)}\bigl(\ga_{\pi^*(P),s}^*(c_1^i)\bigr)
 \overset{\text{(a)}}{=}\sum\nolimits_{i\geqslant 0}(t^i\diamond\al)\cup \xi_i(\tilde\th).
\end{align*}
\smallskip

\noindent
(e)\hskip\labelsep Let $\al\in H^*(A)$ and $\be\in H^*(P).$ By the K\"unneth Theorem, it suffices to verify the claimed formula at $\al\t\be.$ Using (b) and naturality \eqref{s2eq2}, we compute
\begin{align*}
&\Xi_{A\t P,A\t\th}(\al\t\be)=\Xi_{A\t P,A\t\th}\bigl((\id_A\t\pi)^*(\al\t 1_{H^*(B)})\cup\pi_P^*(\be)\bigr)\\
 &=(\al\t 1_{H^*(B)})\cup \Xi_{A\t P,A\t\th}\circ\pi_P^*(\be)\\&=(\al\t 1_{H^*(B)})\cup \pi_B^*\circ\Xi_{P,\th}(\be)=\al\t\Xi_{P,\th}(\be).\qedhere
\end{align*}
\end{proof}

\subsection{\texorpdfstring{Proof of Theorem~\ref{s1thm1}(a)--(f) and Theorem~\ref{Prop_S_Theta}}{Proof Theorem 1.1(a)-(f) and Theorem 1.2}}
\label{sGenSection}

Recall from Theorem~\ref{s3thm1} that rational stable pushforward operations of degree $2e$ correspond to formal infinite series $\Xi_\cla^{2e}=\sum_{i\geqslant 0}\xi_i x_i,$ where $\xi_i \in H^{2e+2i}(\F\t\{r\})$ satisfy \eqref{sequence}.

\begin{dfn}
\label{Dfn_PE}
\hangindent\leftmargini
\textup{(a)}\hskip\labelsep
For each $r\in\Z$ the \emph{projective Euler operation} is the rational stable pushforward operation with class $\Xi_\cla^{\PE}=\sum_{i\geqslant0}\frac{\y_{i+r+1}}{i!} x_i\in\Pi^{2r+2}_{\F\t\{r\}}.$ For a principal $\G$-bundle $\pi\colon P\to B$ with orientation $\th$ of rank $r,$ we write $\pi^\th_!=\Xi^{\PE}_{P,\th}$ for the corresponding pushforward operation.
\begin{enumerate}
\setcounter{enumi}{1}
\item
The \emph{projective rank operation} is the rational stable  pushforward operation of degree $0$ with class $\Xi^{\rk}_\cla=\sum_{i\geqslant0}(-1)^i\z_i x_i.$ We simply write $s_\th^*$ for $\Xi^{\rk}_{P,\th}.$
\end{enumerate}
\end{dfn}

Using Definition~\ref{dfnPartial} and Lemma~\ref{altdfnPartial},
one checks that the coefficients of $\Xi_\cla^{\PE}$ and $\Xi_\cla^{\rk}$ satisfy \eqref{sequence}, so the existence of the rational stable pushforward operations $\pi^\th_!$ and $s_\th^*$ is a consequence of Theorem~\ref{s3thm1}. Parts (a) and (b) of Theorems~\ref{s1thm1}--\ref{Prop_S_Theta} follow from the fact that $\Xi_\cla^{\PE}$ and $\Xi_\cla^{\rk}$ are stable pushforward operations, see Definition~\ref{s2dfn1}(c) and Definition~\ref{s22def0}. Parts (c) of Theorems~\ref{s1thm1}--\ref{Prop_S_Theta} follow from Corollary~\ref{Cor_EvenDegree_Many_Properties}(a), where $\xi_i=\frac{c_{i+r+1}}{i!}$ for the projective Euler operation and $\xi_i=(-1)^i\z_i$ for the projective rank operation. Parts (d) of Theorems~\ref{s1thm1}--\ref{Prop_S_Theta} follow from Corollary~\ref{Cor_EvenDegree_Many_Properties}(c), since $\xi_0=c_{r+1}$ for the projective Euler operation and $\xi_0=\z_0=r$ for the projective rank operation. Parts (e) of Theorems~\ref{s1thm1}--\ref{Prop_S_Theta} follow from Corollary~\ref{Cor_EvenDegree_Many_Properties}(b). Parts (f) of Theorems~\ref{s1thm1}--\ref{Prop_S_Theta} follow from Corollary~\ref{Cor_EvenDegree_Many_Properties}(d).

\subsection{Module structure on group of pushforward operations}
\label{Sec_Module_Structure}

The next step in the classification is to make $\Pi^*_{\F\t\{r\}}$ a module over a large ring. 

\begin{dfn}
\label{dfn37}
Let $P\to B$ be a principal $\G$-bundle with orientation $\th.$ Recall the endomorphism $t\diamond-$ of $H^*(P)$ from Definition~\ref{DiamondAction}. Let $\tilde\th\in K(P)$ be the underlying K-theory class of $\th.$ For each $j,$ multiplication by the Chern character $\z_j(\tilde\th)\in H^{2j}(P)$ defines further endomorphisms of $H^*(P).$

Let $\Xi^{*|*+k}$ be a stable pushforward operation of degree $k.$ We can then construct new operations $\z_j(\Xi)$ of degree $k+2j$ and $t(\Xi)$ of degree $k-2$ by
\ea
 \z_j(\Xi)_{P,\th}(\al)&=\Xi_{P,\th}\bigl(\al\cup \z_j(\tilde\th)\bigr),
&t(\Xi)_{P,\th}(\al)&=\Xi_{P,\th}(t\diamond\al).\label{Rmod}
\ea
\end{dfn}

\begin{prop}
If\/ $\Xi^{*|*+k}$ is a stable pushforward operation, then\/ $\z_j(\Xi)$ and\/ $t(\Xi)$ are again stable pushforward operations.	
\end{prop}

\begin{proof}
We check naturality. As in Definition~\eqref{s2dfn1}(c), let $(\phi,\Phi)\colon P_1\to P_2$ be a morphism of principal $\G$-bundles such that $\th_1=\Phi^*(\th_2)$ for the orientations. Then $\tilde\th_1=\Phi^*(\tilde\th_2)$ for the associated K-theory classes, so $\z_j(\tilde\th_1)=\Phi^*(\z_j(\tilde\th_2))$ and
\begin{align*}
&\z_j(\Xi)_{P_1,\th_1}\bigl(\Phi^*(\al)\bigr)=\Xi_{P_1,\th_1}\bigl(\Phi^*(\al)\cup\z_j(\tilde\th_1)\bigr)=\Xi_{P_1,\th_1}\bigl(\Phi^*(\al\cup\z_j(\tilde\th_2))\bigr)\\
 &\overset{\eqref{s2eq2}}{=}\bigl(\phi^*\circ\Xi_{P_2,\th_2}\bigr)\bigl(\al\cup\z_j(\tilde\th_2)\bigr)
 =\bigl(\phi^*\circ\z_j(\Xi)_{P_2,\th_2}\bigr)(\al).
\end{align*}
Similarly, using $\Phi\circ\ell_{P_1}=\ell_{P_2}\circ(\id_{\G}\t\Phi)$ for the principal actions,
\begin{align*}
 &t(\Xi)_{P_1,\th_1}\bigl(\Phi^*(\al)\bigr)=\Xi_{P_1,\th_1}\bigl(t\backslash(\Phi\t\id_{\G})^*\ell_{P_2}^*(\al)\bigr)
 =\Xi_{P_1,\th_1}\bigl(\Phi^*(t\backslash\ell_{P_2}^*(\al))\bigr)\\
 &\overset{\eqref{s2eq2}}{=}\bigl(\phi^*\circ\Xi_{P_2,\th_2}\bigr)\bigl(t\backslash\ell_{P_2}^*(\al)\bigr)=\bigl(\phi^*\circ t(\Xi)_{P_2,\th_2}\bigr)(\al).
\end{align*}
Moreover, $\z_j(\Xi)(0)=0$ and $t(\Xi)(0)=0$ show that both are pointed pushforward operations. To verify stability for $\z_j(\Xi),$ let $P\to B$ be a principal $\G$-bundle with orientation $\th.$ Let $\varsigma_P,$ $\varsigma_B$ be the suspensions as in \eqref{s2eq7}. Using the projection $\pi_P\colon P\t S^1\to P,$ we compute
\begin{align*}
&\z_j(\Xi)_{P\t S^1,\th\t S^1}\bigl(\varsigma_P(\al)\bigr)=\Xi_{P\t S^1,\th\t S^1}\bigl(\varsigma_P(\al)\cup\pi_P^*(\z_j(\tilde\th))\bigr)\\
 &=\Xi_{P\t S^1,\th\t S^1}\bigl(\varsigma_P(\al\cup\z_j(\tilde\th))\bigr)\overset{\eqref{s2eq7}}{=}\bigl(\varsigma_B\circ\Xi_{P,\th}\bigr)\bigl(\al\cup\z_j(\tilde\th)\bigr)=\bigl(\varsigma_B\circ\z_j(\Xi)_{P,\th}\bigr)(\al).\!\!
\end{align*}
The verification for $t(\Xi)$ is similar, using $t\backslash\ell_{P\t S^1}^*(\varsigma_P(\al))=\varsigma_P(t\backslash\ell_P^*(\al)).$
\end{proof}

\begin{prop}
Let\/ $\Xi^{*|*+2e}$ be an even-degree rational stable pushforward operation with\/ $\Xi^{2e}_\cla=\sum_{i\geqslant 0}\xi_i x_i.$ Setting\/ $\xi_{-1}=0,$ we have
\ea
 \label{zj-sequence}
 \z_j(\Xi_\cla^{2e}) &= \sum_{i\geqslant 0} \left(\sum_{k=0}^j \binom{i+k}{k}\xi_{i+k}\z_{j-k}\right) x_i,\\
 \label{shift-sequence}
 t(\Xi_\cla^{2e}) &= \sum_{i\geqslant 0} \xi_{i-1} x_i.
\ea
\end{prop}

\begin{proof}
By Corollary~\ref{Cor_EvenDegree_Many_Properties}(a) it suffices to verify these assertions for the trivial bundle $P,$ where the orientation is given by a class $\vartheta\in K(B)$ and $\tilde\th=V(1)\boxtimes\vartheta.$ Using $\z_j(\tilde\th)=\sum_{k\geqslant 0}\z_k(V(1))\t\z_{j-k}(\vartheta)=\sum_{k\geqslant 0}\frac{c_1^k}{k!}\t\z_{j-k}(\vartheta),$ we compute
 \begin{align*}
 &\z_j(\Xi)_{P,\th}(c_1^i/i!\t 1_{H^*(B)})\\
 &=\Xi_{P,\th}\bigl((c_1^i/i!\t 1_{H^*(B)})\cup \z_j(\tilde\th)\bigr)\\
 &=\Xi_{P,\th}\Bigl(\sum\nolimits_{k\geqslant 0}\bigl(c_1^i/i!\t 1_{H^*(B)}\bigr)\cup\bigl(c_1^k/k!\t\z_{j-k}\bigr)(\vartheta)\Bigr)\\
  &=\!\sum\nolimits_{k\geqslant0}\frac{1}{i!k!}\Xi_{P,\th}(c_1^{i+k}\t 1_{H^*(B)})\cup\z_{j-k}(\vartheta) && \text{by Cor.~\ref{Cor_EvenDegree_Many_Properties}(b)}\\
 &=\!\sum\nolimits_{k\geqslant0}\frac{(i+k)!}{i!k!}\xi_{i+k}(\vartheta)\cup\z_{j-k}(\vartheta)  && \text{by Cor.~\ref{Cor_EvenDegree_Many_Properties}(a).}
 \end{align*}
 This implies \eqref{zj-sequence}. The proof of \eqref{shift-sequence} is similar, using $t\diamond c_1^i=ic_1^{i-1}.$
\end{proof}

The commutator of \eqref{zj-sequence} with \eqref{shift-sequence} is
\e
\label{commutator}
 \z_j(t(\Xi^{2e}_\cla))-t(\z_j(\Xi^{2e}_\cla)) = \z_{j-1}(\Xi^{2e}_\cla).
\e

\begin{dfn}
\label{PiSModule}
Let $S$ be the graded abelian group $\Q\llbracket t\rrbracket\ot H^*(\F\t\{r\}),$ where $t$ is a power series variable of degree $-2,$ equipped with the product that commutes with infinite series and satisfies the relation $\z_j t-t\z_j=\z_{j-1},$ where $\z_0=r.$


By \eqref{commutator}, the two constructions in Definition~\ref{dfn37} make $\Pi_{\smash{\F\t\{r\}}}^*$ into a graded module over $S,$ where the action is given by the formulas \eqref{Rmod}.
\end{dfn}

\subsection{\texorpdfstring{Proof of Theorem~\ref{thmA}}{Proof of Theorem 1.4}}
\label{s42}

To determine the structure of $\Pi^*_{\F\t\{r\}}$ as a module over $S,$ we first introduce an auxiliary ring $R=\Q[z_0,z_1,z_2,\ldots]$ with a bigrading $|z_k|=(1,k)$ and derivation given by
\begin{align*}
 \6(z_k)&=z_{k-1}\enskip(\forall k\geqslant1),
 &\6(z_0)&=0.
\end{align*}
The ring $R$ and the variables $z_k$ are intended as abstractions of the ring $H^*(\F\t\{r\})$ and the classes $\z_k.$ Note that $\6$ is not a differential, as $\6\circ\6\neq0.$ The derivation property means that $\6(ab)=\6(a)b+a\6(b)$ and $\6(1)=0.$

Define a unital algebra homomorphism $\ep\colon R\to\Q$ by $\ep(z_k)=0$ for $k\geqslant0$ and the $\Q[z_0]$-linear map (since $\6$ is element-wise nilpotent, the infinite sum makes sense)
\[
 \ga=\sum\nolimits_{k\geqslant0}(-1)^kz_k\6^k.
\]
The motivation for introducing $\ga$ is that it encodes the action of $H^*(\G)$ on $\Pi^{2*}_{\smash{\F\t\{r\}}}$ and it is useful for constructing preimages under $\6.$

Clearly $\ep\circ\6=0$ on generators and $\6\circ\ga=0$ is also easy to check. Write $R^{d,e}\subset R$ for the vector subspace of elements of bidegree $(d,e).$ Then $\6$ and $\ga$ have bidegrees $(0,-1)$ and $(1,0).$ We have a chain complex (with $\Q[0,0]$ put in bidegree $(0,0)$)
\begin{equation}
\label{hard-sequence}
 \begin{tikzcd}
  R^{d-1,e}\arrow[r,"\ga"] & R^{d,e}\arrow[r,"\6"] & R^{d,e-1}\arrow[r,"\ep"] & \Q[0,0]^{d,e-1}\arrow[r] & 0.
  \end{tikzcd}
\end{equation}

\begin{lem}
\label{hard-lemma}
 For all\/ $(d,e)\neq(0,0)$ the sequence \eqref{hard-sequence} is exact.
\end{lem}

\begin{proof}
Observe that $R^{0,0}=\Q 1,$ $R^{0,e}=\{0\}$ for $e\neq0,$ $R^{1,e}=\Q z_e,$ and $R^{d,0}=\Q z_0^d.$ If $d=1$ and $e\geqslant 1,$ the exactness of \eqref{hard-sequence} reduces to the statement that $\partial\colon \Q z_e\to \Q z_{e-1}$ is an isomorphism, which holds by definition. The cases $(d,e)\in\{(1,0),(2,0)\}$ and $(d,e)=(0,e)$ for all $e\geqslant1$ are checked by direct inspection. Moreover, when $d<0$ or $e<0$ the sequence vanishes, so we restrict to $d,e\geqslant0.$

We proceed by induction on the total degree $n=d+e.$ We have already noted that the base case $n=2,$ namely $(d,e)\in\{(2,0),(1,1),(0,2)\},$ holds true. For the inductive step, we must prove that \eqref{hard-sequence} is exact for every bidegree $(d,e)$ with $n=d+e>2.$ By induction, \eqref{hard-sequence} is exact for all $d+e<n.$ Since we have already discussed the case $d=0,$ we may also assume $d>0.$ Define a surjective morphism of algebras $\phi\colon R\to R$ by $\phi(z_k)=z_{k-1}$ for $k\geqslant1$ and $\phi(z_0)=0.$ Then $\phi$ maps $R^{d,e}$ to $R^{d,e-d}.$ There is a commutative diagram (notice that we use the map $-\ga\circ\partial,$ not $\ga,$ in the third column)
\[
\begin{tikzcd}[row sep=3ex]
 0\ar[r]&R^{d-2,e}\arrow[r,"z_0"]\arrow[d,"\ga"]&R^{d-1,e}\arrow[d,"\ga"]\arrow[r,"\phi"] & R^{d-1,e-d+1}\arrow[d,"-\ga\circ\6"]\ar[r] & 0\\
 0\ar[r]&R^{d-1,e}\arrow[d,"\6"]\arrow[r,"z_0"]&R^{d,e}\arrow[d,"\6"]\arrow[r,"\phi"] & R^{d,e-d}\arrow[d,"\6"]\ar[r] & 0\\
 0\ar[r]&R^{d-1,e-1}\ar[d]\arrow[r,"z_0"]&R^{d,e-1}\ar[d,"\ep"]\arrow[r,"\phi"] & R^{d,e-d-1}\ar[d,"\ep"]\ar[r] & 0\\
0\ar[r] &0\ar[d]\ar[r]&\Q[0,0]^{d,e-1}\ar[d]\ar[r]&\Q[0,0]^{d,e-d-1}\ar[d]\ar[r]&0\\
&0&0&0&
\end{tikzcd}
\]
whose first three rows are exact because $\Ker\phi$ is the ideal generated by $z_0$ and whose forth row is exact because $(d,e-1)=(0,0)\iff (d,e-d-1)=(0,0).$ This short exact sequence of chain complexes thus induces a long exact sequence
\[
\begin{tikzcd}
\cdots\ar[r]&H_*(\text{Column 1})\ar[r]&H_*(\text{Column 2})\ar[r]&H_*(\text{Column 3})\ar[r]&\cdots.
\end{tikzcd}
\]
We will prove the exactness of Column 2 by showing that Columns 1 and 3 are exact. Column 1 is exact by induction, as $(d-1)+e<n$ and $(d-1,e)\neq(0,0).$ Also, $\Q[0,0]^{d-1,e-1}=\{0\}$ since $(d,e)\neq(1,1).$ For Column 3, first note that
\begin{equation}
\label{RetainExactness}
\begin{tikzcd}[column sep=4ex]
 R^{d-1,e-d}\rar{\ga}&R^{d,e-d}\rar{\6}&R^{d,e-d-1}\rar{\ep}&\Q[0,0]^{d,e-d-1}\rar&0
\end{tikzcd}
\end{equation}
is exact by induction, as $d+(e-d)=n-d<n$ and $(d,e-d)\neq(0,0).$ By induction we also have that $\6\colon R^{d-1,e-d+1}\to R^{d-1,e-d}$ is surjective since $(d-1)+(e-d+1)=n-d<n$ and $(d-1,e-d+1)\neq(0,0)$ and since for $(d,e)\neq(1,1)$ the cokernel is $\Q[0,0]^{d-1,e-d}=\{0\}.$ Therefore we may replace $\ga$ by $-\ga\circ\6$ in \eqref{RetainExactness} and retain exactness, leading to the required exactness of Column 3.
\end{proof}

The additional variable $z_0$ corresponds to the rank $r\in\Z.$ Define a surjective algebra morphism $\psi\colon R\to H^*(\F\t\{r\})$ by $\psi(z_0)=r$ and $\psi(z_k)=\z_k$ for $k\geqslant1.$ If we set $R^e=\bigoplus\nolimits_{d\geqslant0}R^{d,e},$ the map $\psi\colon R^e\to H^{2e}(\F\t\{r\})$ doubles the degree. The kernel of $\psi$ is the graded ideal $\an{z_0-r}.$ Let
\[
 \ga_r=r\cdot\id_{H^*(\F\t\{r\})}+\sum\nolimits_{k\geqslant1}(-1)^k\z_k\6^k.
\]
For $e>0$ we have the following commutative diagram with exact rows.
\begin{center}
\begin{tikzcd}[row sep=3ex]
0\ar[r]&\an{z_0-r}\cap R^e\ar[r]\ar[d,"\ga"]&R^e\ar[r,"\psi"]\ar[d,"\ga"]&H^{2e}(\F\t\{r\})\ar[r]\ar[d,"\ga_r"]&0\\
0\ar[r]&\an{z_0-r}\cap R^e\ar[r]\ar[d,"\6"] & R^e\ar[r,"\psi"]\ar[d,"\6"] & H^{2e}(\F\t\{r\})\ar[d,"\6"]\ar[r]&0\\
0\ar[r]&\an{z_0-r}\cap R^{e-1}\ar[r]\ar[d] & R^{e-1}\ar[r,"\psi"]\ar[d] & H^{2e-2}(\F\t\{r\})\ar[d]\ar[r]&0\\
&0&0&0&
\end{tikzcd}
\end{center}

The middle column is exact by Lemma~\ref{hard-lemma}, provided $e>1.$ Then the left column is also exact for $e>1.$ For example, if $0=\6((z_0-r)\cdot \la)=(z_0-r)\cdot (\6\la)$ for $\la\in R^e,$ then $\6\la=0$ and by the exactness of the middle column we can write $\la=\ga(\xi)$ for some $\xi\in R^e.$ Hence $\ga((z_0-r)\cdot \xi)=(z_0-r)\cdot \la.$ Similarly for the surjectivity of $\6.$ Hence the left column is exact and, as above, the long exact sequence in homology implies the exactness of the right vertical sequence. In case $e=1$ we note that $\6\colon H^2(\F\t\{r\})\to H^0(\F\t\{r\})$ from Definition~\ref{dfnPartial} is surjective for $r\neq 0$ and, if $r=0,$ has image the polynomials in $H^*(\F\t\{0\})$ with zero constant term. The case $e=0$ being trivial, we summarize our discussion as follows.

\begin{lem}
\label{better-lemma}
\hangindent\leftmargini
\textup{(a)}\hskip\labelsep
For\/ $r=0,$ there is an exact sequence
\[
\hskip3.5ex H^{2e}(\F\t\{0\})\overset{\ga_0}{\longra} H^{2e}(\F\t\{0\})\overset{\6}{\longra} H^{2e-2}(\F\t\{0\})\overset{\ep}{\longra}\Q[0]^{2e-2}\longra 0.	
\]
\begin{enumerate}
\setcounter{enumi}{1}
\item
For\/ $r\neq 0,$ there is an exact sequence
\[
H^{2e}(\F\t\{r\})\overset{\ga_r}{\longra} H^{2e}(\F\t\{r\})\\\overset{\6}{\longra} H^{2e-2}(\F\t\{r\})\longra 0.	
\]
\end{enumerate}
\end{lem}

We now compute the groups of stable pushforward operations $\Pi^*_{\smash{\F\t\{r\}}}$ and complete the proof of Theorem~\ref{thmA}. Recall that the groups $\Pi^*_{\smash{\F\t\{r\}}}$ are part of the exact sequence \eqref{s3eq7} and hence determined by the kernel and cokernel of the morphism $\de_r\colon\Ga_\F^{2*+2}\to\Ga_\F^{2*}$ defined in \eqref{s3eq4}, where $\Ga^{2*}_\F$ is the free module over $H^*(\F)$ on generators $x_i$ of degree $-2i$ for all $i\geqslant 0.$

\begin{prop}
\label{Lem_Pi_Odd}
\hangindent\leftmargini
\textup{(a)}\hskip\labelsep
If\/ $r=0,$ the cokernel of\/ $\de_0$ is one-dimensional and spanned by the image of\/ $x_0$ under the projection\/ $\Ga^0_\F\to\Pi^3_{\smash{\F\t\{0\}}}$ in \eqref{s3eq7}. Hence\/ $\Pi^\mathrm{odd}_{\smash{\F\t\{0\}}}=\Q\eta$ for\/ $\eta\in\Pi^3_{\smash{\F\t\{0\}}}.$
\begin{enumerate}
\setcounter{enumi}{1}
\item
If\/ $r\neq 0,$ then\/ $\de_r$ is onto. Therefore,\/ $\Pi^\mathrm{odd}_{\smash{\F\t\{r\}}}=0.$
\end{enumerate}
\end{prop}

\begin{proof}
An element of the codomain $\Ga^{2*}_\F$ can be written $\La=\sum_{i\geqslant0}\la_ix_i$ for $\la_i\in H^*(\F).$ By \eqref{s3eq4}, $\de_r\bigl(\sum_{i\geqslant0}\xi_ix_i\bigr)=\sum_{i\geqslant0}\la_ix_i$ is equivalent to
\ea
 \6\xi_0&=\la_0, &\6\xi_i&=\la_i-\xi_{i-1} \enskip (\forall i\geqslant 1).\label{DifferentialSystem}
\ea

\noindent
(a)\hskip\labelsep
If $r=0,$ then by Lemma~\ref{better-lemma}(a) the image of $\6$ is the set of all polynomials in $\Q[c_1,c_2,\ldots]$ with zero constant term. Hence the restriction of $\6$ to $H^2(\F\t\{0\})\to H^0(\F\t\{0\})$ has a one-dimensional cokernel. Suppose that for $i\geqslant 1$ we have already constructed $\xi_0,\ldots, \xi_{i-1}$ satisfying \eqref{DifferentialSystem}. Solving $\6 \xi_i=\la_i-\xi_{i-1}$ is only possible if $\ep(\la_i-\xi_{i-1})=0,$ but we can change $\xi_{i-1}$ to $\tilde \xi_{i-1}=\xi_{i-1}+\ep(\la_i-\xi_{i-1})$ and ensure $\ep(\la_i-\tilde \xi_{i-1})=0,$ while $\6\tilde \xi_{i-1}=\6 \xi_{i-1}$ shows that $\xi_0,\ldots, \xi_{i-2}, \tilde{\xi}_{i-1}$ still solves \eqref{DifferentialSystem}. Hence we can construct a preimage of $\La$ under $\de_0$ if and only if $\ep(\la_0)=0$ and the cokernel $\Pi^\mathrm{odd}_{\F\t\{0\}}$ of $\de_0$ is spanned by the image of $x_0\in\Ga^0_\F$ in $\Pi^3_{\F\t\{0\}}.$

The exact sequence \eqref{s3eq7} in Theorem~\ref{s3thm1} now implies that $\Pi^\mathrm{odd}_{\smash{\F\t\{0\}}}=\Q\eta,$ where $\eta$ is the image of $x_0$ under the final map in \eqref{s3eq7}.
\smallskip

\noindent
(b)\hskip\labelsep
If $r\neq 0,$ then $\6$ is surjective by Lemma~\ref{better-lemma}(b). Hence we can always solve $\6\xi_0=\la_0$ and, given $\xi_0,\ldots, \xi_{i-1},$ find a preimage $\xi_i$ satisfying \eqref{DifferentialSystem}. This recursion shows that every $\La\in\Ga^{2*}_\F$ has a preimage under $\de_r,$ which is therefore onto.
\end{proof}


Proposition~\ref{Lem_Pi_Odd} completes the computation of $\Pi^\mathrm{odd}_{\F\t\{r\}}.$ To compute the even part, recall from Definition~\ref{Dfn_PE} the projective Euler class $\Xi_\cla^{\PE}=\sum_{i\geqslant0}\frac{\y_{i+r+1}}{i!} x_i$ and the projective rank class $\Xi^{\rk}_\cla=\sum_{i\geqslant0}(-1)^i\z_i x_i.$

\begin{dfn}
For $s\in H^*(\F)\subset S$ let $s(\Xi^{\rk}_\cla)=\sum_{i\geqslant 0}\si_ix_i$ be the class of the pushforward operation obtained by the action of $s,$ where we use the $S$-module structure on $\Pi_{\smash{\F\t\{r\}}}^*$ from Definition~\textup{\ref{PiSModule}}. Write $s(\Xi^{\rk}_\cla)^{(i)}$ for the $i$\textsuperscript{th} term $\si_i.$	
\end{dfn}

\begin{lem}
\label{lem_motivate_ga}
We have\/ $s(\Xi^{\rk}_\cla)^{(i)}=\sum_{k\geqslant 0}(-1)^{i+k}\binom{i+k}{k}\z_{i+k}\6^k(s).$ In particular,\/ $s(\Xi^{\rk}_\cla)^{(0)}=\ga_r(s)$ is the constant term.
\end{lem}

\begin{proof}
It suffices to consider monomials $s=\z_{j_1}\cdots \z_{j_n}.$ Our proof proceeds by induction on $n.$ The base case $n=1$ follows from \eqref{zj-sequence} and $\6^k \z_{j_1}=\z_{j_1-k}.$ For the inductive step, recall that the general Leibniz rule for derivations states that
\e
\label{general-Leibniz}
 \6^m(\z_{j_1}\z_{j_2}\cdots\z_{j_n})=\sum_{k\geqslant0}\binom{m}{k}\6^k(\z_{j_1})\6^{m-k}(\z_{j_2}\cdots\z_{j_n}).
\e
Of course, the binomial coefficients $\binom{m}{k}$ are zero unless $0\leqslant k\leqslant m.$ Using the associativity of the action of $S$ on $\Pi^*_{\F\t\{r\}},$ we compute:
\begin{align*}
	&s(\Xi^{\rk}_\cla)^{(i)}=\bigl(\z_{j_1}(\z_{j_2}\cdots \z_{j_n}(\Xi^{\rk}_\cla))\bigr)^{(i)}\\
	&=\sum_{k\geqslant 0}\binom{i+k}{k} \z_{j_1-k}\cup(\z_{j_2}\cdots \z_{j_n})(\Xi^{\rk}_\cla)^{(i+k)}
	&& \text{by \eqref{zj-sequence}}\\
	&=\sum_{k\geqslant 0}\binom{i+k}{k} \z_{j_1-k}\\
	&\qquad\qquad\cup\sum_{\ell\geqslant 0}(-1)^{i+k+\ell}\binom{i+k+\ell}{\ell}\z_{i+k+\ell}\6^\ell(\z_{j_2}\cdots \z_{j_n}) &&\text{by induction}\\
	&=\sum_{m,k\geqslant 0} (-1)^{i+m}\binom{i+k}{k}\binom{i+m}{m-k}\z_{j_1-k}\z_{i+m}\6^{m-k}(\z_{j_2}\cdots \z_{j_n})
	&&\text{\parbox{2cm}{reindex\\$m=k+\ell$}}\\
	&=\sum_{m,k\geqslant 0} (-1)^{i+m}\binom{i+m}{i}\binom{m}{k}\z_{i+m}\z_{j_1-k}\6^{m-k}(\z_{j_2}\cdots \z_{j_n})
	&&\text{by \eqref{B2A}}
\end{align*}
\begin{align*}	
	&=\sum_{m\geqslant 0} (-1)^{i+m}\binom{i+m}{i}\z_{i+m}\6^m(\z_{j_1}\cdots\z_{j_n}) &&\text{by \eqref{general-Leibniz}.}
\end{align*}
This completes the proof of the inductive step $n-1\mapsto n.$
\end{proof}

\begin{lem}
\label{aux-lem-0}
Let\/ $r\neq 0.$ For all\/ $\Xi^{2e}_\cla\in\Pi^{2e}_{\smash{\F\t\{r\}}}$ there exist\/ $s\in H^{2e}(\F)\subset S$ and\/ $\La^{2e+2}_\cla\in\Pi^{2e+2}_{\smash{\F\t\{r\}}}$ such that\/ $\Xi^{2e}_\cla=s(\Xi^{\rk}_\cla)+t(\La^{2e+2}_\cla).$
\end{lem}

\begin{proof}
Let $\Xi^{2e}_\cla=\sum_{i\geqslant 0}\xi_ix_i.$ Since $\6\xi_0=0$ by \eqref{sequence}, Lemma~\ref{better-lemma}(b) implies that there exists $s\in H^{2e}(\F)$ with $\ga_r(s)=\xi_0.$ Now Lemma~\ref{lem_motivate_ga} implies that $\Xi^{2e}_\cla-s(\Xi^{\rk}_\cla)$ has vanishing constant term and by \eqref{shift-sequence} can therefore be written as $t(\La^{2e+2}_\cla)$ for some $\La^{2e+2}_\cla\in\Pi^{2e+2}_{\F\t\{r\}}$.
\end{proof}

\begin{lem}
\label{aux-lem-1}
 Let\/ $r\in\Z.$ For all\/ $\Xi^{2e}_\cla\in\Pi^{2e}_{\smash{\F\t\{r\}}}$ with\/ $e\neq0$ there exist\/ $j\geqslant0,$ $\Th^{2e-2j}_\cla\in\Pi^{2e-2j}_{\smash{\F\t\{r\}}},$ and\/ $\La^{2e+2}_\cla\in\Pi^{2e+2}_{\smash{\F\t\{r\}}}$ such that\/ $\Xi^{2e}_\cla=\ch_j(\Th^{2e-2j}_\cla)+t(\La^{2e+2}_\cla).$
\end{lem}

\begin{proof}
Let $\Xi^{2e}_\cla=\sum_{i\geqslant 0}\xi_ix_i.$
If $\xi_0=0,$ then $\Xi^{2e}_\cla=t(\La^{2e+2}_\cla)$ for some $\La_\cla^{2e+2}$ and we are done. If $\xi_0\neq0,$ we must have $e>0,$ since $\xi_0\in H^{2e}(\F).$ As $\6 \xi_0=0,$ we may apply Lemma~\ref{better-lemma} to $\xi_0$ and find $\tilde\om\in H^{2e}(\F)$ such that $\xi_0=\ga_r(\tilde\om).$ Since $\partial$ is nilpotent, we can expand this equation as
\[
 \xi_0=\ch_0\tilde\om-\ch_1\6\tilde\om+\ch_2\6^2\tilde\om-\ldots+(-1)^j\ch_j\6^j\tilde\om
\]
for some $j\geqslant0.$ Define $\om_j=(-1)^j\tilde\om,$ $\om_{j-1}=\6\om_j,$ $\ldots,$ $\om_0=\6^j\om_j.$ As $\om_j$ has degree $2e>0,$ we have $\ep(\om_j)=0.$ Hence Lemma~\ref{better-lemma} allows us to pick $\om_{j+1}\in H^{2e+2}(\F)$ with $\6\om_{j+1}=\om_j.$ Continuing in this way, we obtain $\om_{j+2}, \om_{j+3}, \ldots$ such that $\Th^{2e-2j}_\cla=\sum_{i\geqslant0}(-1)^i\om_ix_i\in\Pi^{2e-2j}_{\smash{\F\t\{r\}}}$ is in the kernel of $\de_r.$ Moreover, by \eqref{zj-sequence} the constant term of $\ch_j(\Th^{2e-2j}_\cla)$ is $\om_0\ch_j-\om_1\ch_{j-1}+\ldots+(-1)^j\om_j\ch_0=\xi_0.$ Hence $\Xi^{2e}_\cla-\ch_j(\Th^{2e-2j}_\cla)$ has zero constant term and can thus be written as $t(\La^{2e+2}_\cla).$
\end{proof}

\begin{lem}
\label{aux-lem-2}
Let\/ $r=0.$ For all\/ $\Xi^{2e}_\cla\in\Pi^{2e}_{\smash{\F\t\{0\}}}$ there exist\/ $s\in S$ and\/ $\La^{2e+2}_\cla\in\Pi^{2e+2}_{\smash{\F\t\{0\}}}$ such that\/ $\Xi^{2e}_\cla=s(\Xi^{\PE}_\cla)+t(\La^{2e+2}_\cla).$
\end{lem}

\begin{proof}
Let $\Xi^{2e}_\cla=\sum_{i\geqslant 0}\xi_ix_i.$
If $e<0,$ then $\xi_0=0$ and we can write $\Xi^{2e}_\cla=t(\La^{2e+2}_\cla)$ and put $s=0.$ If $e=0,$ then $\xi_0$ is a constant, but the existence of $\xi_1$ with $-\6\xi_1=\xi_0$ from \eqref{sequence} and Lemma~\ref{better-lemma}(a) implies that the constant must be zero. Hence we again have $\Xi^{2e}_\cla=t(\La^{2e+2}_\cla)$ and put $s=0.$ For $e\geqslant 1$ we proceed by induction on $e.$

In the base case $e=1$ we have $H^2(\F)=\Q c_1,$ so we can write $\xi_0=s c_1$ with $s\in\Q.$ Then $\Xi^{2e}_\cla-s\,\Xi_\cla^{\PE}$ has vanishing constant term, hence $\Xi^{2e}_\cla=s\,\Xi^{\PE}_\cla+t(\La^{2e+2}_\cla)$ for some $\La^{2e+2}_\cla.$
For the inductive step $e-1\mapsto e,$ use Lemma~\ref{aux-lem-1} to write
\e
\label{eq-1-aux-lem-2}
 \Xi^{2e}_\cla=\ch_j(\Th^{2e-2j}_\cla)+t(\La^{2e+2}_\cla)
\e
for some $j\geqslant 0,$ $\Th^{2e-2j}_\cla\in\Pi^{2e-2j}_{\smash{\F\t\{0\}}},$ and $\La^{2e+2}_\cla\in\Pi^{2e+2}_{\smash{\F\t\{0\}}}.$ We will show by reverse induction  on $k=j,j-1,\ldots,0$ that there exist $s_k\in S,$ $\Th_k\in\Pi^{2e-2k}_{\smash{\F\t\{0\}}},$ and $\La_k\in\Pi^{2e+2}_{\smash{\F\t\{0\}}}$ such that
\e
\label{eq-2-aux-lem-2}
 \Xi^{2e}_\cla=s_k(\Xi^{\PE}_\cla)+\ch_k(\Th_k)+t(\La_k).
\e
For $k=0$ we will then have $\ch_k=r=0,$ so \eqref{eq-2-aux-lem-2} completes the inductive step.

As to the induction over $k,$ in the base case $k=j$ we set $s_j=0,$ $\Th_j=\Th^{2e-2j}_\cla,$ and $\La_j=\La^{2e+2}_\cla$ using the classes from \eqref{eq-1-aux-lem-2}. For the inductive step $k\mapsto k-1$ with $k-1\geqslant 0,$ we begin with $s_k, \Th_k, \La_k$ satisfying \eqref{eq-2-aux-lem-2}. Since $\Th_k$ has degree $2e-2k<2e,$ we apply the inductive hypothesis (for the outer induction on $e$) to $\Th_k$ and find $\tilde s\in S$ and $\tilde\La\in \Pi^{2e-2k+2}_{\smash{\F\t\{0\}}}$ such that $\Th_k=\tilde{s}(\Xi^{\PE}_\cla)+t(\tilde\La).$ Hence
\begin{align*}
 \Xi^{2e}_\cla&=s_k(\Xi^{\PE}_\cla)+\ch_k(\tilde{s}(\Xi^{\PE}_\cla)+t(\tilde\La))+t(\La_k) && \text{by \eqref{eq-2-aux-lem-2}}\\
 &=(s_k+\ch_k\tilde s)(\Xi^{\PE}_\cla)+t(\ch_k(\tilde\La)+\La_k)+\ch_{k-1}(\tilde\La) &&\text{by \eqref{commutator}.}
\end{align*}
Setting $s_{k-1}=s_k+\ch_k\tilde s,$ $\Th_{k-1}=\tilde\La,$ and $\La_{k-1}=\ch_k(\tilde\La)+\La_k$ completes the inductive step $k\mapsto k-1$ and also concludes the inductive step $e-1\mapsto e.$
\end{proof}


Proposition~\ref{Lem_Pi_Odd}(a) shows that $\Pi^\mathrm{odd}_{\smash{\F\t\{0\}}}=\Q\eta$ if $r=0$ and Proposition~\ref{Lem_Pi_Odd}(b) shows that $\Pi^\mathrm{odd}_{\smash{\F\t\{0\}}}=\{0\}$ if $r\neq 0.$ To complete the proof of Theorem~\ref{thmA}, it therefore remains only to prove the following.

\begin{prop}
\hangindent\leftmargini
\textup{(a)}\hskip\labelsep
If\/ $r=0,$\/ $\Pi^\mathrm{ev}_{\smash{\F\t\{0\}}}$ is generated as an\/ $S$-module by\/ $\Xi^{\PE}_\cla.$
\begin{enumerate}
\setcounter{enumi}{1}
\item
If\/ $r\neq 0,$\/ $\Pi^\mathrm{ev}_{\smash{\F\t\{r\}}}$ is generated as an\/ $S$-module by\/ $\Xi^{\rk}_\cla.$
\end{enumerate}
\end{prop}

\begin{proof}
The proofs of (a) and (b) are the same, except that one uses the projective Euler class and Lemma~\ref{aux-lem-2} for (a) and the projective rank class and Lemma~\ref{aux-lem-0} for (b). We only spell out the proof of (a). Let $\Xi^{2e}_\cla\in\Pi^{2e}_{\smash{\F\t\{0\}}}$ and write $\Xi^{(0)}=\Xi^{2e}_\cla.$ By applying Lemma~\ref{aux-lem-2} recursively, for all $n\geqslant 0$ there exist $s_n\in S$ and $\Xi^{(n)}\in \Pi^{2e+2n}_{\smash{\F\t\{0\}}}$ such that $\Xi^{(n)}=s_n(\Xi^{\PE}_\cla)+t(\Xi^{(n+1)}).$ Define $s=\sum_{n\geqslant 0} t^n s_n\in S.$ Then $\Xi^{2e}_\cla=s(\Xi^{\PE}_\cla),$ since for every $n$ the first $n$ terms of $\Xi^{2e}_\cla$ and $s(\Xi^{\PE}_\cla)$ agree:
\begin{align*}
 \Xi^{2e}_\cla&=s_0(\Xi^{\PE}_\cla)+t(\Xi^{(1)})=(s_0+ts_1)(\Xi^{\PE}_\cla)+t^2(\Xi^{(2)})\\
  &=\cdots=(s_0+ts_1+\ldots+t^ns_n)(\Xi^{\PE}_\cla)+t^{n+1}(\Xi^{(n+1)}).\qedhere
\end{align*}
\end{proof}

\subsection{\texorpdfstring{Proof of Theorem~\ref{s1thm1}(g)--(i)}{Proof of Theorem 1.1(g)--(i)}}
\label{s43}

We have already checked Theorem~\ref{s1thm1}(a)--(f) in \S\ref{sGenSection} for the projective Euler operation $\Xi_\cla^{\PE}=\sum_{i\geqslant0}\frac{\y_{i+r+1}}{i!} x_i\in\Pi^{2r+2}_{\F\t\{r\}}$ from Definition~\ref{Dfn_PE}. It remains to prove Parts (g)--(i).
\smallskip

\noindent
(g)\hskip\labelsep
By Corollary~\ref{Cor_EvenDegree_Many_Properties}(a), it suffices to treat the case of a trivial principal $\G$-bundle with a global section $s\colon B\to P,$ where the orientation amounts to an ordinary K-theory class $s^*(\tilde\th)\in K(B)$ which, for simplicity, we write in this section without the tilde as $s^*(\th).$ Recall the map $\ga_{P,s}\colon P\to \G$ defined by $\ga_{P,s}(gs(b))=g$ for all $b\in B$ and $g\in \G.$ Recall from Corollary~\ref{Cor_EvenDegree_Many_Properties} and Theorem~\ref{s1thm1}(a) that the projective Euler operation is normalized by
\[
 \pi_!^\th(\ga_{P,s}^*(c_1^i))=c_{i+r+1}(s^*\th).
\]

 For the dual bundle $\breve\pi\colon\breve{P}\to B$ the action is inverted and using the same section $\breve{s}\coloneqq s,$ we have $\ga_{\breve{P},\breve{s}}=u\circ\ga_{P,s}$ for the inversion $u\colon \G\to \G.$ For the dual orientation $\breve\th$ we have $c_{i+r+1}(\breve{s}^*(\breve\th))=(-1)^{i+r+1}c_{i+r+1}(s^*\th).$ Also, $\ga_{\breve{P},\breve{s}}^*(c_1^i)=\ga_{P,s}^*(u^*c_1^i)=(-1)^i\ga_{P,s}^*(c_1^i)$ and therefore
\begin{align*}
 \breve\pi^{\breve\th}_!\bigl(\ga_{P,s}^*(c_1^i)\bigr)&=(-1)^i \breve\pi^{\breve\th}_!\bigl(\ga_{\breve{P},\breve{s}}^*(c_1^i)\bigr)=(-1)^ic_{i+r+1}\bigl(\breve{s}^*(\breve\th)\bigr)\\
 &=(-1)^{r+1} c_{i+r+1}(s^*\th)=(-1)^{r+1}\pi^\th_!(\ga_{P,s}^*(c_1^i)).
\end{align*}

\noindent
(h)\hskip\labelsep
Just as in Corollary~\ref{Cor_EvenDegree_Many_Properties}(a), a rational stable pushforward operation of even degree for principal $(\G\t \G)$-bundles is also uniquely determined by its values on trivial bundles (the key point is that $\G\t\G$ has no odd cohomology). Let $\pi_1\colon P_1\to B$ and $\pi_2\colon P_2\to B$ be principal $\G$-bundles with sections $s_1\colon B\to P_1$ and $s_2\colon B\to P_2.$ These determine further sections
\begin{align*}
B&\overset{s_3}{\longra} P_1\ot_\G P_2=P_3,	&b&\longmapsto s_1(b)\ot_\G s_2(b),\\
P_3&\overset{\si_3}{\longra} P_1\t_B P_2,	&gs_1(b)\ot_\G s_2(b)&\longmapsto(gs_1(b),s_2(b)),\\
P_1&\overset{\si_1}{\longra} \pi_1^*(P_2),	&p_1&\longmapsto(p_1,s_2(\pi_1(p_1))),\\
P_2&\overset{\si_2}{\longra} \pi_2^*(P_1),	&p_2&\longmapsto(s_1(\pi_2(p_2)),p_2),
\end{align*}
which fit into a commutative diagram
\[
 \begin{tikzcd}[column sep=30ex]
  \mathrlap{\pi_2^*(P_1)\cong P_1\t_B P_2\cong\pi_1^*(P_2)}\phantom{\pi_2^*(P_1)}
  \ar[dd,"\ka_2"]
  \ar[rd,start anchor={[xshift=1.2cm,yshift=-0.1cm]},end anchor={[xshift=-1.7cm,yshift=0.25cm]},"\ka_3"]
  \ar[rr,start anchor={[xshift=3.35cm,yshift=0cm]},"\ka_1"]
  &
  & P_1
  \ar[ll,bend left=10,dashed,start anchor={[xshift=-0.25cm,yshift=-0.1cm]}, end anchor={[xshift=3.35cm,yshift=0cm]},shorten >=1ex,"\si_1"]
  \ar[dd,"\pi_1"' near start]\\
  & \hspace{-0.4cm}\mathclap{P_3=P_1\ot_\G P_2}\ar[rd,shorten <=5ex,"\pi_3"]\ar[lu,bend left=15,start anchor={[xshift=-1.75cm,yshift=0.0cm]},end anchor={[xshift=0.85cm,yshift=0cm]},dashed,"\si_3"] &\\
  P_2\ar[rr,"\pi_2"']\ar[uu,dashed,bend left=30,shorten <=0.5ex,"\si_2"] && B\mathrlap{.}\ar[ll,bend right=8,shorten <=1ex,dashed,"s_2"']\ar[uu,bend right,dashed,"s_1"']\ar[lu,dashed,bend right=15,end anchor={[xshift=1.25cm]},"s_3"']
 \end{tikzcd}
\]
Define also $\ga_{(s_1,s_2)}\colon P_1\t_B P_2\to\G\t\G,$ $(g_1s_1,g_2s_2)\mapsto(g_1,g_2).$ Then $H^*(P_1\t_B P_2)$ is generated as an $H^*(B)$-module by $\ga_{(s_1,s_2)}^*(c_1^i\t c_1^j)$ and it suffices to check \eqref{eqn:composition} on these classes. We calculate the first term in \eqref{eqn:composition}. Using $\ga_{(s_1,s_2)}={(\ga_{\pi_2^*(P_1),\si_2},\ga_{P_2,s_2}\circ\ka_2)},$ we have $\ga_{(s_1,s_2)}^*(c_1^i\t c_1^j)=\ga_{\pi_2^*(P_1),\si_2}^*(c_1^i)\cup\ka_2^*(\ga_{P_2,s_2}^*(c_1^j)),$ so by using linearity over $\ka_2^*,$ which holds by Theorem~\ref{s1thm1}(e), and the normalization of the projective Euler operation in Theorem~\ref{s1thm1}(c), we get
\begingroup
\let\oldcup\cup
\renewcommand\cup{}
\e
\label{cik13}
(\ka_2)^{\ka_1^*\th_1+\ka_3^*\th_3}_!\bigl(\ga_{(s_1,s_2)}^*(c_1^i\t c_1^j)\bigr)=\ga_{P_2,s_2}^*(c_1^j)\cup
 c_{i+r_1+r_3+1}\bigl(\si_2^*\ka_1^*\th_1+\si_2^*\ka_3^*\th_3\bigr).
\e
(Here and in the following calculations, we omit the cup product symbol.)
Since
\begin{align*}
 &c_{i+r_1+r_3+1}\bigl(\si_2^*\ka_1^*\th_1+\si_2^*\ka_3^*\th_3\bigr)=\enskip\sum_
  {\scriptstyle\mathclap{\substack{a+q=\\i+r_1+r_3+1}}}\enskip c_a(\si_2^*\ka_1^*\th_1)\cup c_q(\si_2^*\ka_3^*\th_3)&&\text{\parbox[t]{3cm}{by Whitney sum\\\phantom{by }formula}}\\
 &=\enskip\sum_{\scriptstyle\mathclap{\substack{a+q=\\i+r_1+r_3+1}}}\enskip \pi_2^*c_a(s_1^*\th_1)\oldcup(\ga_{P_2,s_2},s_3\pi_2)^*\ell_{P_3}^*c_q(\th_3)&&
 \begin{aligned}
\text{by }&\ka_1 \si_2=s_1\pi_2\\
&\hskip-10ex\ka_3\si_2=\ell_{P_3}(\ga_{P_2,s_2},s_3\pi_2)
 \end{aligned}\\
 &
 =\enskip\sum_{\scriptstyle\mathclap{\substack{a+q=\\i+r_1+r_3+1}}}\enskip \pi_2^*c_a(s_1^*\th_1)
 \cup\sum_{b+c=q}\binom{r_3-c}{b}\ga_{P_2,s_2}^*(c_1^b)\cup\pi_2^*c_c(s_3^*\th_3)
 &&\text{by \eqref{PullbackUnderlyingK},}
\end{align*}
we can rewrite \eqref{cik13} as $\sum_{\scriptstyle{\substack{a+b+c=\\i+r_1+r_3+1}}}\binom{r_3-c}{b}\ga_{P_2,s_2}^*(c_1^{j+b})\cup\pi_2^*c_a(s_1^*\th_1)\cup\pi_2^*c_c(s_3^*\th_3).$
Now apply $(\pi_2)^{\th_2}_!$ and use linearity over $\pi_2^*$ and Corollary~\ref{Cor_EvenDegree_Many_Properties}(a) to get
\begin{align}
\label{erster-streich}
    &(\pi_2)^{\th_2}_!(\ka_2)^{\ka_1^*\th_1+\ka_3^*\th_3}_!\bigl(\ga_{(s_1,s_2)}^*(c_1^i\t c_1^j)\bigr)\\[1ex]
    & \hskip-10ex\mathmakebox[\displaywidth][r]{\begin{aligned}
&=\quad\sum_{\scriptstyle\mathclap{\substack{a+b+c=\\i+r_1+r_3+1}}}\quad\binom{r_3-c}{b}c_a(s_1^*\th_1)\cup c_{j+b+r_2+1}(s_2^*\th_2)\cup c_c(s_3^*\th_3)\\
 &=\quad\sum_{\scriptstyle\mathclap{\substack{a+b+c=\\i+j+r_1+r_2+r_3+2}}}\quad\binom{r_3-c}{b-j-r_2-1}c_a(s_1^*\th_1)\cup c_b(s_2^*\th_2)\cup c_c(s_3^*\th_3),
\end{aligned}}\notag
\end{align}
where in the last step we subtract $j+r_2+1$ from the index $b.$
Symmetrically, the substitution $(1,2,i,j,a,b)\mapsto(2,1,j,i,b,a)$ yields
\begin{align}
\label{zweiter-streich}
    &(\pi_1)^{\th_1}_!(\ka_1)^{\ka_2^*\th_2+\ka_3^*\th_3}_!\bigl(\ga_{(s_1,s_2)}^*(c_1^i\t c_1^j)\bigr)\\[1ex]
    &\mathmakebox[\displaywidth][r]{\begin{aligned}
&=\quad\sum_{\scriptstyle\mathclap{\substack{a+b+c=\\i+j+r_1+r_2+r_3+2}}}\quad\binom{r_3-c}{a-i-r_1-1}c_a(s_1^*\th_1)\cup c_b(s_2^*\th_2)\cup c_c(s_3^*\th_3),
\end{aligned}}\notag
\end{align}
where on the left hand side we have $c_1^i\t c_1^j$ rather than $c_1^j\t c_1^i$ since exchanging $(1,2)\mapsto(2,1)$ leads to a pushforward operation defined on $H^*(P_2\t_B P_1),$ which differs from $(\pi_1)^{\th_1}_!(\ka_1)^{\ka_2^*\th_2+\ka_3^*\th_3}_!$ by precomposition with the swap of factors.

To compute the middle route through \eqref{composition}, it is useful to define $\widehat{\mathrm{m}}_\G\colon \G\t \G\to \G,$ $\widehat{\mathrm{m}}_\G(g,h)=gh^{-1}.$
By the convention for the action $\ell_{P_1\t_B P_2}$ from Page~\pageref{ConventionFPaction}, we have $\ga_{(s_1,s_2)}=\bigl(\widehat{\mathrm{m}}_\G\circ(\ga_{P_3,s_3}\circ\ka_3,\ga_{P_1\t_B P_2,\si_3}),\ga_{P_1\t_B P_2,\si_3}\bigr).$ Combined with $\widehat{\mathrm{m}}_\G^*(c_1^i)=(c_1\t1-1\t c_1)^i,$ the Binomial Theorem, and \eqref{B0}, we then find 
\begin{align*}
 \ga_{(s_1,s_2)}^*(c_1^i\t c_1^j)&=\sum_{m=0}^i\;(-1)^{i-m}\binom{i}{m}\ka_3^*\ga_{P_3,s_3}^*(c_1^m)\cup\ga_{P_1\t_B P_2,\si_3}^*(c_1^{i-m})\cup\ga_{P_1\t_B P_2,\si_3}^*(c_1^j).
\end{align*}
Set $n=i+j-m.$ Using linearity over $\ka_3^*,$ the image under $(\ka_3)^{\ka_1^*\breve\th_1+\ka_2^*\th_2}_!$ is
\[
 \sum_{\mathclap{n+m=i+j}}\quad(-1)^{i-m}\binom{i}{m} \ga_{P_3,s_3}^*(c_1^m)\cup c_{n+r_1+r_2+1}(\si_3^*\ka_1^*\breve\th_1+\si_3^*\ka_2^*\th_2).
\]
Inserting into this the calculation (where $\breve s_3$ denotes $s_3$ regarded as a section of the dual bundle $\breve P_3\to B$)
\begin{align*}
 &c_{n+r_1+r_2+1}(\si_3^*\ka_1^*\breve\th_1+\si_3^*\ka_2^*\th_2)\\[2ex]
 &=\quad\sum_{\scriptstyle\mathclap{\substack{p+b=\\n+r_1+r_2+1}}}\quad
c_p(\si_3^*\ka_1^*\breve\th_1)\cup c_b(\si_3^*\ka_2^*\th_2) && \hskip-20ex\text{by Whitney sum formula}\\[-2ex]
 &=\quad\sum_{\scriptstyle\mathclap{\substack{p+b=\\n+r_1+r_2+1}}}\quad (\ga_{\breve P_3,\breve s_3},s_1\pi_3)^*\ell_{\breve P_1}^*c_p(\breve\th_1)\cup\pi_3^*c_b(s_2^*\th_2)
 &&
 \hskip-20ex\begin{aligned}
 \text{by }&\ka_2\si_3=s_2\pi_3,\\
 &\ka_1\si_3=\ell_{\breve P_1}(\ga_{\breve P_3,\breve s_3},s_1\pi_3)
 \end{aligned}\\
 &=\quad\sum_{\scriptstyle\mathclap{\substack{p+b=\\n+r_1+r_2+1}}}\quad\enskip
\sum_{\mathclap{k+a=p}}\binom{r_1-a}{k}(-1)^k\ga_{P_3,s_3}^*(c_1^k)\cup\pi_3^*c_a(s_1^*\breve\th_1)\cup\pi_3^*c_b(s_2^*\th_2),
 \end{align*}
where at the last step we used \eqref{PullbackUnderlyingK}, $\ga_{\breve P_3,\breve s_3}=u\circ\ga_{P_3,s_3},$ and $u^*(c_1^k)=(-1)^kc_1^k$, we obtain
\[
 \sum_{\scriptstyle\mathclap{\substack{a+b+m+k=\\i+j+r_1+r_2+1}}}\quad(-1)^{i-m+k}\binom{i}{m}\binom{r_1-a}{k}\ga_{P_3,s_3}^*(c_1^{m+k})\cup\pi_3^*c_a(s_1^*\breve\th_1)\cup\pi_3^*c_b(s_2^*\th_2).
\]
Now reindex $m$ by $q=m+k,$ use the formula $\sum_k\binom{i}{q-k}\binom{n}{k}=\binom{n+i}{q}$ for binomial coefficients to eliminate the sum over $k,$ and use $c_a(s_1^*\breve\th_1)=(-1)^ac_a(s_1^*\th_1)$ to get
\[
 \sum_{\scriptstyle\mathclap{\substack{a+b+q=\\i+j+r_1+r_2+1}}}\quad (-1)^{i-q-a} \binom{r_1-a+i}{q}\pi_3^*c_a(s_1^*\th_1)\cup\pi_3^*c_b(s_2^*\th_2)\cup\ga_{P_3,s_3}^*(c_1^q).
\]
Applying the map $(\pi_3)^{\th_3}_!,$ we then find
\begin{align}
\label{dritter-streich}
    &(\pi_3)^{\th_3}_!(\ka_3)^{\ka_1^*\breve\th_1+\ka_2^*\th_2}_!\bigl(\ga_{(s_1,s_2)}^*(c_1^i\t c_1^j)\bigr)\\[1ex]
    & \hskip-5ex\mathmakebox[\displaywidth][r]{\begin{aligned}
 &=\quad\sum_{\scriptstyle\mathclap{\substack{a+b+q=\\i+j+r_1+r_2+1}}}\quad (-1)^{i-q-a}\binom{r_1-a+i}{q}c_a(s_1^*\th_1)\cup c_b(s_2^*\th_2)\cup c_{q+r_3+1}(s_3^*\th_3)\\
 &=\quad\sum_{\scriptstyle\mathclap{\substack{a+b+c=\\i+j+r_1+r_2+r_3+2}}}\quad(-1)^{b-j-r_1-r_2-1}\binom{r_1-a+i}{c-r_3-1}c_a(s_1^*\th_1)\cup c_b(s_2^*\th_2)\cup c_c(s_3^*\th_3),
\end{aligned}}\notag
\end{align}
where in the last step we have reindexed by $c=q+r_3+1$ and used the equation for the summation to rewrite the sign. The composition property \eqref{eqn:composition} now follows by comparing \eqref{erster-streich}, the negative of \eqref{zweiter-streich}, and \eqref{dritter-streich}, while noticing that
\[
 \binom{r_3-c}{b-j-r_2-1}
 -\binom{r_3-c}{a-i-r_1-1}
 =(-1)^{b-j-r_2-1}\binom{r_1-a+i}{c-r_3-1}
\]
for each term $(a,b,c)$ with $a+b+c=i+j+r_1+r_2+r_3+2$ in the sum. Indeed, putting $n=r_3-c$ and $k=b-j-r_2-1$ we can rewrite this as the identity $\binom{n}{k}-\binom{n}{n-k}=(-1)^k\binom{k-n-1}{-n-1},$ which is proven in the appendix as \eqref{B4}. This completes the proof of the composition property Theorem~\ref{s1thm1}(h).
\smallskip
\endgroup

%
%
%

\noindent
(i)\hskip\labelsep If $r\neq 0,$ $\eta_P=0$ and the claim is obvious. If $r=0,$ this follows from Theorem~\ref{thmA}, as $\al\mapsto \pi^\th_!(\al)\cup\eta_P$ is a stable pushforward operation in $\Pi_{\smash{\F\t\{0\}}}^5=\{0\}.$\qed

\subsection{Projective Euler operation in homology}
\label{Ssec_Homol_PE}

\begin{thm}
\label{thmHomologEuler}
For each principal\/ $\G$-bundle\/ ${\pi\colon P\to B}$ and orientation\/ $\th\in K_P(B)$ in twisted complex K-theory of rank\/ $r,$ there is a \textbf{projective Euler operation}\/ $\pi_\th^!\colon H_*(B;\Q)\to H_{*-2r-2}(P;\Q),$ uniquely characterized by\/ $\an{\al,\pi_\th^!(v)}=\an{\pi^\th_!(\al),v}$ for all\/ $\al\in H^{*-2r-2}(P;\Q)$ and\/ $v\in H_*(B;\Q).$

The preduals of the properties stated in Theorem~\textup{\ref{s1thm1}} hold for\/ $\pi^!_\th.$ In particular, for every pullback diagram \eqref{pullback-diagram} and\/ $\th_1=\Phi^*(\th_2),$ we have a commutative square
\begin{equation}
\label{HomolNaturality}
\begin{tikzcd}
H_{*-2r-2}(P_1;\Q)\arrow[r,"\Phi_*"] & H_{*-2r-2}(P_2;\Q)\\
H_*(B_1;\Q)\arrow[u,"(\pi_1)_{\th_1}^!"]\arrow[r,"\phi_*"'] & H_*(B_2;\Q).\arrow[u,"(\pi_2)_{\th_2}^!"]
\end{tikzcd}	
\end{equation}
Moreover, the homological version of Theorem~\textup{\ref{s1thm1}(g)} states that
\e
\label{HomolDualAxiom}
 \breve\pi^!_{\breve\th}=(-1)^{r+1}\pi^!_\th.
\e
\item
Moreover, in the situation of Theorem~\textup{\ref{s1thm1}\ref{thmB-f}}, we have
\e
\label{HomolComposition}
 (\ka_2)^!_{\ka_1^*\th_1+\ka_3^*\th_3}\circ(\pi_2)^!_{\th_2}-(\ka_1)^!_{\ka_2^*\th_2+\ka_3^*\th_3}\circ(\pi_1)_{\th_1}^!=(-1)^{r_1}(\ka_3)^!_{\ka_1^*\breve\th_1+\ka_2^*\th_2}\circ(\pi_3)^!_{\th_3}.
\e
\end{thm}

\begin{proof}
Since we take coefficients in $\Q,$ the Universal Coefficient Theorem implies that the cohomology groups are dual to the homology groups, but to avoid any dimensional constraints on the homology groups, we will identify the homological pushforward operations as a predual (rather than a dual) of the cohomological projective Euler operation. Once preduality is established, the claimed properties follow automatically. To prove the existence of a predual of the projective Euler operation, we apply Lemma~\ref{lemExPredual} with $V=H_n(B;\Q),$ $W=H_{n-2r-2}(P;\Q),$ and $\phi=\pi_!^\th$ for a fixed degree $n\in\Z.$ The Universal Coefficient Theorem \cite[3.6.6]{Wei} then identifies $H^n(B;\Q)$ with $V^*$ and $H^{n-2r-2}(P;\Q)$ with $W^*.$ To verify the hypothesis \eqref{ExPredual} of Lemma~\ref{lemExPredual}, let $v\in H_n(B;\Q).$ As singular chains have compact support, we can pick a finite CW-complex $\bar B,$ map $i\colon \bar B\to B,$ and $\bar v\in H_n(\bar B;\Q)$ such that $i_*(\bar v)=v.$ Define $\bar\pi\colon\bar P=i^*(P)\to\bar B$ and $\bar\th=i^*(\th).$ Let $I\colon\bar P\to P$ be the canonical map. By the Serre spectral sequence and the degree-wise finite-dimensionality of the rational cohomology of $\G,$ $H_{n-2r-2}(\bar P;\Q)$ is finite-dimensional. Let $W_v=I_*(H_{n-2r-2}(\bar P;\Q)).$ Let $\be\in H^{n-2r-2}(P;\Q)$ be such that $\be|_{W_v}=0.$ Then $I^*(\be)=0$ and using naturality we find
\[
 \an{\phi(\be),v}=\an{\pi_!^\th(\be),i_*(\bar v)}=\an{i^*\pi_!^\th(\be),\bar{v}}\overset{\eqref{s1eqn3}}{=}\an{\bar\pi_!^{\bar\th}I^*(\be),\bar v}=0,
\]
so \eqref{ExPredual} holds.
\end{proof}

\section{Homological Lie brackets on moduli spaces}
\label{s5}

Before listing the additional data needed for the construction of a Lie bracket, we fix some notation. For sets $X_{\ep_1},\ldots, X_{\ep_p}$ and $\ep_{i_1},\ldots,\ep_{i_q}\in\{\ep_1,\ldots,\ep_p\},$ write $X_{\ep_{i_1},\ldots,\ep_{i_q}}=X_{\ep_{i_1}}\t\cdots\t X_{\ep_{i_q}}$ for the Cartesian product. Define also a map
\begin{align*}
 \tau^{\ep_1,\ldots,\ep_p}_{\ep_{i_1},\ldots,\ep_{i_q}}\colon X_{\ep_1,\ldots,\ep_p}&\longra X_{\ep_{i_1},\ldots,\ep_{i_q}},\\(m_1,\ldots,m_p)&\longmapsto(m_{i_1},\ldots,m_{i_q}),
\end{align*}
which projects onto and exchanges factors as indicated.

\begin{ass}
\label{ass:lie}
Let $\F$ and $\G$ be as in Assumption~\ref{standard-setup}. We abbreviate $G=\G.$ Let $M$ be a topological space and assume the following.
\begin{enumerate}
\item
${\Phi\colon M\t M\to M}$ is an operation that is associative and commutative up to homotopy. The set $\pi_0(M)$ of path-components of $M$ is then a commutative monoid. Write $M_\al\subset M$ for the path-component corresponding to $\al\in\pi_0(M)$ (so $\al=M_\al$). If we set $\Phi_{\al,\be}=\Phi\vert_{M_{\al,\be}},$ we can restate our assumption as
\[
\Phi_{\al+\be,\ga}\circ(\Phi_{\al,\be}\t\id_{M_\ga})\simeq \Phi_{\al,\be+\ga}\circ(\id_{M_\al}\t\Phi_{\be,\ga}),\quad \Phi_{\al,\be}\simeq \Phi_{\be,\al}\circ\tau^{\al,\be}_{\be,\al}.
\]
\item
${\Psi\colon G\t M\to M}$ is a group action such that the quotient map ${M\to M/G}$ is a principal $G$-bundle. We require $\Psi(g,\Phi(m_1,m_2))=\Phi(\Psi(g,m_1),\Psi(g,m_2))$ for all $m_1,m_2\in M$ and $g\in G,$ so the action distributes over $\Phi.$ Denote the quotient space by $B_\al=M_\al/G.$
\item
Let $\De^n G$ be the diagonal in $G^n.$ Let $P_{\al,\be,\ga}=M_{\al,\be,\ga}/\De^3G$ and $R_{\al,\be}=M_{\al,\be}/\De^2 G$ be the quotient spaces by the diagonal actions. Let $Q_{\al,\be,\ga}=R_{\al,\be}\t B_\ga=M_{\al,\be}/\De^2 G\t B_\ga.$ Equip these spaces with the $G$-actions
\begin{align*}
\ell_{P_{\al,\be,\ga}}\bigl(g,\De^3 G(m_\al,m_\be,m_\ga)\bigr)&=\De^3 G(gm_\al,gm_\be,m_\ga),\\
\ell_{Q_{\al,\be,\ga}}\bigl(g,(\De^2 G(m_\al,m_\be),b_\ga)\bigr)&=(\De^2 G(gm_\al,m_\be),b_\ga),\\
\ell_{R_{\al,\be}}\bigl(g,\De^2 G(m_\al,m_\be)\bigr)&=\De^2 G(gm_\al,m_\be).
\end{align*}
The quotient maps of these actions are denoted by
\begin{align*}
	\ka_{\al,\be,\ga}\colon P_{\al,\be,\ga}&\longra Q_{\al,\be,\ga},&\De^3 G(m_\al,m_\be,m_\ga)&\longmapsto(\De^2 G(m_\al,m_\be),Gm_\ga),\\
	\pi_{\al,\be}\t\id_{B_\ga}\colon Q_{\al,\be,\ga}&\longra B_{\al,\be,\ga},&(\De^2 G(m_\al,m_\be),b_\ga)&\longmapsto(Gm_\al, Gm_\be, b_\ga),\\
	\pi_{\al,\be}\colon R_{\al,\be}&\longra B_{\al,\be},&\De^2G(m_\al,m_\be)&\longmapsto(Gm_\al,Gm_\be),
\end{align*}
and are projections of principal $G$-bundles.
\item
For all $\al,\be\in\pi_0(M)$ there are orientations $\th_{\al,\be}\in K_{R_{\al,\be}}(B_{\al,\be})$ such that
\e
(\tau^{\al,\be}_{\be,\al})^*(\th_{\be,\al})=\breve\th_{\al,\be},\label{th-dual}
\e
where we use that $(\tau^{\al,\be}_{\be,\al})^*(R_{\be,\al})\cong R_{\al,\be}.$

The map $\Phi_{\al,\be}\t\id_{M_\ga}\colon M_{\al,\be,\ga}\to M_{\al+\be,\ga}$ determines an isomorphism $P_{\al,\be,\ga}\cong(\Phi_{\al,\be}\t\id_{B_\ga})^*(R_{\al+\be,\ga}).$ Moreover, the projections define isomorphisms $P_{\al,\be,\ga}\cong(\pi_{\al,\be}\t\id_{B_\ga})^*(\tau_{\al,\ga}^{\al,\be,\ga})^*(R_{\al,\ga})$ and $P_{\al,\be,\ga}\cong(\pi_{\al,\be}\t\id_{B_\ga})^*(\tau_{\be,\ga}^{\al,\be,\ga})^*(R_{\be,\ga}).$ Using these isomorphism, we require
\begin{multline}
\bigl(\Phi_{\al,\be}\t\id_{M_\ga}\bigr)^*(\th_{\al+\be,\ga})\\
=\bigl(\pi_{\al,\be}\t\id_{B_\ga}\bigr)^*\bigl(\tau^{\al,\be,\ga}_{\al,\ga}\bigl)^*(\th_{\al,\ga})+\bigl(\pi_{\al,\be}\t\id_{B_\ga}\bigr)^*\bigl(\tau^{\al,\be,\ga}_{\be,\ga}\bigr)^*(\th_{\be,\ga}).
\label{th-bilinear}
\end{multline}
\item
There are signs $\ep_{\al,\be}\in\{\pm1\}$ for all $\al,\be\in\pi_0(M)$ satisfying
\ea
\label{ep}
 \ep_{\al,\be}\cdot\ep_{\be,\al}&=(-1)^{\chi(\al,\be)+\chi(\al,\al)\chi(\be,\be)},&
 \ep_{\al,\be}\cdot\ep_{\al+\be,\ga}&=\ep_{\be,\ga}\cdot\ep_{\al,\be+\ga}.
\ea

\end{enumerate}
\end{ass}

By (e), the \emph{Euler form} $\chi(\al,\be)=\rk\th_{\al,\be}$ is symmetric and bi-additive.

\begin{rem}
Usually, the signs $\ep_{\al,\be}$ are determined geometrically and correspond to orientations; see \cite[\S8.3]{Joy}. These orientation problems are solved in the series \cite{JTU,JoUp1,JoUp2,JoUp3} using the excision technique of \cite{Upme}.
\end{rem}

\subsection{\texorpdfstring{Proof of Theorem~\ref{GradedLie}}{Proof of Theorem 1.5}}\label{sLie}

Clearly, the Lie bracket defined in \eqref{def:Lie} by
\[
[\ze,\eta]=\ep_{\al,\be}(-1)^{a\chi(\be,\be)}(\Phi_{\al,\be}/G)_*\circ(\pi_{\al,\be})_{\th_{\al,\be}}^!(\ze\t\eta)
\]
for $\ze\in H_a(M_\al/G)$ and $\eta\in H_b(M_\be/G)$ is bilinear. We prove skew symmetry. To compute $[\eta,\ze],$ note that the pullback principal $G$-bundle $(\tau^{\al,\be}_{\be,\al})^*(R_{\be,\al})$ is isomorphic to the dual bundle $\breve R_{\al,\be}.$ By \eqref{th-dual}, the pullback orientation on $\breve R_{\al,\be}$ is $\breve \th_{\al,\be}.$ Hence
\begin{align*}
 (\pi_{\be,\al})_{\th_{\be,\al}}^!\circ(\tau^{\al,\be}_{\be,\al})_*
 &=(\tau^{\al,\be}_{\be,\al})_*\circ\bigl(\breve\pi_{\al,\be}\bigr)^!_{\breve \th_{\al,\be}} &&\text{by nat.~\eqref{HomolNaturality}}\\
 &=(-1)^{\chi(\al,\be)+1}(\tau^{\al,\be}_{\be,\al})_*\circ(\pi_{\al,\be})_{\th_{\al,\be}}^!
&&\text{by duality \eqref{HomolDualAxiom}.}
\end{align*}
Combining this with $\Phi_{\al,\be}\simeq \Phi_{\be,\al}\circ\tau^{\al,\be}_{\be,\al}$ and the skew symmetry of the cross product, $\eta\t\ze=(-1)^{ab}{(\tau^{\al,\be}_{\be,\al})_*(\ze\t\eta)},$ we obtain
\[
[\eta,\ze]=\ep_{\be,\al}(-1)^{b\chi(\al,\al)+ab+\chi(\al,\be)+1}(\Phi_{\al,\be}/G)_*\circ(\pi_{\al,\be})_{\th_{\al,\be}}^!(\ze\t\eta).
\]
This differs from $[\ze,\eta]$ by the sign
\begin{multline*}
\ep_{\al,\be}(-1)^{a\chi(\be,\be)}\ep_{\be,\al}(-1)^{b\chi(\al,\al)+ab+\chi(\al,\be)+1}\\
\overset{\mathclap{\eqref{ep}}}{=}(-1)^{\chi(\al,\be)+\chi(\al,\al)\chi(\be,\be)+a\chi(\be,\be)+b\chi(\al,\al)+ab+\chi(\al,\be)+1}=(-1)^{1+|\ze|'|\eta|'},
\end{multline*}
where we recall the grading $|\ze|'=a+2-\chi(\al,\al).$ This proves skew symmetry.
\medskip

It remains only to prove the Jacobi identity
\e
(-1)^{|\ze|'|\la|'}\bigl[[\ze,\eta],\la\bigr]+(-1)^{|\eta|'|\ze|'}\bigl[[\eta,\la],\ze\bigr]+(-1)^{|\la|'|\eta|'}\bigl[[\la,\ze],\eta\bigr]=0\label{Jacobi}
\e
for $\ze\in H_a(M_\al/G),$ $\eta\in H_b(M_\be/G),$ and $\la\in H_c(M_\ga/G).$ 
Expanding the definition,
\begin{align*}
\bigl[[\ze,\eta],\la\bigr]
&=\ep_{\al,\be}(-1)^{a\chi(\be,\be)}\bigl[(\Phi_{\al,\be}/G)_*\circ(\pi_{\al,\be})_{\th_{\al,\be}}^!(\ze\t\eta),\la\bigr]\\
&=\ep_{\al,\be}(-1)^{a\chi(\be,\be)}\ep_{\al+\be,\ga}(-1)^{(a+b)\chi(\ga,\ga)}(\Phi_{\al+\be,\ga}/G)_*\\
&\quad\circ(\pi_{\al+\be,\ga})_{\th_{\al+\be,\ga}}^!\circ\bigl((\Phi_{\al,\be}/G)\t\id_{M_\ga/G}\bigr)_*\bigl((\pi_{\al,\be})_{\th_{\al,\be}}^!(\ze\t\eta)\t\la\bigr).
\end{align*}
To continue this calculation, observe that there is a pullback square
\[
 \begin{tikzcd}[column sep=15ex]
 P_{\al,\be,\ga}
 \arrow[d,"\ka_{\al,\be,\ga}"]\arrow[r,"(\Phi_{\al,\be}\t\id_{B_\ga})"] & R_{\al+\be,\ga}\arrow[d,"\pi_{\al+\be,\ga}"]\\
 Q_{\al,\be,\ga}
\arrow[r,"(\Phi_{\al,\be}/G)\t\id_{B_\ga}"] & B_{\al+\be}\t B_\ga
\end{tikzcd}
\]
and that the pullback orientation on $\ka_{\al,\be,\ga}$ is $(\tau_{\al,\ga}^{\al,\be,\ga})^*(\th_{\al,\ga})+(\tau^{\al,\be,\ga}_{\be,\ga})^*(\th_{\be,\ga})$ by \eqref{th-bilinear}.
By the dual of Corollary~\ref{Cor_EvenDegree_Many_Properties}(e), we have $(\pi_{\al,\be})_{\th_{\al,\be}}^!\t\id_{H_*(M_\ga/G)}=(\pi_{\al,\be}\t\id_{M_\ga/G})^!_{(\tau^{\al,\be,\ga}_{\al,\be})^*(\th_{\al,\be})}.$ Using these facts and naturality \eqref{HomolNaturality}, we get
\begin{equation}
\label{erster-ausdruck}
\begin{split}
&(-1)^{|\ze|'|\la|'}\bigl[[\ze,\eta],\la\bigr]=(-1)^{ac+a\chi(\be,\be)+b\chi(\ga,\ga)+c\chi(\al,\al)+\chi(\al,\al)\chi(\ga,\ga)}\ep_{\al,\be}\ep_{\al+\be,\ga}\\
 &\quad(\Phi_{\al+\be,\ga}/G)_*\circ\bigl((\Phi_{\al,\be}\t\id_{M_\ga})/G\bigr)_*\circ(\ka_{\al,\be,\ga})_{(\tau_{\al,\ga}^{\al,\be,\ga})^*(\th_{\al,\ga})+(\tau^{\al,\be,\ga}_{\be,\ga})^*(\th_{\be,\ga})}^!\\
 &\quad\circ(\pi_{\al,\be}\t\id_{M_\ga/G})_{(\tau^{\al,\be,\ga}_{\al,\be})^*(\th_{\al,\be})}^!(\ze\t\eta\t\la).
\end{split}
\end{equation}
Permuting $(a,b,c),$ $(\al,\be,\ga),$ and $(\ze,\eta,\la)$ cyclically in \eqref{erster-ausdruck} gives
\begin{equation}
\label{zweiter-ausdruck}
\begin{split}
	&(-1)^{|\eta|'|\ze|'}\bigl[[\eta,\la],\ze\bigr]=(-1)^{ba+b\chi(\ga,\ga)+c\chi(\al,\al)+a\chi(\be,\be)+\chi(\be,\be)\chi(\al,\al)}\ep_{\be,\ga}\ep_{\be+\ga,\al}\\
&\quad(\Phi_{\be+\ga,\al}/G)_*\circ\bigl((\Phi_{\be,\ga}\t\id_{M_\al})/G\bigr)_*\circ(\ka_{\be,\ga,\al})_{(\tau^{\be,\ga,\al}_{\be,\al})^*(\th_{\be,\al})+(\tau^{\be,\ga,\al}_{\ga,\al})^*(\th_{\ga,\al})}^!\\
&\quad\circ(\pi_{\be,\ga}\t\id_{M_\al/G})_{(\tau^{\be,\ga,\al}_{\be,\ga})^*(\th_{\be,\ga})}^!(\eta\t\la\t\ze)
\end{split}
\end{equation}
and
\begin{equation}
\label{dritter-ausdruck}
\begin{split}
&(-1)^{|\la|'|\eta|'}\bigl[[\la,\ze],\eta\bigr]=(-1)^{cb+c\chi(\al,\al)+a\chi(\be,\be)+b\chi(\ga,\ga)+\chi(\ga,\ga)\chi(\be,\be)}\ep_{\ga,\al}\ep_{\ga+\al,\be}\\
&\quad(\Phi_{\ga+\al,\be}/G)_*\circ\bigl((\Phi_{\ga,\al}\t\id_{M_\be})/G\bigr)_*\circ(\ka_{\ga,\al,\be})_{(\tau^{\ga,\al,\be}_{\ga,\be})^*(\th_{\ga,\be})+(\tau^{\ga,\al,\be}_{\al,\be})^*(\th_{\al,\be})}^!\\
&\quad\circ(\pi_{\ga,\al}\t\id_{M_\be/G})_{(\tau^{\ga,\al,\be}_{\ga,\al})^*(\th_{\ga,\al})}^!(\la\t\ze\t\eta).
\end{split}	
\end{equation}
We have computed all of the individual terms of the Jacobi identity \eqref{Jacobi}, which we will derive from the composition property \eqref{HomolComposition}. Consider
\begin{align*}
P_1&=Q_{\al,\be,\ga}=(M_\al\t M_\be)/G\t M_\ga/G\\
&\text{with }\ell_{P_1}\bigl(g,(\De^2 G(m_\al,m_\be),Gm_\ga)\bigr)=\bigl(\De^2G(gm_\al,m_\be),Gm_\ga\bigr),\\
P_2&=(M_\be\t M_\ga)/G\t M_\al/G\\
&\text{with }\ell_{P_2}\bigl(g,(\De^2 G(m_\be,m_\ga),Gm_\al)\bigr)=\bigl(\De^2 G(gm_\be,m_\ga),Gm_\al\bigr),
\end{align*}
and the obvious projections $\pi_1$ and $\pi_2$ to $B_{\al,\be,\ga}.$ We define the orientations on $P_1\cong(\tau^{\al,\be,\ga}_{\al,\be})^*(R_{\al,\be})$ and $P_2\cong(\tau^{\be,\ga,\al}_{\be,\ga})^*(R_{\be,\ga})$ by
\begin{align*}
 \th_1&=(\tau^{\al,\be,\ga}_{\al,\be})^*(\th_{\al,\be}),
&\th_2&=(\tau^{\be,\ga,\al}_{\be,\ga})^*(\th_{\be,\ga}).
\end{align*}
To define $\th_3$ on $P_3=P_1\ot_G P_2,$ use the pullback square
\[
 \begin{tikzcd}[column sep=8ex]
  \breve Q_{\ga,\al,\be}\rar["\mu"]\dar["\breve\pi_{\ga,\al}\t\id_{M_\be/G}" yshift=1.5mm] & P_1\ot_G P_2\dar["\pi_3"]\\
	B_{\ga,\al,\be}\rar["\tau^{\ga,\al,\be}_{\al,\be,\ga}"] & B_{\al,\be,\ga}\mathrlap{,}
 \end{tikzcd}
 \enskip
 \begin{tikzcd}[column sep=5ex,font=\footnotesize]
  \bigl((m_\ga,m_\al)\De^2G,m_\be G\bigr)\rar[mapsto] & 
  \smash{\begin{array}{@{}c@{}}
  	\bigl((m_\al,m_\be)\De^2 G,m_\ga G\bigr)\\\ot\bigl((m_\be,m_\ga)\De^2G,m_\al G\bigr)
  \end{array}}\\
  (m_\ga G,m_\al G, m_\be G)\rar[mapsto] & (m_\al G,m_\be G, m_\ga G)\mathrlap{.}
 \end{tikzcd}
 \]
 Observe here that $\mu$ is well-defined and $G$-equivariant. Define the orientation $\th_3$ on $\pi_3$ by requiring its pullback to $\breve Q_{\ga,\al,\be}$ to be $(\tau^{\ga,\al,\be}_{\ga,\al})^*(\breve \th_{\ga,\al}).$ In particular, by naturality \eqref{HomolNaturality} and duality \eqref{HomolDualAxiom} we have
\e
\label{JAC0}
 (\pi_3)^!_{\th_3}\circ(\tau_{\al,\be,\ga}^{\ga,\al,\be})_* = (-1)^{\chi(\al,\ga)+1}\mu_*\circ(\pi_{\ga,\al}\t\id_{M_\be/G})^!_{(\tau^{\ga,\al,\be}_{\ga,\al})^*(\th_{\ga,\al})}.
\e
To evaluate the terms in \eqref{HomolComposition}, we must identify the involved bundles and pullback orientations. For these bundles $P_1,P_2,P_3$ we use the notation \eqref{composition}, in particular, the projections $\ka_1\colon \pi_1^*(P_2)\to P_1,$ $\ka_2\colon \pi_2^*(P_1)\to P_2,$ and $\ka_3\colon P_1\t_B P_2\to P_3.$ There are pullback diagrams of principal $G$-bundles
\[
 \begin{tikzcd}[column sep=huge]
  Q_{\be,\ga,\al}\dar["\pi_{\be,\ga}\t\id_{M_\al/G}" yshift=1.5mm]\rar["\id_{P_2}"] & P_2\dar["\pi_2"] &   \breve P_{\be,\ga,\al}\rar["\xi_1"]\dar["\breve\ka_{\be,\ga,\al}"] & \pi_2^*(P_1)\dar["\ka_2"]\\
  B_{\be,\ga,\al}\rar["\tau^{\be,\ga,\al}_{\al,\be,\ga}"] & B_{\al,\be,\ga}\mathrlap{,}  & Q_{\be,\ga,\al}\rar["\id_{P_2}"] & P_2\mathrlap{,}
  \end{tikzcd}
\]
where $\xi_1(\De^3G(m_\be,m_\ga,m_\al))$ has fiber coordinate $(\De^2G(m_\al,m_\be),Gm_\ga)$ in $P_1$ and base coordinate $(\De^2G(m_\be,m_\ga),Gm_\al)$ in $P_2.$ Hence by naturality \eqref{HomolNaturality}
\ea
\label{JAC1}
 (\pi_2)^!_{\th_2}\circ(\tau^{\be,\ga,\al}_{\al,\be,\ga})_*&=(\pi_{\be,\ga}\t\id_{M_\al/G})^!_{(\tau^{\be,\ga,\al}_{\be,\ga})^*(\th_{\be,\ga})},\\
\label{JAC2}
(\ka_2)^!_{\ka_1^*\th_1+\ka_3^*\th_3}&=(\xi_1)_*\circ(\breve\ka_{\be,\ga,\al})^!_{(\tau^{\be,\ga,\al}_{\al,\be})^*(\th_{\al,\be})+(\tau_{\ga,\al}^{\be,\ga,\al})^*(\breve\th_{\ga,\al})}.
\ea
In \eqref{JAC2}, we use that the pullback of the orientation $\ka_1^*\th_1$ along $\xi_1$ is $\xi_1^*\ka_1^*\th_1=(\tau_{\al,\be}^{\al,\be,\ga}\circ\ka_1\circ\xi_1)^*(\th_{\al,\be})=(\tau_{\al,\be}^{\be,\ga,\al})^*(\th_{\al,\be}).$ There is a commutative diagram
\[
\begin{tikzcd}
	& \breve P_{\be,\ga,\al}\rar{\xi_1}\arrow[dl,"\tau_{\ga,\al}^{\ga,\al,\be}"' xshift=3mm,yshift=1mm]\dar & P_1\t_B P_2\dar{\ka_3}\\
	\breve R_{\ga,\al} & \breve Q_{\ga,\al,\be}\arrow[l,"{\tau_{\ga,\al}^{\ga,\al,\be}}"' xshift=1ex]\rar{\mu} & P_3,
\end{tikzcd}
\]
where we slightly abuse the notation $\tau_{\ga,\al}^{\ga,\al,\be}$ to denote the maps induced by the projection onto the first two components. By definition, $\th_3=\mu^*(\tau_{\ga,\al}^{\ga,\al,\be})^*(\breve\th_{\ga,\al}),$ so the diagram shows that the pullback of $\ka_3^*\th_3$ along $\xi_1$ is $(\tau_{\ga,\al}^{\be,\ga,\al})^*(\breve\th_{\ga,\al}).$ Using duality \eqref{HomolDualAxiom} and \eqref{th-dual}, \eqref{JAC2} becomes
\[
(\ka_2)^!_{\ka_1^*\th_1+\ka_3^*\th_3}
=(-1)^{\chi(\al,\be)+\chi(\al,\ga)+1}(\xi_1)_*\circ(\ka_{\be,\ga,\al})^!_{(\tau^{\be,\ga,\al}_{\be,\al})^*(\th_{\be,\al})+(\tau_{\ga,\al}^{\be,\ga,\al})^*(\th_{\ga,\al})}.
\]
Combining this with \eqref{JAC1} gives
\begin{multline}
\label{JAC3A}
(\ka_2)^!_{\ka_1^*\th_1+\ka_3^*\th_3}\circ(\pi_2)^!_{\th_2}\circ(\tau^{\be,\ga,\al}_{\al,\be,\ga})_*=(-1)^{\chi(\al,\be)+\chi(\al,\ga)+1}
\\
(\xi_1)_* (\ka_{\be,\ga,\al})^!_{(\tau^{\be,\ga,\al}_{\be,\al})^*(\th_{\be,\al})+(\tau_{\ga,\al}^{\be,\ga,\al})^*(\th_{\ga,\al})}\ (\pi_{\be,\ga}\t\id_{M_\al/G})^!_{(\tau^{\be,\ga,\al}_{\be,\ga})^*(\th_{\be,\ga})},
\end{multline}
which relates the 1\textsuperscript{st} term of \eqref{HomolComposition} with \eqref{zweiter-ausdruck}. For the 2\textsuperscript{nd} term, there is a pullback
\[
 \begin{tikzcd}[column sep=6ex]
 	P_{\al,\be,\ga}\rar["\xi_2"]\dar["\ka_{\al,\be,\ga}"] & \pi_1^*(P_2)\dar["\ka_1"']\\
 	Q_{\al,\be,\ga}\rar["\id_{P_1}"] & P_1\mathrlap{,}
\end{tikzcd}
\]
where $\xi_2(\De^3G(m_\al,m_\be,m_\ga))$ has fiber coordinate $\bigl(\De^2G(m_\be,m_\ga),Gm_\al\bigr)$ in $P_2$ and base coordinate $(\De^2 G(m_\al,m_\be),Gm_\ga)$ in $P_1.$ Hence naturality \eqref{HomolNaturality} implies
\begin{align}
\label{JAC4A}
 &(\ka_1)^!_{\ka_2^*\th_2+\ka_3^*\th_3}\circ(\pi_1)^!_{\th_1}\\
 &=\;(\xi_2)_*\circ(\ka_{\al,\be,\ga})^!_{\cramped{(\tau^{\al,\be,\ga}_{\be,\ga})^*(\th_{\be,\ga})+(\tau^{\al,\be,\ga}_{\ga,\al})^*(\breve\th_{\ga,\al})}}\circ(\pi_{\al,\be}\t\id_{M_\ga/G})^!_{(\tau^{\al,\be,\ga}_{\al,\be})^*(\th_{\al,\be})}\notag\\
&\overset{\mathclap{\eqref{th-dual}}}{=}\;(\xi_2)_*\circ(\ka_{\al,\be,\ga})^!_{(\tau^{\al,\be,\ga}_{\be,\ga})^*(\th_{\be,\ga})+(\tau^{\al,\be,\ga}_{\al,\ga})^*(\th_{\al,\ga})}\circ(\pi_{\al,\be}\t\id_{M_\ga/G})^!_{(\tau^{\al,\be,\ga}_{\al,\be})^*(\th_{\al,\be})}.\notag
\end{align}
This relates the 2\textsuperscript{nd} term of \eqref{HomolComposition} with \eqref{erster-ausdruck}. Finally, there is a pullback square
\[
 \begin{tikzcd}[column sep=5ex]
  \breve P_{\ga,\al,\be}\rar["\xi"]\dar["\breve\ka_{\ga,\al,\be}"] & P_1\t_B P_2\dar["\ka_3"]\\
  \breve Q_{\ga,\al,\be}\rar["\mu"] & P_1\ot_G P_2\mathrlap{,}
 \end{tikzcd}
\]
where $\xi(\De^3G(m_\ga,m_\al,m_\be))=\bigl((\De^2 G(m_\al,m_\be),Gm_\ga),(\De^2G(m_\be,m_\ga),Gm_\al)\bigr),$ recalling the convention for $\ell_{P_1\t_B P_2}$ from Page~\pageref{ConventionFPaction} to see that $\xi$ is equivariant. Since $\ka_1^*\breve\th_1+\ka_2^*\th_2$ pulls back to $(\tau^{\ga,\al,\be}_{\al,\be})^*(\breve\th_{\al,\be})+(\tau^{\ga,\al,\be}_{\be,\ga})^*(\th_{\be,\ga})$ along $\xi,$ we have
\begin{align*}
&(\ka_3)^!_{\ka_1^*\breve\th_1+\ka_2^*\th_2}\circ\mu_*=\xi_*\circ(\breve\ka_{\ga,\al,\be})^!_{(\tau_{\al,\be}^{\ga,\al,\be})^*(\breve\th_{\al,\be})+(\tau_{\be,\ga}^{\ga,\al,\be})^*(\th_{\be,\ga})} && \text{by \eqref{HomolNaturality}}\\
 &=(-1)^{\chi(\al,\be)+\chi(\be,\ga)+1}\;\xi_*\circ(\ka_{\ga,\al,\be})^!_{(\tau_{\al,\be}^{\ga,\al,\be})^*(\th_{\al,\be})+(\tau_{\ga,\be}^{\ga,\al,\be})^*(\th_{\ga,\be})} && \text{by \eqref{HomolDualAxiom} and \eqref{th-dual}}.
\end{align*}
If we precompose both sides of the expression with $(\pi_{\ga,\al}\t\id_{M_\be/G})^!_{(\tau^{\ga,\al,\be}_{\ga,\al})^*(\th_{\ga,\al})}$ and use \eqref{JAC0}, we obtain
\begin{multline}
\label{JAC4B}
	(\ka_3)^!_{\ka_1^*\breve\th_1+\ka_2^*\th_2}\circ(\pi_3)^!_{\th_3}\circ(\tau_{\al,\be,\ga}^{\ga,\al,\be})_*=
	(-1)^{\chi(\al,\ga)+\chi(\al,\be)+\chi(\be,\ga)}\\
	\xi_*\circ(\ka_{\ga,\al,\be})^!_{(\tau_{\al,\be}^{\ga,\al,\be})^*(\th_{\al,\be})+(\tau_{\ga,\be}^{\ga,\al,\be})^*(\th_{\ga,\be})}\circ(\pi_{\ga,\al}\t\id_{M_\be/G})^!_{(\tau^{\ga,\al,\be}_{\ga,\al})^*(\th_{\ga,\al})},
\end{multline}
which relates the 3\textsuperscript{rd} term of \eqref{HomolComposition} with \eqref{dritter-ausdruck}. Recall that in \eqref{HomolComposition} the (non-equivariant) homeomorphisms $\pi_2^*(P_1)\cong P_1\t_B P_2\cong \pi_1^*(P_2)$ are implicit. The maps $\xi_1,$ $\xi_2,$ and $\xi$ fit into a commutative diagram
\vskip-0.75\baselineskip
\[
\begin{tikzcd}
	\breve P_{\be,\ga,\al}\dar["\xi_1"]\arrow[rr,bend left=15,"\tau^{\be,\ga,\al}_{\al,\be,\ga}"] & \breve P_{\ga,\al,\be}\dar["\xi"]\rar["\tau^{\ga,\al,\be}_{\al,\be,\ga}"']\lar["\tau^{\ga,\al,\be}_{\be,\ga,\al}"] & P_{\al,\be,\ga}\dar["\xi_2"]\\
	\pi_2^*(P_1) & P_1\t_B P_2\rar["\cong"]\lar["\cong"'] & \pi_1^*(P_2)\mathrlap{.}
\end{tikzcd}
\]
Using \eqref{JAC3A}, \eqref{JAC4A}, and \eqref{JAC4B}, we find that the composition of \eqref{HomolComposition} with $(\xi_2)^{-1}_*$ multiplied by $(-1)^{\chi(\al,\ga)}$ is
\begin{multline*}
	(-1)^{\chi(\al,\be)+1}(\tau_{\al,\be,\ga}^{\be,\ga,\al})_*\circ(\ka_{\be,\ga,\al})^!\circ(\pi_{\be,\ga}\t\id_{M_\al/G})^!\circ(\tau_{\al,\be,\ga}^{\be,\ga,\al})_*^{-1}\\
	+(-1)^{\chi(\al,\ga)+1}(\ka_{\al,\be,\ga})^!\circ(\pi_{\al,\be}\t\id_{M_\ga/G})^!\\
	=(-1)^{\chi(\be,\ga)}(\tau_{\al,\be,\ga}^{\ga,\al,\be})_*\circ(\ka_{\ga,\al,\be})^!\circ(\pi_{\ga,\al}\t\id_{M_\be/G})^!\circ(\tau_{\al,\be,\ga}^{\ga,\al,\be})_*^{-1},
\end{multline*}
omitting the orientations from the notation from now on.
Post-compose this expression with the maps induced by $(\Phi_{\al,\be}\t\id_{M_\ga})/G\colon P_{\al,\be,\ga}\to(M_{\al+\be}\t M_\ga)/G$ and $\Phi_{\al+\be,\ga}/G\colon (M_{\al+\be}\t M_\ga)/G\to M_{\al+\be+\ga}/G$ in homology and evaluate at the homology class $\ze\t\eta\t\la.$ Using the formulas
\begin{align*}
 (\tau_{\al,\be,\ga}^{\be,\ga,\al})_*^{-1}(\ze\t\eta\t\la) &= (-1)^{a(b+c)}(\eta\t\la\t\ze),\\
(\tau_{\al,\be,\ga}^{\ga,\al,\be})_*^{-1}(\ze\t\eta\t\la) &=(-1)^{(a+b)c}(\la\t\ze\t\eta),\\
(\Phi_{\al+\be,\ga}/G)_*\bigl((\Phi_{\al,\be}\t\id_{M_\ga})/G\bigr)_*(\tau_{\al,\be,\ga}^{\be,\ga,\al})_*&=(\Phi_{\be+\ga,\al}/G)_*\bigl((\Phi_{\be,\ga}\t\id_{M_\al})/G\bigr)_*,\\
 (\Phi_{\al+\be,\ga}/G)_*\bigl((\Phi_{\al,\be}\t\id_{M_\ga})/G\bigr)_*(\tau_{\al,\be,\ga}^{\ga,\al,\be})_*&=(\Phi_{\ga+\al,\be}/G)_*\bigl((\Phi_{\ga,\al}\t\id_{M_\be})/G\bigr)_*,
\end{align*}
we obtain
\begin{multline*}
\begin{multlined}
	(-1)^{\chi(\al,\be)+1+a(b+c)}\\(\Phi_{\be+\ga,\al}/G)_*\bigl((\Phi_{\be,\ga}\t\id_{M_\al})/G\bigr)_*(\ka_{\be,\ga,\al})^!(\pi_{\be,\ga}\t\id_{M_\al/G})^!(\eta\t\la\t\ze)
\end{multlined}
\\
\begin{multlined}
	+(-1)^{\chi(\al,\ga)+1}\\(\Phi_{\al+\be,\ga}/G)_*\bigl((\Phi_{\al,\be}\t\id_{M_\ga})/G\bigr)_*(\ka_{\al,\be,\ga})^!(\pi_{\al,\be}\t\id_{M_\ga/G})^!(\ze\t\eta\t\la)
\end{multlined}
\\
\begin{multlined}
=(-1)^{\chi(\be,\ga)+(a+b)c}\\(\Phi_{\ga+\al,\be}/G)_*\bigl((\Phi_{\ga,\al}\t\id_{M_\be})/G\bigr)_*(\ka_{\ga,\al,\be})^!(\pi_{\ga,\al}\t\id_{M_\be/G})^!(\la\t\ze\t\eta).
\end{multlined}
\end{multline*}
Bring all terms to the right-hand side, multiply by $(-1)^{ac+a\chi(\be,\be)+b\chi(\ga,\ga)+c\chi(\al,\al)}$ and use \eqref{erster-ausdruck}, \eqref{zweiter-ausdruck}, and \eqref{dritter-ausdruck} to get
\begin{multline*}
(-1)^{\chi(\al,\be)+\chi(\al,\al)\chi(\be,\be)}\ep_{\be,\ga}\ep_{\be+\ga,\al}(-1)^{|\eta|'|\ze|'}[[\eta,\la],\ze]\\
+(-1)^{\chi(\al,\ga)+\chi(\al,\al)\chi(\ga,\ga)}\ep_{\al,\be}\ep_{\al+\be,\ga}(-1)^{|\ze|'|\la|'}[[\ze,\eta],\la]\\
+(-1)^{\chi(\be,\ga)+\chi(\be,\be)\chi(\ga,\ga)}\ep_{\ga,\al}\ep_{\ga+\al,\be}(-1)^{|\la|'|\eta|'}[[\la,\ze],\eta]=0.
\end{multline*}
Finally, use \eqref{ep} to check that all of the appearing signs agree, that is,
\begin{multline*}
 (-1)^{\chi(\al,\be)+\chi(\al,\al)\chi(\be,\be)}\ep_{\be,\ga}\ep_{\be+\ga,\al}=(-1)^{\chi(\al,\ga)+\chi(\al,\al)\chi(\ga,\ga)}\ep_{\al,\be}\ep_{\al+\be,\ga}\\=(-1)^{\chi(\be,\ga)+\chi(\be,\be)\chi(\ga,\ga)}\ep_{\ga,\al}\ep_{\ga+\al,\be}.
\end{multline*}
The Jacobi identity \eqref{Jacobi} follows.
\null\nobreak\hfill\ensuremath{\square}

%
%
%
%
%

\appendix
\section{Elementary technical facts}
\label{sAppdx}

The first goal of this appendix is to prove an unfamiliar identity for the binomial coefficients $\binom{n}{k}$ with integer $n$ and $k.$ These are defined by
\[
 \binom{n}{k}
=
\begin{cases}
\prod_{i=1}^k\frac{n-i+1}{i}&\text{if $k>0,$}\\
1&\text{if $k=0,$}\\
0&\text{if $k<0.$}
\end{cases}
\]
Recall the standard properties
\begin{align}
k<0&\text{ or } 0\leqslant n<k&&\iff\binom{n}{k}=0,\label{B0}\\
k\in\Z&\text{ and }0\leqslant n&&\implies\binom{n}{k}=\binom{n}{n-k},\label{B1}\\
k\in\Z&\text{ and }n\in\Z&&\implies\binom{n}{k}=(-1)^k\binom{k-n-1}{k},\label{B2}\\
0\leqslant h+k&\text{ and } 0\leqslant n&&\implies\binom{n-h}{k}\binom{n}{h}=\binom{n}{n-h-k}\binom{h+k}{k}.\label{B2A}
\end{align}
The following result may be viewed as an extension of \eqref{B1} to integer $n.$

\begin{prop}
For all integers\/ $n$ and\/ $k$ we have
\e
\label{B3}
(-1)^k\binom{k-n-1}{-n-1}
=
\begin{cases}
\binom{n}{k}-\binom{n}{n-k} & \text{if\/ $n\geqslant 0,$}\\
-\binom{n}{n-k} & \text{if\/ $n<0$ and\/ $k\leqslant n,$}\\
\binom{n}{k} & \text{if\/ $n<0$ and\/ $k>n.$}
\end{cases}
\e
Hence by \eqref{B0} in every case we have
\e
\label{B4}
 (-1)^k\binom{k-n-1}{-n-1}=\binom{n}{k}-\binom{n}{n-k}.
\e
\end{prop}

\begin{proof}
Suppose $n\geqslant 0.$ Then the left hand side of \eqref{B3} vanishes by \eqref{B0} and the right hand side vanishes by \eqref{B1}.
Suppose $n<0$ and $k\leqslant n.$ Then
\[
\binom{n}{n-k}\overset{\mathclap{\eqref{B2}}}{\;=\;}(-1)^{n-k}\binom{-k-1}{n-k}\overset{\mathclap{\eqref{B1}}}{\;=\;}(-1)^{n-k}\binom{-k-1}{-n-1}\overset{\mathclap{\eqref{B2}}}{\;=\;}(-1)^{k+1}\binom{k-n-1}{-n-1}.
\]
Finally, suppose $n<0$ and $k>n.$ Then
\[
	\binom{n}{k}\overset{\mathclap{\eqref{B2}}}{\;=\;}(-1)^k\binom{k-n-1}{k}\overset{\mathclap{\eqref{B1}}}{\;=\;}(-1)^k\binom{k-n-1}{-n-1}.\qedhere
\]
\end{proof}

The second goal of the appendix is to explain how to define a morphism of possibly infinite-dimensional vector spaces over a field $\K$ in terms of its \emph{algebraic} dual. The following is the key result. Let $\ev_v=\an{-,v}$ denote the evaluation at $v.$

\begin{lem}
Let\/ $V$ be a vector space over\/ $\K.$ The image of the embedding\/ $\ev\colon V\to V^{**},$ $v\mapsto\ev_v,$ into the double dual of\/ $V$ is the set of all morphisms\/ $\ep\colon V^*\to\K$ for which there exists a finite-dimensional subspace\/ $V_\ep\subset V$ such that
\e
\label{ImageDD}
 \al\in V^*, \al|_{V_\ep}=0\implies\ep(\al)=0.
\e
\end{lem}

\begin{proof}
For $\ep=\ev_v$ the property \eqref{ImageDD} holds for the span $V_\ep$ of $v.$ Conversely, given $V_\ep$ satisfying \eqref{ImageDD}, pick a basis $(v_i)_{i=1,\ldots,n}$ of $V_\ep$ and let $(v_i^*)_{i=1,\ldots,n}$ be the dual basis of $V_\ep^*.$ Note that for all $\al\in V^*$ its restriction to $V_\ep$ is $\sum_{i=1}^n\al(v_i)v_i^*.$ Extend $v_i^*$ arbitrarily to $V.$ Define $v=\sum_{i=1}^n\ep(v_i^*)v_i.$ We have $\ep=\ev_v,$ since for all $\al\in V^*$
\[
 (\ep-\ev_v)(\al)=\ep(\al)-\sum_{i=1}^n\ep(v_i^*)\al(v_i)=\ep\Bigl(\al-\sum_{i=1}^n\al(v_i)v_i^*\Bigr)\overset{\eqref{ImageDD}}{=}0.\qedhere
\]
\end{proof}

\begin{lem}
\label{lemExPredual}
Let\/ $V$ and\/ $W$ be vector spaces over\/ $\K$ and let\/ $\phi\colon W^*\to V^*$ be a morphism. There exists a \textup(unique\textup) morphism\/ $f\colon V\to W$ with\/ $\phi=f^*$ if and only if for every\/ $v\in V$ there exists a finite-dimensional subspace\/ $W_v\subset W$ such that
\e
\label{ExPredual}
 \be\in W^*,\be|_{W_v}=0\implies\an{\phi(\be),v}=0.
\e
\end{lem}

\begin{proof}
If $\phi=f^*$ and $v\in V,$ then \eqref{ExPredual} holds for the span $W_v$ of $f(v).$ Conversely, given \eqref{ExPredual} we can construct $f$ in the commutative square
\[
\begin{tikzcd}
V^{**}\ar[r,"\phi^*"] & W^{**}\\
V\ar[u,"\ev"]\arrow[r,dotted,"f"] & W\ar[u,"\ev"]
\end{tikzcd}
\]
by showing for each $v\in V$ that $\ep=\phi^*(\ev_v)$ is in the image of $W.$ Define $W_\ep=W_v$ and check \eqref{ImageDD}: if $\be\in W^*$ satisfies $\be|_{W_\ep}=0,$ then $\ep(\be)=\an{\phi(\be),v}\overset{\eqref{ExPredual}}{=}0.$
\end{proof}

\medskip

\noindent Department of Mathematics, University of Aberdeen, Fraser Noble Building, Elphinstone Rd, Aberdeen, AB24 3UE, U.K.\\ E-mail: {\tt markus.upmeier@abdn.ac.uk.}

\end{document}